\documentclass[10pt]{article}

\usepackage[utf8]{inputenc}
\DeclareUnicodeCharacter{00A0}{~} 
\usepackage[ruled,vlined,linesnumbered]{algorithm2e}
\usepackage{mathtools}
\usepackage{amssymb}
\usepackage{amsthm}
\usepackage{thmtools}
\usepackage{thm-restate}
\usepackage{comment}
\usepackage{xcolor} 
\usepackage[unicode]{hyperref}
\usepackage[T1]{fontenc}
\usepackage{geometry}
\usepackage{multirow}
\usepackage[colorinlistoftodos]{todonotes}
\usepackage{lmodern}
\usepackage{anyfontsize}
\usepackage{stmaryrd}
\usepackage{natbib}
\usepackage{cleveref}
\usepackage{graphicx}
\usepackage{subcaption}
\usepackage{dsfont}
\usepackage{float}
\usepackage{amsbsy}
\usepackage{sidecap}
\usepackage{cancel}
\setcitestyle{numbers,open={[},close={]}}

\definecolor{red}{rgb}{0.7,0.15,0.15}
\definecolor{green}{rgb}{0,0.5,0}
\definecolor{blue}{rgb}{0,0,0.7}
\hypersetup{colorlinks, linkcolor={red},citecolor={green}, urlcolor={blue}}
			
\makeatletter \@addtoreset{equation}{section}

\newtheorem{theorem}{Theorem}[section]

\newtheorem{corollary}[theorem]{Corollary}

\newtheorem{lemma}[theorem]{Lemma}
\newtheorem{proposition}[theorem]{Proposition}

\newtheorem{definition}[theorem]{Definition}
\newtheorem{remark}[theorem]{Remark}

\def\no{\noindent}
\def\beq{\begin{eqnarray}}
\def\eeq{\end{eqnarray}}
\def\be*{\begin{eqnarray*}}
\def\ee*{\end{eqnarray*}}

\def \R{\mathbb{R}}

\def\0{\mathbf{0}}

\def\normeL2#1{\left\|{#1}\right\|_{L^2}}

\def \1{\mathds{1}}

\DeclareUnicodeCharacter{014D}{\=o}
 \title{Ancestral lineages under horizontal trait transfer: spinal decomposition and time reversal beyond stationarity}

\author{Mateo Deangeli Bravo\thanks{Ecole Polytechnique Paris, Centre de Math\'ematiques Appliqu\'ees,
        mateo.deangeli@polytechnique.edu.}
    }
             \date{\today}

\begin{document}

\maketitle
 
\begin{abstract}

We investigate a stochastic individual-based population model with trait-dependent reproduction, death, competition, variation, and deleterious transfer. Nonlinearities appear in the competition term and the transfer form, which moreover is asymmetric.
In the large-population limit, the process converges to a deterministic measure-valued equation. 

The latter helps us to uncover the past history of living individuals. We replace in the interactions the original stochastic process with the deterministic large population approximation obtaining a non-homogeneous measure-valued branching process, for which we use a spinal decomposition, generalizing a method that existed only for the stationary case. Reversing time, we identify the law of typical ancestral lineages. This approach encompasses a variety of settings and requires delicate functional analysis. In contrast to earlier approaches in the literature, which relied either on stochastic calculus or on favourable jump structures under stationarity assumptions, our method provides a more direct and more general point of view.
\end{abstract}

\vspace{3mm}
\no{\bf Keywords.} Stochastic individual-based model, large population, limit theorem, gene transfer, harmful contagion, many-to-one formula, non homogeneous Markov processes. 

\vspace{3mm}
\no{\bf MSC2010.60J25, 60J80, 92D25} 
\section{Introduction}

In this work, we focus on modeling systems that present harmful contagion dynamics. More specifically, we consider individuals, each characterized by a positive trait measuring the intensity of a burden or vulnerability. We assume that individual mortality increases with the trait, and that when pairs of individuals interact, the one with the lower trait adopts the trait of the other. In other words, we examine the interplay between asymmetric alignment toward deleterious traits and the absence of any reward for adopting such traits.

The phenomenon that originally motivates this work is Horizontal Gene Transfer (HGT)—that is, the exchange of genetic material between non-related individuals—which is recognized as a major driver of bacterial evolution \cite{keeling2008horizontal,moran2010lateral,ochman2000lateral,tettelin2008comparative}. Among the different HGT mechanisms, conjugation—the direct transfer of DNA through cell-to-cell contact—plays a central role. It is notably mediated by plasmids, which can carry and spread metabolic or resistance genes across lineages \cite{oliver2008population}. In a neutral environment devoid of selective pressures (such as antibiotics), carrying and expressing plasmid-encoded machinery represents a pure metabolic burden, directly increasing host cell mortality \cite{baltrus2013exploring,jaenike2012population,romanchuk2014bigger,stewart1977population}. During a conjugative event, physical DNA transfer from a donor cell to a recipient cell causes the recipient to acquire the genetic element, increasing its overall burden.
While plasmid transfer offers a discrete biological illustration of asymmetric horizontal alignment, it inspires a broader conceptual framework. In reality, bacterial communities present an immense diversity of plasmids with varying transfer rates, copy numbers, and metabolic costs. Treating the overall plasmid burden as a continuous variable serves as a convenient modeling approximation: it avoids the combinatorial complexity of tracking distinct plasmid libraries within each cell, allowing us to focus specifically on the interplay between horizontal transfer dynamics and the cumulative deleterious cost. In this work, we thus generalize this principle to systems characterized by a continuous trait where local interactions systematically drive individuals toward higher levels of a purely costly burden.

A prime biological example of a continuous trait is the expression of efflux pumps in bacteria (see \cite{du2018multidrug,huang2022bacterial,zhao2020quorum}). Here, the trait represents the density or expression level of these membrane pumps (which normally serve to eject antibiotics and toxins from the cell). Through cell-to-cell signaling such as quorum sensing, bacteria with higher pump levels secrete more signals, triggering a positive feedback loop that aligns neighboring cells toward the higher expression threshold. In a neutral environment without antibiotics, this overexpression provides no selective benefit. Instead, it imposes a heavy metabolic cost—wasting cellular energy and weakening the cell membrane—which directly increases the cell's mortality rate.

Far from being restricted to microbiology, this asymmetric alignment also applies to human behaviour. Consider a population where individuals have a continuous trait—such as daily smoking intensity or exposure to risky behaviours—that increases mortality. In this context, "transfer" happens through social imitation or peer pressure, driving individuals to copy those with higher risk levels. Whether in extreme physical challenges or binge drinking during social events, interactions push people toward higher, more dangerous levels without offering any long-term benefit. This unilateral transfer thus serves as a simple baseline model for social contagion dynamics.

The questions that interest us are the following: what are the origins of the individuals alive at present? Do they predominantly descend from ancestors carrying low levels of the burden, or from lineages characterized by high exposure to the deleterious trait? Examining these ancestral lineages provides a retrospective view of the evolutionary process and helps uncover the structural determinants that allow high-cost traits to persist or be purged in structured populations.
\newline

We consider a population modelled by a measure-valued stochastic process inspired by \cite{A2,champagnat2021stochastic,fournier_meleard}. Each individual is characterized by a trait belonging to $\mathbb{R}_+$, quantifying the level of burden. The population is described by an individual-based process $(Z^K_t)_{t\in \R_+}$ corresponding to the trait distribution of the living individuals at time $t>0$,
\begin{equation}\label{def:ZK}
Z^K_t(dx)=\frac{1}{K}\sum_{i=1}^{N^K_t}\delta_{X^i_t}(dx),
\end{equation}
where $X^i_t$ is the trait of the $i$-th individual at time $t$, $N^K_t$ is the population size at time $t$ and $K$ is a scaling parameter. Indeed, the population is parameterized by an integer $K$ (which can be thought of as a carrying capacity or the order of the initial population size) that will be taken to infinity to obtain the large population limit. The population process is a \textit{càdlàg} process with values in $M_F(\mathbb{R}_+)$, the set of finite measures on $\mathbb{R}_+$, equipped with the weak convergence topology. We will denote by $\mathbb{D}([0,T],E)$ the set of \textit{càdlàg} functions from $[0,T]$ to $E$.

The population evolves through birth, death, competition, variation, and trait transfer. We work in a neutral environment: the reproduction-optimal trait is constant and, by convention, set to $0$. 

We assume transfer is unilateral: an individual with trait $y$ adopts the trait $x$ from another individual in a population $\nu\in M_F(\mathbb{R}_+)$ with the intensity $h(x,y,\nu)$ as in \cite{A2}:

\begin{equation}\label{eq:definition transfert}
h(x,y,\nu)=\frac{\tau}{\beta+\mu\langle\nu,1\rangle}\1_{{x>y}},
\end{equation}
where $\langle\nu,1\rangle=\int\nu(dx)$ is the number of individuals in the $\nu$ population. The quantity $\tau$ corresponds to the average frequency of transfer events, $\beta$ to the average time of encounters between two individuals, and $\mu$ to the average time taken for the transfer to be ended. Depending on the values of these three parameters, the transfer regime changes. It is described as Frequency-Dependent when $\beta=0$ and as Density-Dependent when $\mu=0$. When $\mu$ and $\beta$ are both greater than zero, this is referred to as a mixed regime, called the Beddington-DeAngelis regime in \cite{A1,A2} (see also \cite{geritz2012mechanistic}).

 An individual with trait $x \in \mathbb{R}_+$ gives birth to an offspring of the same trait at rate $b(x)$ and dies at rate $d(x) + C N^K_t/K$ where $C>0$. In this death rate, $d(x)$ represents the intrinsic mortality rate of the individual, while the additional term $C N^K_t/K$ models competition effect, assumed here to be of logistic type: it scales linearly with the total population size and does not explicitly depend on the distribution of traits. Additionally to transfer, an individual of trait $x$ may undergo trait changes through  variation mechanism: increase or decrease of the burden. Spontaneous variations occur at a rate $\gamma(x)$ according to a kernel $m(x,z)dz$, which describes the probability distribution of the new trait $z$.
\newline

The simulation of such dynamic (see Figure \ref{fig:IBM exemple}) exhibits interesting behaviours; for example cyclic phenomena already pointed out numerically in \cite{A2,A3} and partially explained in \cite{champagnat2021stochastic}. However, considering only the individuals alive at the present time whilst tracing the lines of their ancestors backward (see Figure \ref{fig:IBM exemple AL}) reveals a stark contrast. While the forward population exploration occupies a large trait space, the surviving lineages form a thin specific set of paths. 
Understanding how to mathematically isolate, scale, and reverse these specific "victorious" lineages constitutes the core motivation of this work. Our investigation begins by establishing the large-population scaling limit ($K \to \infty$) of the individual-based stochastic process \cite{fournier_meleard,champagnat2006unifying}. Using the limit to freeze nonlinearities, as already done in \cite{meleardB1, meleardB2}, provides a time-inhomogeneous branching process that allows to perform spinal decomposition. The ancestral lineages are obtained as the time reversal of the spine.

\begin{figure}[h!]
\centering\begin{minipage}[c]{0.48\textwidth}\centering\includegraphics[width=\textwidth]{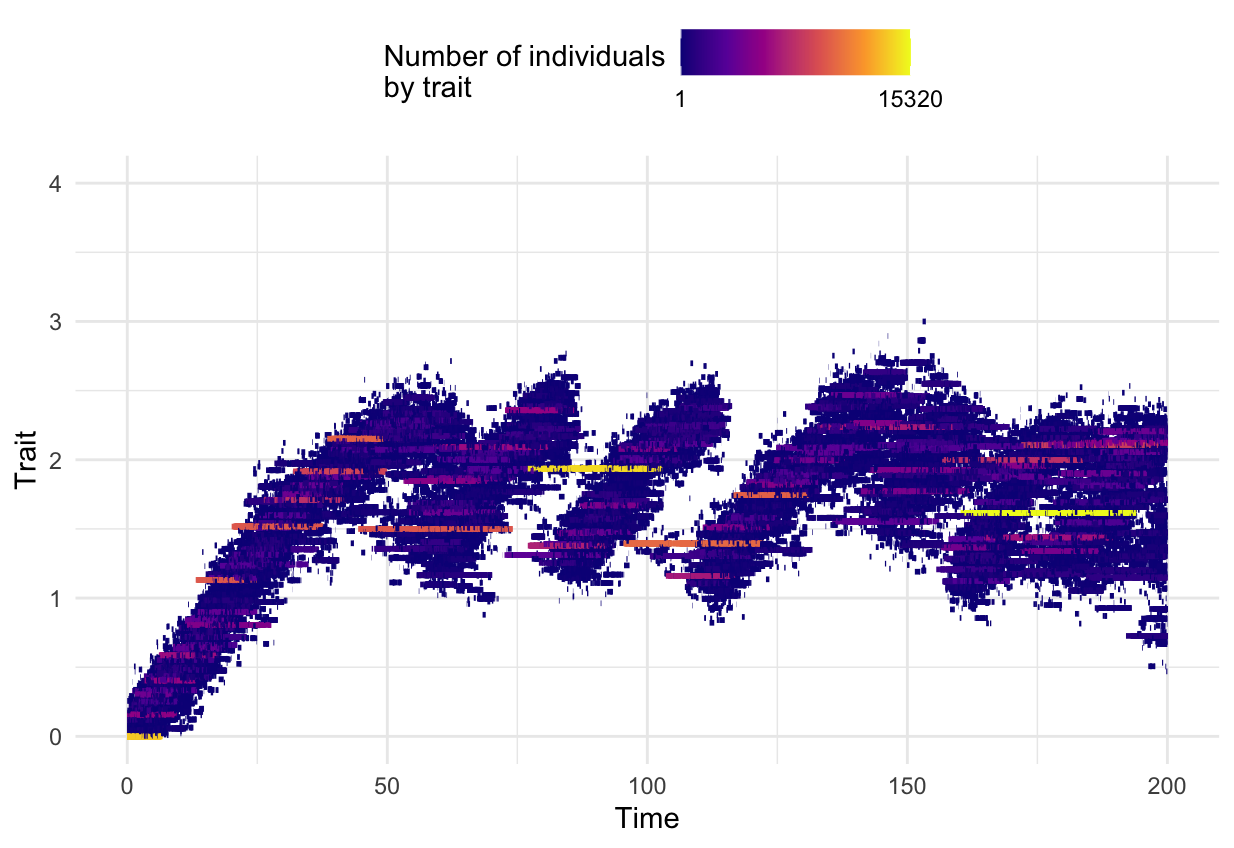}\end{minipage}\hfill\begin{minipage}[c]{0.45\textwidth}\caption{Simulation of the individual-based process - $K=200$, $b(x)=5$, $d(x)=5+(x^2-9)/2$, $\gamma(x)=0.15$, $C=1/2$, $\sigma=0.1$, $m(x,y)=\frac{1}{\sqrt{2\pi\sigma^2}}(e^{-\frac{(x-y)^2}{2\sigma^2}}+e^{-\frac{(x+y)^2}{2\sigma^2}})$, $\tau=\mu=1$,$\beta=0$. - Obtained with the IBMPopSim Package \cite{giorgiibmpopsim}.}\label{fig:IBM exemple}\end{minipage}\end{figure} 
\begin{figure}[h!]
    \centering
    \begin{minipage}[c]{0.48\textwidth}
        \centering
        \includegraphics[width=\textwidth]{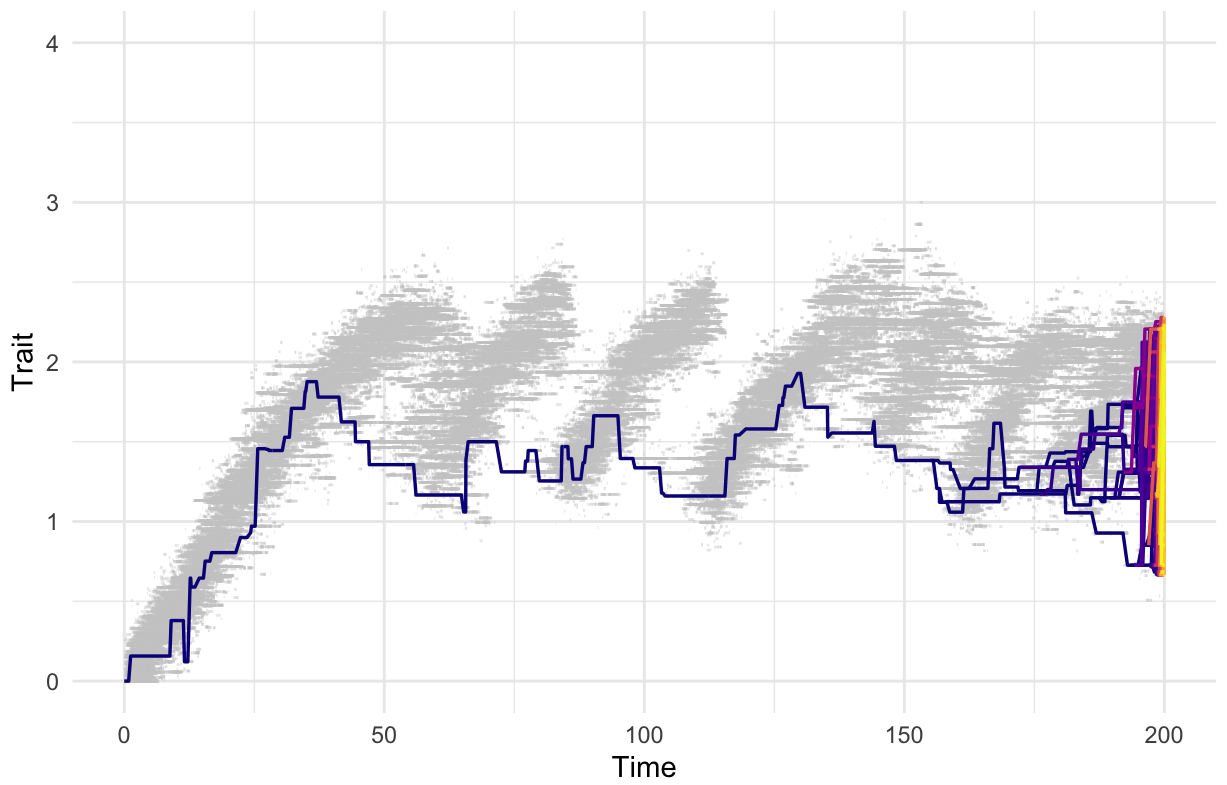}
    \end{minipage}
    \hfill
    \begin{minipage}[c]{0.45\textwidth}
        \caption{Ancestral lineages of the survivors (coloured lines), superimposed on the forward population dynamics displayed in Figure \ref{fig:IBM exemple}.}
        \label{fig:IBM exemple AL}
    \end{minipage}
\end{figure}
An initial class of models focused on population dynamics through a stochastic process approach. \cite{A1} proposed a model based on a birth-death process structured into two sub-populations. The study of its scaling limit leads to a system of Lotka-Volterra-type equations incorporating HGT, a model particularly suited to studying the invasion or persistence of a single plasmid type. These dynamics were subsequently extended by \cite{A2} via an individual-based model. The authors establish the large-population limit and formulate several conjectures concerning the trait substitution sequence (TSS). Continuing along this line, \cite{A3} explored these conjectures in greater detail numerically, highlighting in particular oscillatory behaviour in the population process and its Hamilton-Jacobi limit.
From an analytical perspective, the behaviour of these macroscopic dynamics has also been the subject of in-depth study. For example, \cite{mirrahimi} proved the existence of a stationary solution for an integro-differential partial differential equation modelling horizontal transfer, in a specific context where variations are diffusive and transfer is frequency-dependent, and establish the associated Hamilton-Jacobi equations. They explored these questions in more details in \cite{garriz2026hamilton}. It can be noted that another strand of the literature has focused on the spatio-temporal structure of diversity, drawing on population genetics. In this vein, \cite{baumdicker2014infinitely} has notably modeled the impact of horizontal gene transfer (HGT) from the perspective of coalescent theory and the infinite-allele model.

We extend the framework of \cite{A2} by relaxing several compactness and boundedness assumptions, specifically considering a non-compact trait space $\mathbb{R}_+$, a discontinuous transfer term, an unbounded death rate, and heavy-tailed variation kernels. Under these conditions, we establish the convergence of the microscopic population model to the corresponding integro-differential equation and prove the existence of a strong solution. Our proof introduces a novel argument to handle the technical challenges posed by the unbounded death rate and heavy-tailed variations. In this framework, the scaling limit is given by the density $\xi_t(x)$ satisfying the following non-linear integro-differential equation:
\begin{equation}\label{densité limite}
\begin{aligned}
\partial_t\xi_t(x) &=
    \Big(r(x) - C\!\int_0^\infty  \xi_t(y)\,dy\Big)\xi_t(x)+ \int_0^\infty[ \gamma(y)\xi_t(y)m(y,x)-\gamma(x)\xi_t(x)m(x,y)]dy \\
    & + \tau\xi_t(x)\left(\beta + \mu \int_0^\infty \xi_t(z)\,dz\right)^{-1}\int_0^\infty \text{sgn}(x-y)\xi_t(y)\,dy.
\end{aligned}
\end{equation}

We prove the existence of a strong solution and based on this macroscopic density, we then shift our focus to individual lines to capture the true ancestral trajectories of currently living individuals (previously visualized in Figure \ref{fig:IBM exemple AL}). Due to non-linear interactions induced by competition and horizontal transfer, the original stochastic population process lacks the branching property, preventing a direct genealogical analysis through spinal decomposition. To overcome this, inspired by \cite{meleardB1,meleardB2}, we introduce an auxiliary branching process where these non-linearities are systematically replaced by the deterministic density $\xi_t$, thereby restoring a tractable branching structure. For branching processes, there has been an abundant litterature (see \cite{andre2025moments,hardy2006new,hardy2009spine,marguet,marguet2019uniform}).

The detailed study of the dynamics of ancestral lineages — setting aside issues of transfer — has seen  theoretical developments, notably in \cite{meleardB1}. The authors examine the genealogy of a population whose individual variations follow a drifted diffusion dynamic. A central point of their analysis lies in establishing the existence of a stationary solution for the scaling-limit equation, aided by a favourable framework provided by Itô calculus and the Gaussian structure of the processes. In a methodological continuation, \cite{meleardB2} replaces diffusion with a non-local jump kernel to model variations, exploiting the symmetry of the generators via semi-group theory. We adapt these lineage dynamics tools to address the specific case of horizontal transfer. We introduce a novel approach based on a direct treatment of Feynman-Kac semigroups, which offers a conceptually simpler and more general framework than existing methods. This approach handles widespread time-inhomogeneity—notably due to the absence of a stationary density—and manages the additional terms arising from strong dynamical asymmetries, effectively encompassing previous works.

To reconstruct the genealogy retrospectively, the spinal process obtained by a Feynman-Kac formula must be reversed in time. We achieve this by exploiting Nagasawa’s Theorem \cite{nagasawa64, nagasawa} for Markov processes. 

This allows us to explicitly derive the infinitesimal generator $\widetilde{L}^\mathcal{R}$ of the time-reversed spine:
\begin{equation}\label{generator_backward}
\begin{aligned}
\widetilde{L}^\mathcal{R} \psi(s,x) =
-\partial_1 \psi(s,x)
&+ \int_0^\infty  (\psi(s,z) - \psi(s,x)) \frac{\xi(s,z)}{\xi(s,x)} \gamma(z) m(z,x)  dz \\
&+ \tau\int_0^\infty  (\psi(s,y) - \psi(s,x)) \frac{\xi(s,y)}{\beta + \mu \langle \xi_s, 1 \rangle} \mathds{1}_{{y < x}} dy.
\end{aligned}
\end{equation}

Beyond its theoretical value, this time-reversed characterization provides a powerful computational tool. Instead of executing heavy individual-based simulations across a full population, our framework leads to a simulation solution that solves the deterministic macro-equation \eqref{densité limite} only once, and then allows us to simulate ancestral lineages directly as simple, independent, real-valued backward processes governed by $\widetilde{L}^\mathcal{R}$.

As illustrated by our numerical simulations in Section \ref{ssec:simulations}, these sample trajectories highlight a crucial evolutionary phenomenon. Under a Gaussian variation kernel, once the spine reaches an adapted trait close to $0$, it remains there, implying that contemporary individuals likely originate from ancestors that evolved under regulated or fluctuating transfer regimes. But  under heavy-tailed Cauchy variations, ancestral lineages exhibit large jumps toward highly maladapted traits, illustrating how heavy-tailed variation acts as an evolutionary "rescue" mechanism for lineages that would otherwise face certain extinction under strict selection.
\newline

The article is organized as follows. In Section \ref{sec:pop_process}, we study the population process. We first establish in Section \ref{ssec:conv_deterministic} the convergence of the population process towards a deterministic measure, and then examine in Section \ref{ssec:integro_diff} the properties of the associated integro-differential equation. Section \ref{ssec:numerical} presents numerical resolutions and conjectures on the limiting measure $\xi$.
The Feynman-Kac approach is developed in Section \ref{sec:fk approach} : in Section \ref{ssec:coupling_branching}, we describe the coupling with an auxiliary branching population process, and Section \ref{ssec:historical_def_conv} defines the historical processes and establishes their convergence. Section \ref{ssec:fk_formula} presents the Feynman-Kac formula, while Section \ref{ssec:spine_law} studies the law of the spinal process. Section \ref{sec:dualities} focuses on the dualities of semigroups. Section \ref{ssec:dual_spaces} introduces duality on the appropriate spaces, and Section \ref{ssec:dual_extension} extends these dualities to larger functional spaces.
Next, Section \ref{sec:time_reversal} is dedicated to the time reversal and the return to the initial population process. Section \ref{ssec:time reversal of the spinal process} establish the law of the time-reversed of the spinal process. Section \ref{ssec:return_original} relates these results back to the original population process and Section \ref{ssec:simulations} presents some simulations and aspects of the ancestral lineage process, and develops the biological implications of our study.

\section{Convergence of the population process}\label{sec:pop_process}

We model a population using the measure-valued process \eqref{def:ZK}. Births, deaths, competition, variations, and horizontal transfer are all taken into account. 
We write $\tau(x,y)=\tau\mathds{1}_{\{x>y\}}$ for the numerator of \eqref{eq:definition transfert}. The time interval considered throughout this work is $[0,T]$ and if not specified, any $t$ will be an element of this interval. Because we consider càdlàg processes, we will make use of the classical Skorohod topology, namely the $J_1$ topology.

\subsection*{Assumptions (H)}\label{Assumptions H}

Unless otherwise stated, all functions considered are assumed to be measurable. 
\begin{itemize}
    \item [(a)] We assume that the rates $b,d$ and $\gamma$ are continuous and non-negative. The birth (respectively variation) rate is bounded by a constant $\overline{b}$ (respectively $\overline{\gamma}$) and $\lim_{x\rightarrow\infty}d(x)=\infty$.
    \item [(b)] We assume that the variation kernel satisfies
\begin{equation}\label{hyp:mutation kernel}
    \int_0^\infty  m(x,y) \, dy = 1,\qquad \sup_{y\geq0}\int_0^\infty  m(x,y)dx<\infty
\end{equation}
and that $m$ is continuous on $\mathbb{R}_+^2$. 
    \item [(c)] We assume the existence of a $\eta\in(0,1)$ and a constant $D\geq0$ such that 
    \begin{equation}\label{def:eta parameter}
        \int_0^\infty (d(z)^\eta-d(x)^\eta)m(x,z)dz\leq D+d(x)^\eta.
    \end{equation}
\end{itemize}
Assumption (b) is immediate if we have a symmetric kernel (i.e $m(x,y)=m(y,x)$) since $\int_0^\infty m(x,y)dy=\int_0^\infty m(x,y)dx=1$. An example of such a kernel is when 
\begin{equation}\label{eq:reflected kernel}
    m(x,y)=\hat m(x+y)+\hat m(x-y)
\end{equation}

where $\hat m$ is the density of a symmetric probability distribution on $\mathbb{R}$. In our simulations, $\hat{m}$ will be a centered Gaussian or Cauchy distribution. The assumption that $m(x,y)dy$ is a probability kernel simplifies the variation term in \eqref{densité limite}:
\begin{equation*}
    \int_0^\infty \gamma(x)\xi_t(x)m(x,y)dy=\gamma(x)\xi_t(x).
\end{equation*}
We shall make use of this simplification whenever appropriate.
Assumption (c) 
is a technical one that will allow us to handle the situation where the death rate is not integrable with respect to the variation kernel. When $\hat{m}$ is a Cauchy Distribution, any polynomial death rate is authorized provided that the parameter $\eta$ is sufficiently small.
Assumption (a) is consistent with related work \cite{meleardB1,cloezgabriel} where the death is also unbounded and reflects the idea within our modeling framework that moving away from the optimal trait becomes increasingly deleterious. Nonetheless, under the alternative of the death rate being bounded, Theorem \ref{theoreme cv} remains valid provided that there exists a positive, increasing function $V$ that tends to infinity at infinity and satisfies inequality \eqref{def:eta parameter}, and that the parameter $\mu$ is strictly positive. Indeed, in a density-dependent scenario (i.e., where $\mu = 0$), and assuming the mortality rate is bounded, the penalty for deviating from the optimum reaches a threshold. Furthermore, since the transfer rate is no longer uniformly bounded by the term $\tau/\mu$, it becomes impossible to guarantee that individuals will not disperse.

We describe the dynamics of $(Z^K_t)_{t\geq 0}$ \eqref{def:ZK}, corresponding to the above rates, by stochastic differential equations driven by Poisson point measures and detailed in Appendix \ref{appendix : EDS}.
We denote by $L^K$ the infinitesimal generator of the process, defined for $\nu\in\mathcal{M}_F(\mathbb{R}_+)$ and $\varphi:\mathcal{M}_F(\mathbb{R}_+)\mapsto\mathbb{R}$ measurable bounded by
\begin{equation*}
\begin{aligned}
L^K \varphi(\nu) &= K\int_0^\infty  b(x) \Big( \varphi(\nu + \tfrac{\delta_x}{K}) - \varphi(\nu) \Big) \, \nu(dx) +K \int_0^\infty  (d(x) + C\langle\nu,1\rangle) \Big( \varphi(\nu - \tfrac{\delta_x}{K}) - \varphi(\nu) \Big) \, \nu(dx)\\
&\quad + K\int_0^\infty  \gamma(x) \Big( \int_0^\infty  \big( \varphi(\nu + \tfrac{\delta_z}{K}-\tfrac{\delta_x}{K}) - \varphi(\nu) \big) m(x,z) dz \Big) \, \nu(dx) \\
&\quad + K\int_0^\infty  \int_0^\infty  h(x,y,\nu) \Big( \varphi(\nu + \tfrac{\delta_x}{K} - \tfrac{\delta_y}{K}) - \varphi(\nu) \Big) \, \nu(dy) \nu(dx).
\end{aligned}
\end{equation*}
 We state that under assumptions $\textbf{(H)}$, the population process does not explode in finite time and that there is a control of the death rate. The proof of the following Lemma \ref{lemma_moment_control} is postponed in Appendix \ref{proof:lemma_moment_control}.
\begin{lemma}\label{lemma_moment_control}
Assume (H) and that for some $p \geq 1$,
\begin{equation}\label{eq: estimate size time 0}
    \sup_K \mathbb{E}\big[\langle Z^K_0, 1 \rangle^p\big]< \infty.
\end{equation}

\textit{(i)} Then, for any finite $T>0$,
\begin{equation*}
    \sup_K \mathbb{E}\Big[\sup_{t \le T} \langle Z^K_t, 1 \rangle^p\Big] < \infty.
\end{equation*}
Suppose that $p\geq2$, we also have the estimates 
\begin{equation}\label{eq:estimate death}
    \sup_{K}\mathbb{E}\left[\int_0^T \langle Z^K_s,d\rangle ds\right]+ \sup_{K}\mathbb{E}\left[\left(\int_0^T \langle Z^K_s,d\rangle ds\right)^2\right]<\infty.
\end{equation}
\textit{(ii)} If, moreover, for $\eta$ defined in \eqref{def:eta parameter},
\begin{equation}\label{eq: estimate death time 0}
    \sup_K\mathbb{E}[\langle Z^K_0, d^\eta\rangle]< \infty
\end{equation}
then
\begin{equation}\label{eq : estimate lyapunov}
    \sup_K\mathbb{E}[\sup_{t\leq T}\langle Z^K_t,d^\eta\rangle]+\sup_K\mathbb{E}\left[\int_0^T\langle Z^K_s,d^{1+\eta}\rangle ds\right]<\infty.
\end{equation}
\end{lemma}
The estimates \eqref{eq:estimate death} and \eqref{eq : estimate lyapunov} are based on Assumption \textbf{(H.c)} and, to the best of our knowledge, this is the first time they are used to prove the convergence to the scaling limit.

Under the assumption \textbf{(H)}, \eqref{eq: estimate size time 0} and \eqref{eq: estimate death time 0}, for any measurable bounded function $\varphi$, it is well known \cite{bansaye2015stochastic} that the process $(Z^K)$ can be represented as a semimartingale:
\begin{equation}\label{eq:semimartingale}
\begin{aligned}
\langle Z^K_t, \varphi \rangle &= \langle Z^K_0, \varphi \rangle + M^{K,\varphi}_t + \int_0^t \int_0^\infty  \Big\{ b(x) - d(x) - C\langle Z^K_s,1\rangle \Big\}\varphi(x)  Z^K_s(dx) \, ds \\
&\quad + \int_0^t \int_0^\infty  \Big\{ \gamma(x) \int_0^\infty  (\varphi(z)-\varphi(x)) m(x,z) dz \Big\} Z^K_s(dx) \, ds \\
&\quad + \int_0^t \int_0^\infty  \Big\{ \int_0^\infty  h(x,y,Z^K_s) (\varphi(x) - \varphi(y)) Z^K_s(dy) \Big\} Z^K_s(dx) \, ds,
\end{aligned}
\end{equation}
where $M^{K,\varphi}$ is a square-integrable martingale with predictable quadratic variation
\begin{equation*}\label{predictable_quadratic_variation}
\begin{aligned}
\langle M^{K,\varphi} \rangle_t &= \frac{1}{K} \int_0^t \int_0^\infty  \varphi^2(x) \Big( b(x) + d(x) + C\langle Z^K_s,1\rangle \Big) Z^K_s(dx) \, ds \\
&\quad + \frac{1}{K} \int_0^t \int_0^\infty  \Big\{ \gamma(x) \int_0^\infty  (\varphi(z)-\varphi(x))^2 m(x,z) dz \Big\} Z^K_s(dx) \, ds \\
&\quad + \frac{1}{K} \int_0^t \int_0^\infty  \Big\{ \int_0^\infty  h(x,y,Z^K_s) (\varphi(x)-\varphi(y))^2 Z^K_s(dy) \Big\} Z^K_s(dx) \, ds.
\end{aligned}
\end{equation*}
Lemma \ref{lemma_moment_control} guarantees the existence of a constant $B>0$ such that 
\begin{equation}\label{eq:convergence variation quadratique}
    \mathbb{E}[\langle M^{K,\varphi}\rangle_t]\leq \frac{B}{K},\, \forall t\leq T.
\end{equation}

\subsection{Convergence towards a deterministic measure}\label{ssec:conv_deterministic}
Let $T>0$. In this section, we investigate the asymptotic behaviour of the processes $(Z^K_t)_{t\in [0,T]}$ as $K \to +\infty$, assuming that the initial conditions $Z^K_0$ converge to a non-trivial deterministic measure $\xi_0$.
 We can state the convergence theorem:
\begin{theorem}\label{theoreme cv}[Large population limit]
Assume that the hypotheses $\mathbf{H}$ hold. In addition, suppose that it exists $p>2$ such that
\begin{equation}\label{eq:hypothèses du modèle sur les moments}
    \sup_K \mathbb{E}[\langle Z^K_0, 1 \rangle^p] +\sup_K\mathbb{E}[\langle Z^K_0,d^\eta\rangle] < \infty,
\end{equation}
and that the initial conditions $Z^K_0$ converge in probability in $\mathbb{D}([0,T], M_F(\mathbb{R}_+))$ to a deterministic finite measure $\xi_0$. Then, as $K \to \infty$, the processes $(Z^K)$ converge in probability in $\mathbb{D}([0,T], M_F(\mathbb{R}_+))$ to the unique deterministic continuous function $\xi \in \mathcal{C}([0,T], M_F(\mathbb{R}_+))$ satisfying $\sup_{t\geq0}\langle\xi_t,1\rangle<\infty$ and for all bounded $\varphi : \mathbb{R}_+ \to \mathbb{R}$,
\begin{equation}\label{limit_measure}
\begin{aligned}
\langle \xi_t, \varphi \rangle &= \langle \xi_0, \varphi \rangle
+ \int_0^t \int_0^\infty  \varphi(x) \big[b(x) - d(x) - C \langle \xi_s, 1 \rangle\big] \xi_s(dx) \, ds \\
&\quad + \int_0^t \int_0^\infty  \gamma(x) \left( \int_0^\infty  (\varphi(z)-\varphi(x)) m(x,z) dz \right) \xi_s(dx) \, ds \\
&\quad + \int_0^t \int_0^\infty  \int_0^\infty  (\varphi(x) - \varphi(y)) \frac{\tau(x,y)}{\beta + \mu \langle \xi_s, 1 \rangle} \, \xi_s(dy) \, \xi_s(dx) \, ds.
\end{aligned}
\end{equation}
Moreover, the convergence in $L^2$ of $Z^K_0$ implies the convergence in $L^2$ of the processes $(Z^K_t)_{t\in[0,T]}$. 
\end{theorem}
\begin{proof}

We follow a well-known argument developed in \cite{fournier_meleard} and detailed in \cite{bansaye2015stochastic}.
More precisely, the proof consists in a tightness and uniqueness argument. Nevertheless there are some notable differences. First, the state space is not compact, the death rate is unbounded and the variation kernel might be heavy-tailed. Secondly, the transfer rate is not continuous. We start by proving the uniqueness of the limit.
\newline

\textbf{Uniqueness of the limit:}

Because of the non-boundedness of the death rate, we need to establish an alternative formula to \eqref{limit_measure}. We start by rewriting the equation \eqref{limit_measure} for a test-function $f\in C^1_b(\mathbb{R}, L^\infty(\mathbb{R}_+,\mathbb{R}))$.
Fix $t>0$, we get rid of the death term by setting
\begin{equation*}
    f^t(s,x)=\mathbb{E}_x\left[\varphi(Y_{t-s})e^{-\int_0^{t-s}d(Y_u)du}\right], \, s\leq t,
\end{equation*}
where $\varphi$ is a measurable bounded test-function on $\mathbb{R}_+$ and $Y$ a real-valued pure jump Markov process of generator
\begin{equation*}
    L\varphi(x)=\gamma(x)\int_0^\infty (\varphi(z)-\varphi(x))m(x,z)dz.
\end{equation*}
This function $f^t$ satisfies $f^t(t,x)=\varphi(x)$ and $\partial_s f^t(s,x)+Lf^t(s,x)-d(x)f^t(s,x)=0$. 

We can prove the uniqueness using the total variation norm $\|\mu_1-\mu_2\|_{TV}=\sup_{\|\varphi\|_\infty\le 1} |\langle \mu_1-\mu_2,\varphi\rangle|$.
Let $(\xi_t)_t$ and $(\overline{\xi}_t)_t$ be two solutions of \eqref{limit_measure} with the same initial condition, and assume
\[
\sup_{0\leq t\le T}\langle \xi_t+\overline{\xi}_t,1\rangle<\infty.
\]
Knowing that $\|\varphi\|_\infty\leq 1$ implies $\|f^t(s,\cdot)\|_\infty\leq1$, we obtain by standard computations the existence of a constant $L>0$ such that $
    \|\xi_t-\overline{\xi}_t\|_{TV}\leq L\int_0^t \|\xi_s-\overline{\xi}_s\|_{TV}ds.$
We conclude by Gronwall's Lemma.
\newline

\textbf{Tightness for the vague topology:}

We first endow $M_F(\mathbb{R}_+)$ with the vague topology. To show tightness of the sequence of laws $\mathcal{L}(Z^K)$ for the vague topology, following \cite{roelly1986criterion}, we need that for any continuous compactly supported function $\varphi$ on $\mathbb{R}_+$, the sequence of laws of the processes $(\langle Z^K,\varphi\rangle)$ is tight in $\mathbb{D}([0,T],\mathbb{R})$. 

We use the Aldous criterion \cite{aldous1978stopping} and the Rebolledo criterion \cite{joffe1986weak}. 
The check is carried out immediately because of Markov inequality and the control over the population size (by Lemma \ref{lemma_moment_control}); the mortality rate and the test-functions are bounded thanks to the compact support, so variations (even when heavy-tailed) do not add any difficulty.
\newline

\textbf{Tightness for the weak topology :}

To obtain the tightness for the weak topology, following \cite{jourdain2012levy}, let us prove:
\begin{equation}\label{eq:tension weak}
    \lim_{n\rightarrow\infty}\limsup_{K\rightarrow\infty} \mathbb{E}[\sup_{t\leq T}\langle Z^K_t,f_n\rangle]=0
\end{equation}
where $f_n$ is the function equal to $0$ before $n$, to $1$ after $n+1$ and linearly interpolated on $[n,n+1]$. Define $V(x)=d(x)^\eta$ and $\underline{V}(n)=\inf_{x>n}V(x)$. 

Consider $n$ sufficiently large such that $\underline{V}(n)>0$ (it is a consequence of (H.a)). If $x<n$, then $f_n=0\leq V(x)/\underline{V}(n)$. If $x\geq n$, then $f_n\leq 1\leq V(x)/\underline{V}(n)$, so we have
\begin{equation*}
    \langle Z^K_t,f_n\rangle \leq \frac{1}{\underline{V}(n)}\langle Z^K_t,V\rangle
\end{equation*}
hence
\begin{equation*}
    \limsup_{K\rightarrow\infty}\mathbb{E}[\sup_{t\leq T}\langle Z^K_t,f_n\rangle]\leq\frac{1}{\underline{V}(n)}\sup_K\mathbb{E}[\sup_{t\leq T}\langle Z^K_t,V\rangle]
\end{equation*}
and we conclude the proof of \eqref{eq:tension weak} by \eqref{eq : estimate lyapunov} and by letting $n$ tend to infinity. \eqref{eq : estimate lyapunov} allows us apply the criterion of Méléard-Roelly \cite{meleard1993convergences} which ensures that the sequence of law $(Z^K)_K$ is tight endowed with weak topology. Prokhorov's Theorem says that this sequence of distributions is relatively compact. Let $Z$ be a limiting-value. There exists a subsequence of $(Z^K)$ (which we also denote by $(Z^K)$) that converges in distribution to $Z$.

By Skorohod's Representation Theorem (see \cite{billingsley2013convergence}), passing to a new probability space and keeping the same notations, we may assume that $(Z^K)$ converges almost surely to $Z$.

The fact that the limit process belongs to $C([0,T],M_F(\mathbb{R}_+))$ and that the weak convergence holds is a consequence of \cite{meleard1993convergences}. The arguments are similar to the ones in Step 5 of the proof of Theorem 7.4 in \cite{bansaye2015stochastic}. The idea is that the jump size is $K^{-1}$ and tends to $0$.
\newline

\textbf{Identification of the limit :}

We need to check that 
\begin{equation}\label{eq: estimate limit}
    \mathbb{E}\left[\left(\int_0^T\langle Z_s,d\rangle ds\right)^2\right]<\infty.
\end{equation}
We know that $d$ is positive and continuous. For $n\geq0$, let $d_n(x)=n\wedge d(x)$. We define for all $\nu\in \mathbb{D}([0,T],M_F(\mathbb{R}_+))$:
\begin{equation*}
    \Phi_n(\nu)=\left(\int_0^T\int_0^\infty  d_n(x)\nu_s(dx)ds\right)^2.
\end{equation*}
By classical arguments (see \cite{ethier2009markov} Chapter 3), $\Phi_n$ is continuous with respect to the weak topology. By Fatou’s Lemma and \eqref{eq:estimate death}, we have that
\begin{equation*}
        \mathbb{E}[\Phi_n(Z)]\leq \liminf_{K\rightarrow\infty}\mathbb{E}[\Phi_n(Z^K)]\leq \sup_K\mathbb{E}\left[\left(\int_0^T\int_0^\infty  d(x)Z^K_s(dx)ds\right)^2\right]<\infty.
\end{equation*}
By the monotone convergence Theorem we have the a.s. convergence
\begin{equation*}
    \Phi_n(Z)\underset{n\rightarrow\infty}{\rightarrow}\Phi(Z):=\left(\int_0^T\int_0^\infty  d(x)Z_s(dx)ds\right)^2.
\end{equation*}
To finally obtain \eqref{eq: estimate limit}, we write by Fatou's Lemma :
\begin{equation*}
    \mathbb{E}[\Phi(Z)]=\mathbb{E}[\lim_{n\rightarrow\infty}\Phi_n(Z)]\leq \liminf_{n\rightarrow\infty}\mathbb{E}[\Phi_n(Z)]
\end{equation*}
and the bound does not depend on $n$.

Now, we identify the limit. For any integer $n$, for any $t\in[0,T]$, for any $\nu\in\mathbb{D}([0,T],M_F(\mathbb{R}_+))$, and for any $\varphi\in C_b(\mathbb{R}_+)$, we set:
\begin{equation*}
    \begin{aligned}
        F^{(n)}_t(\nu)&=\langle \nu_t,\varphi\rangle -\langle\nu_0,\varphi\rangle -\int_0^t\int_0^\infty  \varphi(x)(b(x)-d_n(x)-C\langle\nu_s,1\rangle)\nu_s(dx)ds\\
        &-\int_0^t\int_0^\infty \int_0^\infty  \gamma(x)(\varphi(z)-\varphi(x))m(x,z)dz\nu_s(dx)ds\\
        &-\int_0^t\int_0^\infty \int_0^\infty  h(y,x,\nu_s)(\varphi(y)-\varphi(x))\nu_s(dy)\nu_s(dx)ds.
    \end{aligned}
\end{equation*}
We denote by $F_t$ the same functional with $d$ replacing $d_n$.
The death term $d_n$ is bounded by $n$ and the horizontal transfer term is continuous. Indeed, the function $(x,y)\mapsto \mathds{1}_{\{y>x\}}(\varphi(y)-\varphi(x))$ is continuous and bounded. Moreover, when $\beta=0$, division by zero is avoided since
\begin{equation*}
\left|\frac{\tau}{\mu\langle\nu,1\rangle}\int_0^\infty \int_0^\infty  \mathds{1}_{\{y>x\}}(\varphi(y)-\varphi(x))\nu(dy)\nu(dx)\right|\leq 2\frac{\tau\|\varphi\|_\infty}{\mu}|\langle\nu,1\rangle|.
\end{equation*}
The functional $F_t^{(n)}$ is continuous on $M_F(\mathbb{R}_+)$. We note that

\begin{equation}\label{eq:functional on ZK}
    F^{(n)}_t(Z^K)=M^{K,\varphi}_t+\int_0^t\int_0^\infty \varphi(x)(d_n(x)-d(x)) Z^K_s(dx)ds.
\end{equation}
We can deduce the uniform integrability for $(F^{(n)}_t(Z^K))_K$ by
using \eqref{eq:estimate death} and \eqref{eq:convergence variation quadratique}.

By the continuity of the functional $F^{(n)}$ and the continuity of the limit $Z$, we have that $F^{(n)}_t(Z^K)$ converges almost surely to $F^{(n)}_t(Z)$, so by uniform integrability, 
\begin{equation*}
    \lim_{K\rightarrow\infty}\mathbb{E}[|F^{(n)}_t(Z^K)|]=\mathbb{E}[|F^{(n)}_t(Z)|].
\end{equation*}
Taking the limit as $K$ goes to infinity, we obtain from \eqref{eq:functional on ZK}:
\begin{equation*}
    \mathbb{E}[|F^{(n)}_t(Z)|]\leq \lim_{K\rightarrow\infty}\sqrt{\frac{B}{K}}+ \lim_{K\rightarrow\infty}\mathbb{E}\left[\left|\int_0^t\int_0^\infty  \varphi(x)(d_n(x)-d(x))Z^K_s(dx)ds\right|\right].
\end{equation*}
Let us now prove that 
\begin{equation}\label{eq : convergence en K via l'estimée d(x)}
    \lim_{K\rightarrow\infty}\mathbb{E}\left[\left|\int_0^t\int_0^\infty  \varphi(x)(d_n(x)-d(x))Z^K_s(dx)ds\right|\right]=\mathbb{E}\left[\left|\int_0^t\int_0^\infty  \varphi(x)(d_n(x)-d(x))Z_s(dx)ds\right|\right].
\end{equation}
It is done by observing that the sequence of deterministic measures $(\nu^K)_K$ defined by
\begin{equation*}
    \langle \nu^K_t,g\rangle =\mathbb{E}\left[\int_0^t\langle Z^K_s,g\rangle ds\right]
\end{equation*}
uniformly integrates any function whose growth is at most linear in $d(x)$, according to the estimates \eqref{eq : estimate lyapunov}. 
We can take the limit and obtain \eqref{eq : convergence en K via l'estimée d(x)}.

To conclude, we have 
\begin{equation*}
    |\varphi(x)(d_n(x)-d(x))|\leq\|\varphi\|_\infty d(x)\mathds{1}_{\{d(x)\geq n\}}.
\end{equation*}
Now, for any $\omega$, $Z_s(\omega,dx)ds$ is a finite measure on $[0,t]\times\mathbb{R}_+$ and $d$ is integrable with respect to this measure. By dominated convergence, we obtain that the quantity
\begin{equation*}
    X_n:=\int_0^t\int_0^\infty  \varphi(x)(d_n(x)-d(x))Z_s(dx)ds
\end{equation*}
tends to $0$ $a.s.$ as $n$ goes to infinity.
Since $|X_n|\leq \|\varphi\|_\infty \int_0^t\langle Z_s,d\rangle ds$, once again by dominated convergence we obtain that 
\begin{equation*}
    \lim_{n\rightarrow\infty}\mathbb{E}\left[\left|\int_0^t\int_0^\infty  \varphi(x)(d_n(x)-d(x))Z_s(dx)ds\right|\right]=0.
\end{equation*}
By Fatou’s lemma, we conclude that $\mathbb{E}[|F_t(Z)|]=0$ and thus we have identified the limit.
The $L^2$ convergence is a consequence of the uniform integrability and the estimates of Lemma \ref{lemma_moment_control}.

The boundedness of $\langle\xi_t,1\rangle$ is proven below in Section \ref{ssec:integro_diff}. Therefore, we can extend \eqref{limit_measure} from continuously bounded to measurable bounded test-functions.
\end{proof}
\subsection{Properties of the integro-differential equation}\label{ssec:integro_diff}

A first practical result is

\begin{proposition}\label{bornitude de la taille de population}
Let $\lambda_t := \langle \xi_t, 1 \rangle$ denote the total mass of the solution of \eqref{limit_measure} at time $t$. Then, under the assumptions of Theorem~\ref{theoreme cv}, for any $T>0$:
\begin{equation}\label{eq: bornitude pop macro}
\inf_{t \in [0,T]} \lambda_t > 0,\qquad \sup_{t \ge 0} \lambda_t \le \max\left(\frac{\overline{b}}{C}, \lambda_0\right) .
\end{equation}
\end{proposition}

\noindent
Therefore, in the deterministic large-population limit, the population size remains strictly positive and uniformly bounded over time.

\begin{proof}
The proof consists in bounding $(\lambda_t)$ between explicit super- and sub-solutions.
\end{proof}

\begin{proposition}\label{prop:densité limite}
Under the Assumption \textbf{(H)}, if $\xi_0$ is absolutely continuous with respect to the Lebesgue measure, then so is $\xi_t$ for all $t>0$, with density $\xi_t(x)$ solving
\begin{equation}\label{densité limite}
\begin{aligned}
    \partial_t\xi_t(x) &=
    \Big(r(x) - C\!\int_0^\infty  \xi_t(y)\,dy\Big)\xi_t(x)+ \int_0^\infty[ \gamma(y)\xi_t(y)m(y,x)-\gamma(x)\xi_t(x)m(x,y)]dy \\
    & + \tau\xi_t(x)\left(\beta + \mu \int_0^\infty \xi_t(z)\,dz\right)^{-1}\int_0^\infty \text{sgn}(x-y)\xi_t(y)\,dy.
\end{aligned}
\end{equation}
$\xi$ is time-differentiable and measurable in space and for all $t$ in $[0,T]$ we have $\partial_t\xi_t\in L^1(\mathbb{R}_+)$. Moreover, key regularity properties of the initial condition are preserved over time: the strict positivity and continuity of $\xi_0$ are propagated to $\xi_t$ for all $t\in[0,T]$, and if $\xi_0$ is bounded, then $\xi_t$ remains bounded on $[0,T]\times\mathbb{R}_+$.
\end{proposition}

\begin{proof}
(Existence of a density) - It is obtained with a similar computation as in \cite{fournier_meleard}. Let $A \subset \mathbb{R}_+$ be a Borel set of Lebesgue measure zero. From equation \eqref{limit_measure}, the continuity of the kernel $m$, the bounds \eqref{eq: bornitude pop macro} and using that $\xi_0$ is absolutely continuous, there exists a constant $L>0$ such that
\[
\langle \xi_t, \mathds{1}_A \rangle \le L\int_0^t  \langle \xi_s, \mathds{1}_A \rangle ds.
\]
Gronwall's lemma then implies $\langle \xi_t, \mathds{1}_A \rangle = 0$ for all $t \le T$, hence $\xi_t$ is absolutely continuous.

(Weak formulation of the density equation) - For any bounded continuous $\varphi$ on $\mathbb{R}_+$, applying Fubini's theorem yields
\begin{equation*}
\begin{aligned}
\int_0^\infty  \varphi(x) \xi_t(x) dx &= \int_0^\infty  \varphi(x) \xi_0(x) dx 
+ \int_0^t \int_0^\infty  \varphi(x) \Big[ r(x)- C\langle \xi_s,1 \rangle \Big] \xi_s(x) dx\, ds \\
&\quad + \int_0^t \int_0^\infty  \int_0^\infty  \varphi(x) \, [\gamma(z)\xi_s(z) m(z,x) -\gamma(x)\xi_s(x)m(x,z)] dz \, dx \, ds \\
&\quad + \int_0^t \int_0^\infty  \int_0^\infty  \varphi(x) \frac{\tau(x,y) - \tau(y,x)}{\beta + \mu \langle \xi_s,1 \rangle} \xi_s(x) \xi_s(y) dy \, dx \, ds.
\end{aligned}
\end{equation*}

Since this holds for all bounded continuous $\varphi$, it follows that $\xi_t(x)$ is a weak solution of \eqref{densité limite}, defined almost everywhere for all $t \in [0,T]$.

(Existence of a strong solution) - The existence of a strong solution to \eqref{densité limite} follows from the Cauchy-Lipschitz theorem. Given that $\xi_t$ is finite almost everywhere, $\langle\xi_t, \varphi\rangle$ is well-defined for any measurable bounded function $\varphi$. For a mapping $f(x, y)$ that is continuous in $x$ and measurable bounded in $y$, we observe that the function $x \mapsto \langle\xi_t, f(x, \cdot)\rangle$ inherits the continuity in $x$. Since $t \mapsto \langle\xi_t, f(x, \cdot)\rangle$ is also continuous, the equation for each fixed $x \in \mathbb{R}_+$ reduces to a linear ODE with bounded, continuous coefficients. 
Fix $x\in\mathbb{R}_+$ and consider
\begin{equation}\label{ODE_indexed_by_x}
\begin{cases}
\partial_t u_x(t) = A_x(t) u_x(t) + B_x(t) \\
u_x(0) = \xi(0,x)
\end{cases}
\end{equation}
for $t \in [0,T]$, assuming $0 < \xi(0,x) < \infty$. The coefficients are
\begin{align*}
A_x(t) &= r(x) - C \lambda_t - \gamma(x) + \frac{\langle\xi_t,\alpha_x\rangle}{\beta + \mu \lambda_t}, \\
\alpha_x(y) &= \alpha(x,y) = \tau \text{sgn}(x-y),\\
B_x(t) &= \int_0^\infty  \gamma(y) m(y,x) \xi_t(y) dy = \langle \xi_t,\gamma(\cdot) m(\cdot,x) \rangle.
\end{align*}

As we discussed, $A_x$ and $B_x$ are continuous in time on $[0,T]$. By the Cauchy-Lipschitz Theorem and the fact that $A_x$ and $B_x$ are continuous and bounded, for each $x \in \mathbb{R}_+$, if $\xi(0,x)$ is strictly positive and finite and Assumption $\textbf{(H)}$ holds, the ODE \eqref{ODE_indexed_by_x} admits a unique continuously differentiable solution $u_x$ on $[0,T]$. The positivity of $\xi(0,x)$ and $B_x$ implies the positivity of $u_x$ on $[0,T]$. 

We shall show that the collection $(u_x)_{x\geq0}$ allows us to construct a solution to the equation \eqref{densité limite}. We define
\begin{equation*}
    \widehat\xi(t,x) = u_x(t), \quad \forall t \in [0,T], \, x \in \mathbb{R}_+.
\end{equation*}

To do this, we must verify that $\widehat{\xi}$ is integrable with respect to $x$. We use a super-solution argument.

Let $\overline{A}(t) := \overline{b} + \frac{\tau \lambda_t}{\beta + \mu \lambda_t} \ge A_x(t)$, and consider $\overline{u}$ as the explicit solution to
\begin{equation}\label{eq:sur-solution}
\partial_t \overline{u}(t,x) = \overline{A}(t)\overline{u}(t,x) + B_x(t), \quad \overline{u}(0,x) = \xi(0,x).
\end{equation}
Under assumption \eqref{hyp:mutation kernel}, $\overline{u}$ is integrable in $x$:
\begin{equation*}
\int_0^\infty \overline{u}(t,x) dx \le \lambda_0 e^{\int_0^t \overline{A}(v) dv} + \frac{\overline{\gamma}\overline{b}}{C} \int_0^t e^{\int_v^t \overline{A}(s)ds} dv < \infty.
\end{equation*}
Because for any $t,x\geq0,\, \widehat{\xi}(t,x)\leq \overline{u}(t,x)$, $\widehat{\xi}$ is integrable in $x$, we deduce that $\widehat\xi$ is a strong solution of \eqref{densité limite}. By uniqueness, any solution coincides with $\widehat\xi$ on $[0,T]\times\mathbb{R}_+$ up to a set of Lebesgue measure zero. Hence, we work with $\widehat\xi$, which we will now denote by $\xi$.

\textit{(Other properties)} - The fact that the time derivative of $\xi$ is spatially integrable follows directly from the spatial integrability of $\xi$ together with the properties of the variation kernel.  

The strict positivity of $\xi_0$ implies the strict positivity of $\xi_t$ by a sub-solution argument, obtained by providing a uniform lower bound on $A_x$ and $B_x$. Similarly, uniform boundedness follows by controlling $B_x$ from above in \eqref{eq:sur-solution}.  

Continuity in $x$ is obtained using Gronwall’s lemma.

\end{proof}
The next proposition is needed in subsection \ref{ssec:dual_extension} to ensure the necessary regularity of the involved semigroups.
\begin{proposition}\label{prop:produit_par_x_fini}
The estimate
\begin{equation*}
    \sup_{x \ge 0} x\xi_t(x) < \infty \quad \text{for all } t \in [0, T]
\end{equation*}
holds under either of the following conditions:

\begin{enumerate}
    \item The map $x \mapsto \int_0^\infty z^p m(x,z) \,\mathrm{d}z$ is bounded or integrable, and $\int_0^\infty x^p \xi_0(x) \,\mathrm{d}x < \infty$ for some $p \ge 1$.
    \item In the case of the reflected kernel defined in \eqref{eq:reflected kernel}, if the following assumptions are satisfied:
    \begin{itemize}
        \item $\sup_{x \ge 0} x\xi_0(x) < \infty$,
        \item $m(x,y) = \hat{m}(x+y) + \hat{m}(x-y)$, where $\hat{m}$ is an even probability density on $\mathbb{R}$,
        \item $\sup_{x \ge 0} x\hat{m}(x) < \infty$.
    \end{itemize}
\end{enumerate}
\end{proposition}
The proof is postponed in Appendix  \ref{proof:reflected_supremum}.

\begin{remark}
For what follows, we will need to work on the full time interval $\mathbb{R}$, despite the time reversal being considered only on $[0,T]$.
It is therefore necessary to extend the definition of $\xi$ to $\mathbb{R}\times\mathbb{R}_+$. The simplest way to do this is by setting for any $T>0$
\begin{equation*}
    \begin{cases}
    \xi(t,x)=\xi(0,x), & \forall t<0, \\
    \xi(t,x)=\xi(T,x), & \forall t>T.
    \end{cases}
\end{equation*}
Consequently, we have $\sup_{(t,x)\in\mathbb{R}\times\mathbb{R}_+} \xi(t,x) < \infty.$
\end{remark}
Starting from here, we will suppose that $\xi_0$ admits a density.
\subsection{Numerical resolution and conjectures on $\xi$}\label{ssec:numerical}
In this section, we present several representations of the solution of \eqref{densité limite} in order to highlight its qualitative behaviour and to motivate its study.

The numerical scheme that is employed is standard: an explicit finite-difference scheme combined with numerical integration via the left Riemann sum (rectangle rule).
We fixed the parameters according to Table \ref{tab:experimental parameters}.
The optimal trait is $0$: the farther away from it, the larger the death rate $d(x)$.  
Parameters not specified in the table will be given directly in the figures.
These are the parameters of interest that we will let vary. We recall their meaning : $\tau$ is the intensity of transfer events, $\beta$ is the mean time of encounter between two individuals before the transfer and $\mu$ is the mean time that the transfer takes effectively once two individuals meet (see \cite{A1,A2}). The strict positivity of $\beta$ switches the model from a frequency-dependent to a Beddington-DeAngelis transfer regime
We mainly observe three distinct behaviours.

The simulated dynamics demonstrate how horizontal transfer shapes population structures and ancestral lineages. Under a density-dependent or Beddington-DeAngelis regime, temporal oscillations damp out as the population converges toward a stable and positive steady state (see Figures \ref{fig:densite_bda_3} and \ref{fig:tpop_bda_3}). Conversely, under a frequency-dependent regime, when $\tau$ is sufficiently large, oscillations take a very long time to damp out toward a positive steady state (see Figures \ref{fig:densite_fd_3} and \ref{fig:tpop_fd_3}), while an excessively large $\tau$ drives rapidly the population to extinction. In the absence of transfer, the cost of the burden leads to its clearance through selection. However, increasing the transfer rate forces the distribution toward high values, while simultaneously reducing the total population size due to the cumulative metabolic burden. These results suggest that present-day individuals are unlikely to descend from lineages that were continuously saturated with deleterious traits, as such heavy loads increase the risk of population extinction. Instead, contemporary individuals likely originate from ancestral lineages that evolved under regulated or fluctuating transfer regimes, allowing them to maintain a light burden without compromising long-term survival.

As mentioned in the introduction, the work presented in \cite{A3, garriz2026hamilton, mirrahimi} explores the behaviour of close deterministic models both analytically and numerically. The study of the long-term behaviour of the scaling limit and the existence of a non-zero stationary profile is the subject of \cite{deangeliBravo2026existence}.
\begin{table}[h!]
\centering
\begin{tabular}{|l|l|}
\hline
\textbf{Parameter} & \textbf{Value / Expression} \\
\hline
Domain & $[0,5]$ \\
\hline
Birth rate & $b(x)=5$ \\
\hline
Death rate & $d(x)=5+\tfrac{1}{2}(x-3)(x+3)$ \\
\hline
variation rate & $\gamma(x)=0.15$ \\
\hline
variation kernel & $m(x,y)=\hat{m}(x-y)+\hat{m}(x+y)
\text{ where }\hat{m}(x)=\tfrac{1}{\sqrt{2\pi\sigma^2}}e^{-\tfrac{x^2}{2\sigma^2}}$
\\
\hline
variation variance & $\sigma=0.1$ \\
\hline
Competition & $C=\tfrac{1}{2}$ \\
\hline
Transfer kernel & $\tau(x,y)=\tau\mathds{1}_{\{x>y\}}$ \\
\hline
Initial condition & $\xi(0,x)=\tfrac{1}{\sqrt{2\pi}}e^{-\tfrac{x^2}{2}}\mathds{1}_{x\geq0}$ \\
\hline
\end{tabular}
\caption{Model parameters used in the simulations.}\label{tab:experimental parameters}
\end{table}
\begin{figure}[h!]
    \centering
    \begin{subfigure}[b]{0.48\linewidth}
        \centering
        \includegraphics[width=\linewidth]{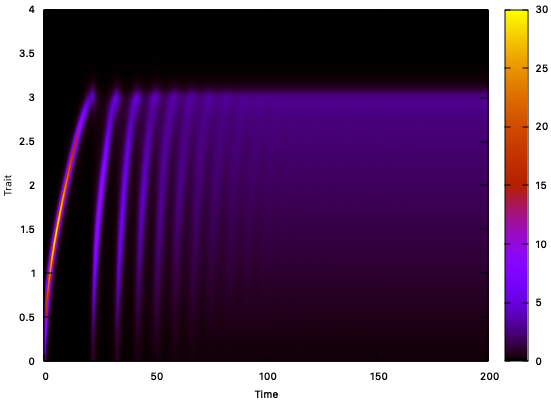}
        \caption{Plot of $\xi$, with time on the horizontal axis and trait value on the vertical axis - $\tau=3,\mu=\beta=1$ (BDA).}
        \label{fig:densite_bda_3}
    \end{subfigure}
    \hfill
    \begin{subfigure}[b]{0.48\linewidth}
        \centering
        \includegraphics[width=\linewidth]{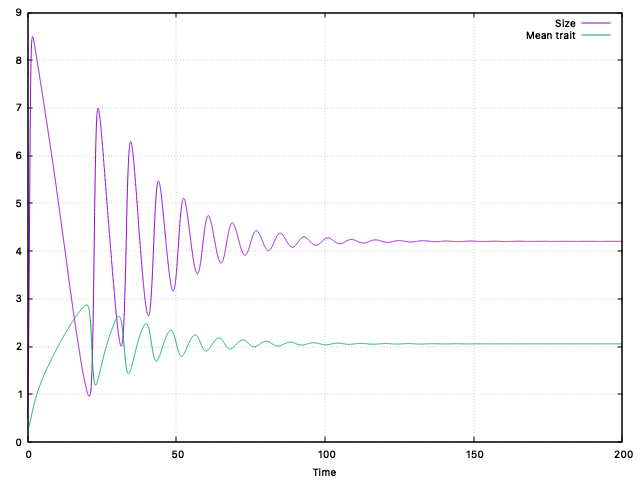}
        \caption{Population size and mean trait value over time. - $\tau=3,\mu=\beta=1$ (FD).}
        \label{fig:tpop_bda_3}
    \end{subfigure}
\end{figure}
\newpage

\begin{figure}[h!]
    \centering
    \begin{subfigure}[b]{0.48\linewidth}
        \centering
        \includegraphics[width=\linewidth]{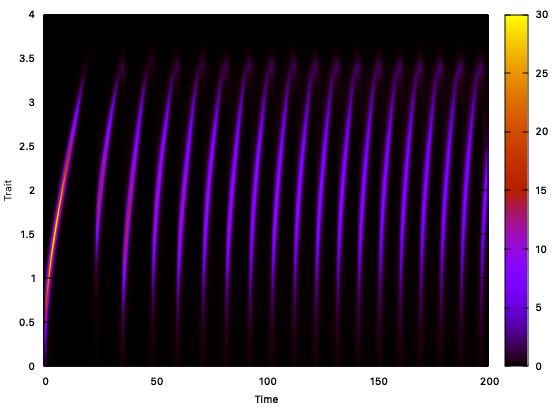}
        \caption{Plot of $\xi$, with time on the horizontal axis and trait value on the vertical axis - $\tau=3,\mu=1,\beta=0$.}
        \label{fig:densite_fd_3}
    \end{subfigure}
    \hfill
    \begin{subfigure}[b]{0.48\linewidth}
        \centering
        \includegraphics[width=\linewidth]{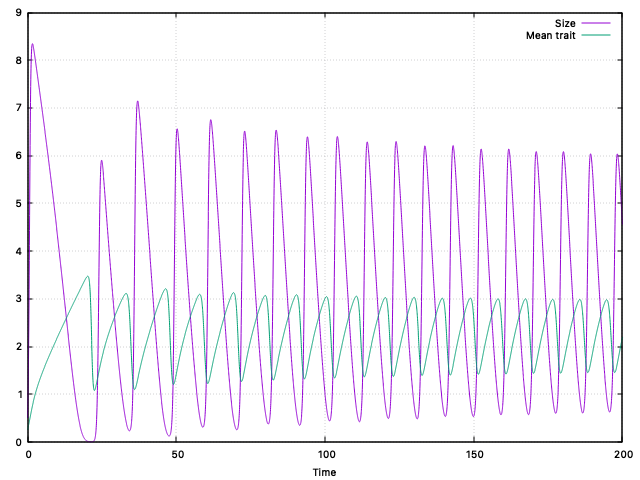}
        \caption{Population size and mean trait value over time - $\tau=3,\mu=1,\beta=0$.}
        \label{fig:tpop_fd_3}
    \end{subfigure}
\end{figure}

\section{Feynman-Kac approach}\label{sec:fk approach}
Up to this point, we have worked with $(Z^K)_K$ and $\xi$, respectively the individual-based processes and their large-population approximation. These objects describe the dynamics of the population, and in particular $\xi$ captures its macroscopic behaviour on finite time intervals. We now shift to the individual scale, while keeping track of the global dynamics, in order to characterize surviving lineages within the population.

We first introduce auxiliary processes that enjoy the branching property which is essential for spinal techniques. We then extend these techniques to entire trajectories. Ultimately, this framework will allow us to establish the existence and characterization of the surviving lineage, namely the spine.
An essential tool for the study of lineages is a Feynman--Kac type formula that allows us to move from the individual-based process to a real-valued process with an associated bias term. 
\subsection{Coupling with the auxiliary branching population process}\label{ssec:coupling_branching}
Following the approach of \cite{meleardB1}, we define an auxiliary process $(\widehat{Z}^K)$ on the same probability space as $(Z^K)_K$. It is defined by an equation similar to \eqref{eq:semimartingale} for $(Z^K)$ but where the competition and transfer terms have been replaced using $\xi$, the solution of \eqref{densité limite} (meaning that the term $C\langle Z^K_s,1\rangle$ is substituted by $C\langle\xi_s,1\rangle$ and that $h(x,y,Z^K_s)$ is substituted by $h(x,y,\xi_s)$). The full equation is given in Appendix \ref{app:historical}, see \eqref{processus auxiliaire}). 

The main interest of this auxiliary process is that it satisfies the branching property, unlike $(Z^K)$, which will be crucial in the analysis of the spinal decomposition. It should be noted, however, that while the branching property is gained, time-homogeneity is lost, since $(\xi_t)$ is not stationary.

We can prove the convergence result justifying the use of $\widehat{Z}^K$ :
\begin{theorem}
Under the same assumptions as in Theorem~\ref{theoreme cv}, the sequence $(\widehat{Z}^K)$ converges weakly to the unique continuous deterministic function solution of \eqref{limit_measure}.  
Moreover, for any bounded continuous function $\varphi$, we have
\[
\lim_{K\to\infty} \mathbb{E}\Big[\sup_{t \leq T} \big|\langle Z_t^K, \varphi \rangle - \langle \widehat{Z}_t^K, \varphi \rangle\big|^2\Big] = 0.
\]
\end{theorem}
\begin{proof}
The convergence of $(\widehat{Z}^K)_K$ to $\xi$ is identical to that of $(Z^K)_K$. The result then follows from the inequality
\begin{equation*}
    |\langle Z^K_t,\varphi\rangle-\langle\widehat{Z}^K_t,\varphi\rangle|\leq|\langle Z^K_t,\varphi\rangle-\langle \xi_t,\varphi\rangle|+|\langle\widehat{Z}^K_t,\varphi\rangle-\langle\xi_t,\varphi\rangle|.
\end{equation*}
\end{proof}
\subsection{Definition and convergence of the historical processes}\label{ssec:historical_def_conv}
In order to sample lineages from the population over a time interval $[0,T]$, we need to introduce the so-called historical process $H^K$ \cite{meleardB0}, which is a random measure taking values in the space of trajectories $\mathbb{D}(\mathbb{R},\mathbb{R}_+)$, defined by
\[
H^K_t \;=\; \frac{1}{K}\sum_{i\in V^K_t} \delta_{(X^i_{s\wedge t},\, s\in\mathbb{R}_+)},
\]

We can also define an auxiliary historical process $\widehat{H}^K$ (naturally associated with $\widehat{Z}^K$), in which competition and transfer are frozen through the density $\xi$. Like $(Z^K)$ and $(\widehat{Z}^K)$, the two historical processes $(H^K)$ and $(\widehat{H}^K)$ are built on the same probability space with the same initial condition and the same Poisson point measures.
Following the framework of \cite{meleardB1, meleardB2}, this section extends the convergence in distribution of the process $(Z^K)$. This extension is required to account for both the time non-homogeneity of $X$ and the strong non-linearity induced by the transfer function. Consequently, we are led to consider additional Markov processes taking values in $\mathbb{D}([0,T],\mathbb{R}_+)$; we refer the reader to the Appendix \ref{app:historical} for further details and for the proof of the following convergence result:
\begin{proposition}\label{cv historique}
Under the same assumptions as in Theorem~\ref{theoreme cv} and Proposition \ref{prop:densité limite},
for any bounded continuous function $\Phi$ of the form \eqref{eq:test functions on D} on $\mathbb{D}([0,T],\mathbb{R}_+)$, we have
\begin{equation*}
    \lim_{K\to\infty}\mathbb{E}\Big[\sup_{t\leq T}\big|\langle H^K_t,\Phi\rangle-\langle \widehat{H}^K_t,\Phi\rangle\big|^2\Big]=0.
\end{equation*}
\end{proposition}
The processes $\widehat{Z}^K$ and $\widehat{H}^K$ satisfy the branching property. Hence, the spinal techniques developed in \cite{bansaye2011limit,marguet} can be used. It is also sufficient to restrict our attention to processes starting from a single individual.

\subsection{Feynman-Kac formula}\label{ssec:fk_formula}
We may now establish the Feynman-Kac formula, which allows us to characterize the behaviour of all lineages in the population through a single lineage biased by size, i.e. we perform the spinal decomposition \cite{bansaye2011limit,cloez2017limit,marguet2019uniform} by transferring the variations of a single lineage (variation and horizontal transfer) to a process $X$ taking values in $\mathbb{R}_+$, while the population growth component is incorporated into an exponential functional. We also introduce the various functional spaces that will be used subsequently and explain their respective roles.

\textbf{Important functional spaces:}
\begin{itemize}
    \item $C^1_b(\mathbb{R},L^\infty(\mathbb{R}_+,\mathbb{R}))$: bounded $C^1$ functions with bounded derivative in the first coordinate and bounded measurable in the second one. This is the natural space for the semigroups and generators under consideration.
    \item $C^1_c(\mathbb{R},L^\infty(\mathbb{R}_+,\mathbb{R}))$: $C^1$ functions with compact support in the first coordinate and bounded measurable in the second one, convenient for discarding boundary terms.
    \item $\mathcal{A} = \{ (s,x) \mapsto \mathds{1}_A(s) f(s,x) \mid A \subset \mathbb{R} \text{ finite interval}, f \in C^1(\mathbb{R},L^\infty(\mathbb{R}_+,\mathbb{R}))\}$: a small extension, notably needed to define the excessive measure for time reversing the spine.
\end{itemize}
We denote by $X$ the process defined by
\begin{equation*}
\begin{aligned}
X_t &= X_0 + \int_0^t \int_0^\infty  \int_0^\infty  (z - X_{u-}) \, \mathds{1}_{\{\theta \le \gamma(X_{u-}) m(X_{u-}, z)\}} \, N_1(du, dz, d\theta) \\
&\quad + \int_0^t \int_0^\infty  \int_0^\infty  (y - X_{u-}) \, \mathds{1}_{\{\theta \le h(y, X_{u-}, \xi_u) \, \xi_u(y)\}} \, N_2(du, dy, d\theta),
\end{aligned}
\end{equation*}
where $N_1$ and $N_2$ are independent Poisson point measures on $\mathbb{R}_+^3$ with intensity measure $dudzd\theta$. The process $X$ describes the trait changes occurring along a single lineage: the first integral corresponds to variations, and the second to horizontal trait transfers between individuals in the population. As with $(\widehat{Z}^K)_K$, for the transfer event, the new trait is sampled from the population described by the density~$\xi$.

Since $X$ is time-inhomogeneous due to the time dependence in $\xi_t$, we consider its extension to time-space $(t, X_t)_{[0,T]}$, which forms a two-dimensional homogeneous Markov process. Let $Q$ denote its infinitesimal generator, computed via Itô’s formula. It is given for any bounded measurable function $\varphi$ on $\mathbb{R} \times \mathbb{R}_+$ differentiable in the first coordinate by
\begin{equation*}
\begin{aligned}
Q\varphi(s,x) &= \partial_1 \varphi(s,x) + \gamma(x) \int_0^\infty  (\varphi(s,z) - \varphi(s,x)) m(x,z) dz \\
&\quad + \int_0^\infty  h(y,x,\xi_s) (\varphi(s,y) - \varphi(s,x)) \, \xi_s(y)dy.
\end{aligned}
\end{equation*}
The generator $Q$ encodes the trait variations given by variation and horizontal transfer. The semigroup of $(t, X_t)_{[0,T]}$ is denoted by $P_t$ and defined as
\[
P_t \varphi(s,x) = \mathbb{E}[\varphi_{s+t}(X_{s+t}) \mid X_s = x]
\]
where, to lighten the notation, we sometimes write $\varphi_{s}(x)$ to denote $\varphi(s,x)$.
The population-size variations are encoded in the growth function $r^\lambda$ defined by
\begin{equation*}
    r^\lambda(s,x)=b(x)-d(x)-C\lambda_s.
\end{equation*}

We denote  by $\widehat{Z}_t = K \widehat{Z}^K_t$, the auxiliary measure-valued branching process with jump length $1$. Like $X$, the process $\widehat{Z}$ is time-inhomogeneous. Analogously, we extend $\widehat{Z}_t$ to the time-space measure $\delta_t \otimes \widehat{Z}_t$:
\begin{equation*}
    \begin{aligned}
    \langle \delta_t\otimes\widehat{Z}_t,\varphi\rangle &=
    \langle \delta_0\otimes\widehat{Z}_0,\varphi\rangle+
    \widehat{\mathcal{M}}^{\varphi}_t
    +\int_0^t
    \int_0^\infty \int_\mathbb{R}\{\partial_1\varphi(u,x)\}\delta_s(du)\widehat{Z}_s(dx)ds \\
    &+\int_0^t \int_0^\infty \int_\mathbb{R} \{ r^\lambda(u,x)\varphi(u,x)\}\delta_s(du)\widehat{Z}_s(dx)ds\\
    &+\int_0^t \int_0^\infty \int_\mathbb{R}
    \left\{ \gamma(x)\int_0^\infty (\varphi(u,z)-\varphi(u,x))m(x,z)dz \right\}
    \delta_s(du)\widehat{Z}_s(dx)ds\\
    &+\int_0^t \int_0^\infty \int_\mathbb{R}\left\{ 
    \int_0^\infty h(y,x,\xi_u)(\varphi(u,y)-\varphi(u,x))\xi_u(dy)
    \right\}\delta_s(du)\widehat{Z}_s(dx)ds
    \end{aligned}
    \end{equation*}
where $\widehat{\mathcal{M}}^\varphi$ is the martingale part.
    \begin{remark}\label{rmk:uniqueness of time-space measure}
Let $(\nu_t)_{t}$ be a measure on $\mathbb{R}_+\times\mathbb{R}_+$. We state the following equation
\begin{equation}\label{Equation E}
    \forall\varphi\in C^1_b(\mathbb{R},L^\infty(\mathbb{R}_+,\mathbb{R})),\,\langle \nu_t,\varphi\rangle
    = \langle \nu_0,\varphi\rangle
    + \int_0^t \langle \nu_u,\, r^\lambda\varphi + Q\varphi \rangle \, du.
\end{equation}
Uniqueness of the solution, up to the initial condition, is guaranteed by Theorem~\ref{theoreme cv}. Indeed, since the solution to \eqref{densité limite} $\xi \in C([0, T], M_F(\mathbb{R}_+))$ is unique, its evaluation at any fixed time $t$, denoted by $\xi_t$, is a uniquely determined finite measure on $\mathbb{R}_+$. Consequently, the product measure $\delta_t \otimes \xi_t$ is also uniquely defined because the component $(\delta_t)_t$ is deterministic.

\end{remark}

Using the Feynman-Kac formula, the evaluation of the random measure $\widehat{Z}$ can be expressed through a single lineage biased by population size. More precisely, the Feynman-Kac semigroup is defined as follows:

\begin{proposition}\label{prop:def_semigroup_FK}
For any function $\varphi \in \mathcal{A}$, we define the Feynman-Kac semigroup associated to the growth function $r^\lambda$ and the infinitesimal generator $Q$ by
\begin{equation}\label{eq:semigroupe FK}
    \widehat{P}_t \varphi(s,x) := \mathbb{E}\Big[ e^{\int_0^t r^\lambda_{s+u}( X_{s+u}) \, du} \, \varphi_{s+t}( X_{s+t}) \,\big|\, X_s=x \Big].
\end{equation}
We have that 
\begin{equation}\label{eq:correspondance semigroupe et pvm}
     \widehat{P}_t \varphi(s,x)
= \mathbb{E}_{\delta_s \otimes \delta_x} \big[ \langle \delta_{s+t} \otimes \widehat{Z}_{s+t}, \varphi \rangle \big].
\end{equation}
\end{proposition}
\begin{proof}
Let us denote by \(\nu_t\) the measure defined for any \(\varphi\in C^1_b(\mathbb{R},L^\infty(\mathbb{R}_+,\mathbb{R}))\) by
\[
\langle \nu_t, \varphi \rangle = \mathbb{E}_{\delta_s \otimes \delta_x} \big[ \langle \delta_{s+t} \otimes \widehat{Z}_{s+t}, \varphi \rangle \big].
\]  
We have
\begin{align*}
\langle \delta_{s+t} \otimes \widehat{Z}_{s+t}, \varphi \rangle 
&= \langle \delta_s \otimes \widehat{Z}_s, \varphi \rangle 
   + \widehat{\mathcal{M}}_{s+t}^\varphi-\widehat{\mathcal{M}}_{s}^\varphi
   +\int_s^{s+t} \langle \delta_u \otimes \widehat{Z}_u, r^\lambda \varphi + Q \varphi \rangle \, du  \\
&= \langle \delta_s \otimes \widehat{Z}_s, \varphi \rangle 
   + \widehat{\mathcal{M}}_{s+t}^\varphi-\widehat{\mathcal{M}}_{s}^\varphi
   +\int_0^t \langle \delta_{s+u} \otimes \widehat{Z}_{s+u}, r^\lambda\varphi + Q \varphi \rangle \, du.
\end{align*}
Taking expectations, we obtain that $\nu_t$ satisfies \eqref{Equation E}

Next, we show that the left-hand side of the equality satisfies the same equation.
Define the measure \(\mu_t\) by
\[
\langle \mu_t, \varphi \rangle = \mathbb{E}\big[ e^{\int_0^t r^\lambda_{s+u}(X_{s+u})\,du} \varphi_{s+t}(X_{s+t}) \mid X_s = x \big].
\]
Combining Dynkin's formula for the process $(t,X_t)_{[0,T]}$, the integration by parts formula and taking conditional expectation, we obtain that $\mu_t$ also satisfies \eqref{Equation E}.

Both measures \(\nu_t\) and \(\mu_t\) are solutions of (\ref{Equation E}) with the same initial condition. By uniqueness (Theorem~\ref{theoreme cv} and Remark \ref{rmk:uniqueness of time-space measure}), we conclude that \(\nu_t = \mu_t\).

\medskip

We need to extend the equality \eqref{eq:correspondance semigroupe et pvm} to the elements of $$\mathcal{A}_b= \{ (s,x) \mapsto \mathds{1}_A(s) f(s,x) \mid A \subset \mathbb{R} \text{ finite interval}, f \in C^1_b(\mathbb{R},L^\infty(\mathbb{R}_+,\mathbb{R}))\}.$$ Let $A$ be a finite interval of $\mathbb{R}$ and $f$ in $C^1_b(\mathbb{R},L^\infty(\mathbb{R}_+,\mathbb{R}))$. By Lemma \ref{lemma : approximation de l'indicatrice} in Appendix \ref{sec:appendix approximation and extension results}, there exists a sequence $(c_n)$ in $C^1_b(\mathbb{R},\mathbb{R})$ that converges pointwise to $\mathds{1}_A$. We set $\varphi_n=c_n f$ and $\varphi=\mathds{1}_A f$. We conclude by dominated convergence.

Finally, we can define $\widehat{P}$ for $\varphi$ in $\mathcal{A}$ using Lemma \ref{lemma : extension} in Appendix.
\end{proof}
\begin{remark}
Proposition \ref{prop:def_semigroup_FK} also ensures that the identity 
\begin{equation*}
  \widehat P_t (\mathds{1}_{[0,T]}f)(s,x)=\mathds{1}_{[0,T]}(s+t)\widehat P_t f(s,x)  
\end{equation*}
holds, which will be required later on in order to prove the excessivity of the measure with respect to which the time reversal is performed.
\end{remark}
\subsection{Law of the spinal process}\label{ssec:spine_law}

From this point on, we denote by $m_{s,t}(x)$ the expected population size at time $s+t$ for the process $(\delta_t\otimes\widehat{Z}_t)_t$ started from $\delta_s\otimes\delta_x$, meaning at time $s$ starting from an individual of trait $x$, defined by  
\begin{equation*}
    m_{s,t}(x) = (\widehat{P}_t 1)(s,x)=\mathbb{E}\left[e^{\int_s^{s+t}r^\lambda_u(X_u)du}\Big|X_s=x\right].
\end{equation*}

We denote by $\widehat{V}_T$ the set of individuals alive at time $T$, starting from a single individual with trait $x$ at time $0$. We write $\widehat{X}^{i}_t$ for the historical lineage of individual $i \in \widehat{V}_T$ at time $t$. 

The proofs of the following results follow closely the arguments developed in \cite{meleardB1, meleardB2} while taking into account the time inhomogeneity.
The formula \eqref{eq:semigroupe FK} characterizes the law of $\widehat Z$. It can be extended to the whole trajectory (see \cite{meleardB1, meleardB2, marguet}).
\begin{lemma}
    For every $T>0$, every bounded continuous function $\Phi : \mathbb{D}([0,T],\mathbb{R}) \to \mathbb{R}$, and every $x \in \mathbb{R}_+$, we have
    \begin{equation*}
        \mathbb{E}_{\delta_x}[\langle \widehat{H}_T,\Phi\rangle]
        = \mathbb{E}_{\delta_x}\left[\sum_{i\in\widehat{V}_T}\Phi(\widehat{X}^i_s,\,s\leq T)\right]
        = \mathbb{E}_{x} \left[e^{\int_0^T r^\lambda_u(X_u)\,ds}\,\Phi(X_u,\,u\leq T) \right].
    \end{equation*}
\end{lemma}
Based on the previous representation, one can introduce a new family of probability measures $\mu_x^T$ on $\mathbb{D}([0,T],\mathbb{R})$, defined for every measurable subset $A$ of $\mathbb{D}([0,T],\mathbb{R})$ by
\begin{equation}\label{eq:mesure de la spine forward}
    \mu^T_x(A)
    =\frac{\mathbb{E}_{\delta_x}[\langle\widehat{H}_T,\mathds{1}_A\rangle]}{\mathbb{E}_{\delta_x}[\langle \widehat{H}_T,1\rangle]}
    =\frac{1}{m_{0,T}(x)}\,\mathbb{E}\!\left[\exp\!\left(\int_0^T r^\lambda_u(X_u)\,du\right)\mathds{1}_A(X_u,\,u\leq T)\,\bigg|\,X_0=x\right].
\end{equation}
 This family gives us the law of the stochastic process characterizing the surviving lineage (forward in time) in the population, the so-called \textit{spine}.
\begin{proposition}\label{semigroupe spine}
The law defined by $\mu^{T}_x$ is that of a time-inhomogeneous Markov process $(Y_t)_{[0,T]}$ called the spine. The law of its time-space extension, i.e. $(t,Y_t)_{[0,T]}$, starting from $(0,x)$, is given by the semigroup $\widetilde{P}_t$ defined for any bounded measurable function $\varphi$ by 
\begin{equation}\label{expression semigroupe spine}
    \widetilde{P}_t\varphi(s,x)
    =\frac{\widehat{P}_t(\varphi m_T)(s,x)}{m_T(s,x)}\,
    \mathds{1}_{\{s\geq 0,\,t\geq 0,\,0\leq s+t\leq T\}},
\end{equation}
where
\begin{equation*}
    m_T(s,x)=m_{s,T-s}(x)=\widehat P_{T-s}1(s,x).
\end{equation*}
\end{proposition}

\begin{remark}
Using time-space extensions requires us to treat the initial time $s$ as a standalone variable, just like the initial state $x$. For a time-homogeneous process, the indicator function in equation \eqref{expression semigroupe spine} would be unnecessary, as the law of the process would depend solely on its starting point $x$. Here, however, time inhomogeneity makes the parameter $s$ crucial: the law $\mu^T_x$ governs a process defined strictly on the interval $[0,T]$. Consequently, if the time $s+t$ falls outside this window, both the process and its extension are no longer well-defined. This indicator function is therefore essential, which in turn requires extending the functional spaces associated with the operators to accommodate it.
\end{remark}

\begin{remark}
Recall the expression of $m_T$:
\begin{equation*}
    m_T(s,x)=m_{s,T-s}(x)=\mathbb{E}\!\left[e^{\int_s^T r^\lambda_u(X_u)\,du}\,\big|\,X_s=x\right].
\end{equation*}
For every $(s,x)\in[0,T]\times\mathbb{R}_+$, $m_T(s,x)$ is strictly positive, and the semigroup $\widetilde{P}$ is therefore well-defined. We postpone the proof to Proposition \ref{prop : properties of mT}
\end{remark}
\begin{proof}[Proof of Proposition \ref{semigroupe spine}]
The proof follows the same lines as the homogeneous case (see Proposition 3.3 in \cite{meleardB2}). However, careful attention must be paid to the time-dependence of $m_T$ on $s$ and to the handling of the indicator functions, which are crucial to the derivation.
\end{proof}
We finish this section by establishing some results on the semigroup $\widehat{P}$ and the density of the spinal process.

\begin{lemma}\label{lemme de transition}
For every $\varphi\in \mathcal{A}$, we have for all $s,t\geq0$ such that $s+t\leq T$,
\begin{equation*}
    \langle \delta_s\otimes\xi_s,\widehat{P}_t\varphi\rangle
= \langle \xi_s,\widehat{P}_t\varphi(s,\cdot)\rangle
= \langle \delta_{s+t}\otimes\xi_{s+t},\varphi\rangle
= \langle \xi_{s+t},\varphi(s+t,\cdot)\rangle.
\end{equation*}
In particular, for $s,t\geq 0$, it follows that
\begin{equation*}
    \int_0^\infty  m_{s,t}(x)\,\xi(s,x)\,dx = \langle \xi_{s+t},1\rangle.
\end{equation*}
\end{lemma}
\begin{proof}
It suffices to observe that
\begin{align*}
    \langle \delta_s\otimes\xi_s,\widehat{P}_t\varphi\rangle
    &= \langle \xi_s,\widehat{P}_t\varphi(s,\cdot)\rangle \\
    &= \int_0^\infty  \xi(s,x)\, \mathbb{E}_{\delta_s\otimes\delta_x}[\langle \delta_{s+t}\otimes\widehat{Z}_{s+t},\varphi\rangle]\,dx \\
    &= \mathbb{E}_{\delta_s\otimes\xi_s}[\langle \delta_{s+t}\otimes\widehat{Z}_{s+t},\varphi\rangle].
\end{align*}

Now, the measure given by the last equality satisfies Equation (\ref{Equation E}) with $\delta_s\otimes\xi_s$ as the initial condition, just like $(\delta_{s+t}\otimes \xi_{s+t})_t$. By uniqueness, we have
\[
\langle \delta_s\otimes\xi_s,\widehat{P}_t\varphi\rangle = \langle \delta_{s+t}\otimes\xi_{s+t},\varphi\rangle.
\]

Taking $\varphi=1$, we can write
\begin{align*}
    \int_0^\infty  m_{s,t}(x)\,\xi(s,x)\,dx
    &= \int_0^\infty  \int_0^\infty  \xi(u,x)\, \widehat{P}_t(\varphi)(u,x)\, \delta_s(du)\, dx \\
    &= \langle \xi_s, \widehat{P}_t(\varphi)(s,\cdot) \rangle \\
    &= \langle \xi_{s+t}, 1 \rangle.
\end{align*}
\end{proof}
We thus deduce that $\langle \delta_s\otimes\xi_s, m_T \rangle = \langle \xi_T, 1 \rangle$ and obtain the following proposition.
\begin{proposition}
For every bounded measurable function $f$ on $\mathbb{R}_+$ and $t\leq T$, we have
\begin{equation*}
    \mathbb{E}_{m_{0,T}\xi_0}[f(Y_t)]
    = \int_0^\infty  f(x)\, m_T(t,x)\,\xi(t,x)\,dx.
\end{equation*}
That is, when initialized with $m_{0,T}\xi_0 / \langle \xi_T,1\rangle$, the law of $Y_t$ admits the density
\[
m_T(t,x)\,\xi(t,x)\,\langle \xi_T, 1 \rangle^{-1}\,dx.
\]
\end{proposition}
\begin{proof}
Let $\psi(u,x)=x$, so that $f(Y_t) = f \circ \psi(t,Y_t)$. We have $f\circ\psi\in\mathcal{A}$. Recall also that for $0\leq s\leq T$,  $m_T(s,x)=m_{s,T-s}(x)$. We conclude by Lemma \ref{lemme de transition}.
\end{proof}
\section{Dualities of semigroups}\label{sec:dualities}
To perform the time reversal of the spinal process, we first need to establish duality results for the Feynman–Kac semigroup $\widehat{P}$. Here, the initial time parameter $s$ is allowed to range over $\mathbb{R}$. This does not conflict with the fact that the final time reversal is restricted to $[0,T]$. Indeed, $\widehat{P}$ acts as an auxiliary operator used to construct $\widetilde{P}$ (the spinal semigroup), and it is $\widetilde{P}$ itself that is strictly confined to $[0,T] \times \mathbb{R}_+$ due to the indicator function in \eqref{expression semigroupe spine}.

The ``natural'' function spaces for the semigroups and generators considered throughout this work are $C^1_b(\mathbb{R}, L^\infty(\mathbb{R}_+,\mathbb{R}))$ and $C^1_b(\mathbb{R}, L^1(\mathbb{R}_+,\mathbb{R}))$.
However, in order to avoid boundary terms, we shall initially restrict ourselves to the smaller space $C^1_c(\mathbb{R}, L^\infty(\mathbb{R}_+,\mathbb{R}))$ and $C^1_c(\mathbb{R}, L^1(\mathbb{R}_+,\mathbb{R}))$,
and subsequently extend the results to larger function spaces.

The semigroup associated with the time-reversed spinal process involves the dual of $\widehat{P}$ (with respect to Lebesgue measure). It is therefore necessary to determine this dual object. We begin by characterizing the dual of the (non-Markovian) generator associated with $\widehat{P}$, then we establish the corresponding semigroup duality, and finally extend these results to sufficiently broad classes of test functions.

\subsection{Duality of the Generators on the appropriate spaces}\label{ssec:dual_spaces}
We start by establishing the dual operator of $Q$.
\begin{lemma}\label{dual generator}
For $\varphi\in C^1_c(\mathbb{R},L^\infty(\mathbb{R}_+,\mathbb{R}))\, \psi\in C^1_c(\mathbb{R},L^1(\mathbb{R}_+,\mathbb{R}))$, let $Q^*$ be defined as
\begin{equation*}
    \begin{aligned}
        (Q^*\psi)(s,x) &= -\partial_1\psi(s,x) 
        + \int_0^\infty  \big( \gamma(z)\psi(s,z)m(z,x) -
        \gamma(x)\psi(s,x)m(x,z)\big)\,dz \\
        &\quad + \int_0^\infty  \big( \xi_s(x)h(x,y,\xi_s)\psi(s,y) 
        - \xi_s(y)h(y,x,\xi_s)\psi(s,x) \big)\,dy .
    \end{aligned}
\end{equation*}
We have
\begin{equation*}
    \int_\mathbb{R}\int_0^\infty  Q\varphi(s,x)\psi(s,x)dsdx= \int_\mathbb{R}\int_0^\infty  \varphi(s,x)Q^*\psi(s,x)dsdx.
\end{equation*}
\end{lemma}

\begin{proof}
This follows directly from an integration by parts combined with Fubini’s theorem. 
No boundary term appears since we work with compactly supported functions.
\end{proof}

\begin{remark}
The following identity will be useful later :
\begin{equation}\label{xi as solution}
    r^\lambda(s,x)\xi(s,x) + (Q^*\xi)(s,x) = 0.
\end{equation}
It is a reformulation of \eqref{densité limite}.
\end{remark}
It is crucial to observe that $Q^*$ is not a Markov infinitesimal  generator. Indeed, one can immediately see that $Q^*1(s,x)\neq0$.

We thus decompose $Q^*$ as
\[
(Q^*\psi)(s,x) = Q'\psi(s,x) + \kappa(s,x)\psi(s,x),
\]
where
\begin{equation*}
    \begin{aligned}
        Q'\psi(s,x) 
        &= -\partial_1\psi(s,x) 
        + \int_0^\infty  \big(\psi(s,z)-\psi(s,x)\big)\gamma(z)m(z,x)\,dz \\
        &\quad + \xi_s(x)\int_0^\infty  h(x,y,\xi_s)\big(\psi(s,y)-\psi(s,x)\big)\,dy,
    \end{aligned}
\end{equation*}
and
\begin{equation*}
    \kappa(s,x) = \int_0^\infty \big(\gamma(z)m(z,x)-\gamma(x)m(x,z)\big)\,dz
    + \int_0^\infty \big(\xi_s(x)h(x,y,\xi_s)-\xi_s(y)h(y,x,\xi_s)\big)\,dy.
\end{equation*}

The operator $Q'$ is indeed a Markov infinitesimal generator. The minus before the time derivative indicates that the time runs backward.
The term $\kappa$ corresponds to a creation and killing of mass arising from the potential asymmetry in the trait variation terms. 

We can immediately observe that 
\begin{equation*}
    \overline{\kappa}:=\sup_{(s,x)\in\mathbb{R}\times\mathbb{R}_+}|\kappa(s,x)|<\infty
\end{equation*}
The process associated with the operator $Q$ is not the spinal process itself, 
but rather an intermediate process describing trait variation along a single lineage external to the population.  
In fact, the duality properties of this intermediate process are not central to our analysis.  
Hence, we aim to obtain the dual of $\widehat P$ directly via its generator, which is non-Markovian since the semigroup itself is not.

The dynamics of the semigroup $\widehat P$ are constructed from a process with generator $Q$, 
with the size-dependent term $r^\lambda$ (accounting for birth, death, and competition) incorporated into an exponential functional.  
Following a computation analogous to that in the proof of Proposition~\ref{prop:def_semigroup_FK}, we can then write 
\[
\widehat{P}_t\varphi(s,x)=\varphi(s,x)+\int_0^t \widehat{P}_u\widehat{L}\varphi(s,x)du
\]
which shows that the generator of $\widehat P$ is $\widehat L$, given explicitly by
\begin{equation*}
    \widehat L\varphi(s,x) = Q\varphi(s,x)+r^\lambda(s,x)\varphi(s,x).
\end{equation*}
The computation is given later on in Lemma \ref{formules de kolmogorov}. In summary, the generator of the Feynman--Kac semigroup can be decomposed as the sum of a Markovian generator and a zero-order term, corresponding respectively to trait variation and population size dynamics.  
The dual of $\widehat L$ can then be computed as follows:
\begin{lemma}\label{lemma: dualité des générateurs}
Define, for $\psi\in C^1_c(\mathbb{R}, L^1(\mathbb{R}_+,\mathbb{R}))$,
\[
\widehat{L}'\psi(s,x) = Q'\psi(s,x) + (r^\lambda(s,x) + \kappa(s,x))\psi(s,x).
\]
Then $\widehat{L}$ and $\widehat{L}'$ are in duality with respect to the Lebesgue measure on $\mathbb{R}\times\mathbb{R}_+$ for elements of $C^1_c(\mathbb{R},L^\infty(\mathbb{R}_+,\mathbb{R}))$ and $C^1_c(\mathbb{R},L^1(\mathbb{R}_+,\mathbb{R}))$.
\end{lemma}

\begin{proof}
Duality follows from the self-adjointness of multiplication by $r^\lambda$ and the duality of $Q$ and $Q'+\kappa$.  
\end{proof}
$\widehat{L}'$ can be written in the same form as $\widehat{L}$, with the size-dependent term $r^\lambda$ replaced by $r^\lambda+\kappa$.  
We define the associated Feynman--Kac semigroup for $Q'$ and $r^\lambda+\kappa$ by
\begin{equation}\label{def:aux semigroup fk}
    \widehat{P}'_t\psi(s,x) := \mathbb{E}\Big[ e^{\int_0^t (r^\lambda+\kappa)_{s-u}(X'_{s-u})\,du} \, \psi_{s-t}(X'_{s-t}) |X'_s=x\Big]
\end{equation}
where $(X'_t)$ denotes the process associated with $Q'$, which can be written as
\begin{equation*}
\begin{aligned}
    X'_{T-t} &= X'_T + \int_0^t \int_0^\infty  \int_\mathbb{R} (z - X'_{(T-u)+}) \, \mathds{1}_{\{\theta \leq \gamma(z) m(z,X'_{(T-u)+})\}} \, N_1(du,dz,d\theta) \\
    &\quad + \int_0^t \int_0^\infty  \int_\mathbb{R} (y - X'_{(T-u)+}) \, \mathds{1}_{\{\theta \leq \xi(T-u,X'_{(T-u)+}) h(X'_{(T-u)+},y,\xi_{T-u})\}} \, N_2(du,dy,d\theta),
\end{aligned}
\end{equation*}
where $N_1$ and $N_2$ are independent Poisson measures of intensity $dudsd\theta$ on $\mathbb{R}_+\times\mathbb{R}_+\times\mathbb{R}_+$. In order to establish the duality between $\widehat{P}$ and $\widehat{P}'$, we first ensure that the Kolmogorov equations are satisfied.

\begin{lemma}\label{formules de kolmogorov}
The operators $\widehat{L}$ and $\widehat{L}'$ satisfy
\begin{equation*}
\begin{aligned}
    \forall \varphi \in C^1_c(\mathbb{R}, L^\infty(\mathbb{R}_+,\mathbb{R})), \quad 
    \frac{d}{dt} \widehat{P}_t \varphi(s,x) &= \widehat{P}_t \widehat{L} \varphi(s,x) = \widehat{L} \widehat{P}_t \varphi(s,x), \\
    \forall \psi \in C^1_c(\mathbb{R}, L^1(\mathbb{R}_+,\mathbb{R})), \quad
    \frac{d}{dt} \widehat{P}'_t \psi(s,x) &= \widehat{P}'_t \widehat{L}' \psi(s,x) = \widehat{L}' \widehat{P}'_t \psi(s,x).
\end{aligned}
\end{equation*}
\end{lemma}
\begin{proof}
The core arguments and tools are similar to those in Proposition \ref{prop:def_semigroup_FK}, namely Dynkin's and integration by parts formulas. Using the semigroup property (derived from the Markov property of the processes $(t, X_t)_{t \in [0,T]}$ and $(t, X'_t)_{t \in [0,T]}$), we then obtain the derivative formulas.
\end{proof}

\subsection{Regularity of the Feynman-Kac Semigroups}\label{ssec:dual_extension}

It is necessary to verify that the operator \(\widehat P\) (resp. \(\widehat P'\)) indeed maps $C^1_c(\mathbb{R},L^\infty(\mathbb{R}_+,\mathbb{R}))$ (resp. $C^1_c(\mathbb{R},L^1(\mathbb{R}_+,\mathbb{R}))$) into itself.
To this end, we establish a new alternative representation of our jump processes functionals, that will allow us to obtain all the desired regularity results.
\begin{lemma}[Series representation of jump process functional - $L^\infty$ case]\label{lemma : analytical expression of functional}
Let $(J_t)$ be a jump process on $\mathbb{R}\times\mathbb{R}_+$, with its generator given for $\varphi\in C^1_b(\mathbb{R}, L^\infty(\mathbb{R}_+,\mathbb{R}))$ by
\[
Q\varphi(s,x) = \partial_1 \varphi(s,x) 
+ \Lambda(s,x)\int_0^\infty  \big(\varphi(s,z)-\varphi(s,x)\big)\nu(s,x,z)\,dz,
\]

where $\Lambda$ is measurable, bounded by a constant $\overline{\Lambda}$, and $\nu(s,x,\cdot)$ is a probability density on $\mathbb{R}_+$.  
Fix $t\geq0$ and define
\begin{equation}\label{def:evaluation of functional jump process}
    f(t,s,x) := \mathbb{E}\Bigg[e^{\int_0^t A_{s+u}(J_{s+ u})du}\,\varphi_{s+t}(J_{s+ t})\ \Big|\ J_s=x\Bigg]
\end{equation}
with $A$ measurable and bounded from above by a constant $\overline{A}$. 
The function $f$ is bounded in $t$ and $x$.

Define the recurrence operator: for bounded measurable $\varphi$ on $\mathbb{R}_+$, 
\[
(\mathcal K_{s,u}\varphi)(x):=
e^{\int_0^u A_{s+v}(x)dv}\Bigg[
\frac{\Lambda(s+u,x)}{\overline{\Lambda}}\int_0^\infty \varphi(z_1)\nu(s+u,x,z_1)\,dz_1
+\Big(1-\frac{\Lambda(s+u,x)}{\overline{\Lambda}}\Big)\varphi(x)
\Bigg].
\]
Then we have the following expression
\begin{equation}\label{serie}
    f(t,s,x)=e^{-\overline{\Lambda}t}\sum_{k\geq0} \tfrac{(\overline{\Lambda}t)^k}{k!}\varphi_k(t,s,x)
\end{equation}
where $\varphi_0(t,s,x)=e^{\int_0^t A_{s+v}(x)dv}\varphi(s+t,x)$ and 
\begin{equation}\label{eq:kth term}
\varphi_k(t,s,x)
=\int_0^t \frac{k(t-u)^{k-1}}{t^k}\,
(\mathcal K_{s,u}\varphi_{k-1}(t-u,s+u,\cdot))(x)\,du.
\end{equation}

\end{lemma}
\begin{remark}
The case where the process runs backward in time has a similar series representation (as long as we do not forget the minus sign in front of the time derivative).
\end{remark}
\begin{proof}
We couple the process $(J)$ with a homogeneous Poisson process $(N)$ of intensity $\overline{\Lambda}$, where each jump time of $(N)$ serves as a potential jump for $(J)$. Using the law of total expectation, we partition the expectation based on the number of potential jumps in the interval $[0,t]$, resulting in a series of conditional expectations. The joint distribution of the ordered potential jump times is known and given by Theorem 2.15 of \cite{durrettprocesses}. By conditioning each term of the series on whether the first potential jump is accepted or rejected, we derive a recurrence relation: every term in the series, except for the first one, can be expressed in terms of the operator $\mathcal{K}_{s,u}$ and the preceding term. A more detailed version of this proof can be found in Appendix \ref{appendix:detailed_proof_series}.
\end{proof}
We now state various regularity results.
\begin{corollary}[Continuity criteria in $s$ - $L^\infty$ case]\label{corollary:series continuity Linfini}

With the same assumptions as in Lemma \ref{lemma : analytical expression of functional} and assuming
\begin{itemize}
    \item $s\mapsto\Lambda(s,x)$ is continuous and $x\mapsto\Lambda(s,x)$ measurable in $x$,
    \item $s\mapsto\nu(s,x,z)$ is continuous and $(x,z)\mapsto\nu(s,x,z)$ measurable,
    \item $A$ is continuous in $s$,
    \item $\varphi$ is continuous in $s$.
\end{itemize}
Then the function $f$ defined in \eqref{def:evaluation of functional jump process} is continuous in $s$.

Assume in addition:
\begin{itemize}
    \item $\Lambda$ is $C^1$ in $s$, and both $\Lambda$ and $\partial_s\Lambda$ are bounded,
    \item $\nu$ is continuous and there exists $h\in L^1(\mathbb{R}_+)$ such that 
    \[
    \forall (s,x,z)\in\mathbb{R}\times\mathbb{R}_+\times\mathbb{R}_+, \quad|\Lambda(s,x)\nu(s,x,z)|\le h(z),\quad |\partial_s(\Lambda(s,x)\nu(s,x,z))|\le h(z),   
    \]
 
    \item $A$ and $\varphi$ are $C^1$, with 
    \[
    \forall (s,x)\in\mathbb{R}\times\mathbb{R}_+,\quad
    |A(s,x)|+|\partial_sA(s,x)|+|\varphi(s,x)|+|\partial_s\varphi(s,x)|\le M
    \]

\end{itemize}

Then for all $t\ge0$ and $x\in\mathbb{R}_+$, the map $s\mapsto f(t,s,x)$ is $C^1$.  

If moreover $\varphi$ has compact support in the first coordinate $s$, then so does $f$ in $s$.
\end{corollary}

\begin{proof}
Using the series representation from Lemma \ref{lemma : analytical expression of functional}, it suffices to verify that $\varphi_0$ is continuous and that the recurrence operator $\mathcal{K}$ preserves continuity. For each $k$, we have $\|\varphi_k\|_\infty \le e^{\overline{A} t}\|\varphi\|_\infty$; thus, by uniform convergence of the defining series, $f(t,s,x)$ is continuous in $s$.

For the second part of the corollary, we verify that $\varphi_0$ is $C^1$ in $s$ and uniformly bounded. By our hypotheses, there exist constants $c_1(t)$ and $c_2(t)$, independent of $k$, such that
$$\|\partial_s\varphi_k\|_\infty \leq c_1(t)\|\varphi_{k-1}\|_\infty + c_2(t)\|\partial_s\varphi_{k-1}\|_\infty.$$
By induction, the series of derivatives converges uniformly; thus, $f(t,s,x)$ is $C^1$ in $s$, and its derivative corresponds to the sum of the series of derivatives. Furthermore, if $\varphi$ has compact support in its first coordinate $s$, this property is preserved by the recursive construction, as well as by its derivatives.
\end{proof}
Up to this point, the series representation introduced in Lemma \ref{lemma : analytical expression of functional} relied on uniform boundedness assumptions to ensure the well-posedness of the conditional expectations. We now extend these results to the space $L^1(\mathbb{R}_+)$.

\begin{lemma}\label{lemma : series representation L^1}
With the same notations and assumptions on $\Lambda$ and $A$ as in Lemma \ref{lemma : analytical expression of functional} and assuming 
\begin{itemize}
	\item $\varphi(s,\cdot)$ belongs to $L^1(\mathbb{R}_+)$
    \item $\sup_{s\in\mathbb{R},\,z\in\mathbb{R}_+}\int_0^\infty \Lambda(s,x)\nu(s,x,z)dx=C_\Lambda<\infty$.
\end{itemize}

Then the recurrence operator $\mathcal{K}_{s,u}$ is a bounded operator on $L^1(\mathbb{R}_+)$ and the series representation \eqref{serie} holds almost everywhere in $L^1(\mathbb{R}_+)$.
\end{lemma}

\begin{proof}
Let $\psi\in L^1(\mathbb{R}_+)$.
Under our hypotheses and given the form of $\mathcal{K}$, we have the following estimate:
$$\|\mathcal{K}_{s,u}\psi\|_1 \leq e^{\overline{A}u}\left(1+\frac{C_\Lambda}{\overline{\Lambda}}\right)\|\psi\|_1.$$
By induction, taking into account the integration over $u$ in equation \eqref{eq:kth term}, there exists a constant $M > 0$ such that $\|\varphi_k(t,s,\cdot)\|_1 \leq M^k\|\varphi(s+t,\cdot)\|_1$. The series converges normally in the Banach space $L^1(\mathbb{R}_+)$, making the equality rigorously true in $L^1(\mathbb{R}_+)$.
\end{proof}
\begin{corollary}\label{corollary:series continuity L1}
With the same notations and assumptions as Lemma \ref{lemma : series representation L^1} and assuming
\begin{itemize}
\item $\varphi$ belongs to $C^1(\mathbb{R}, L^1(\mathbb{R}_+,\mathbb{R}))$,
\item $\Lambda$ and $\nu$ are differentiable in $s$ with $\sup_{s\in\mathbb{R},z\in\mathbb{R}_+} \int_0^\infty |\partial_s(\Lambda(s,x)\nu(s,x,z))|dx<\infty$,
\item $A$ is differentiable in $s$ and $\sup_{(s,x)\in\mathbb{R}\times\mathbb{R}_+} |\partial_s A(s,x)| < \infty$.
\end{itemize}
Then the function $(s,x)\mapsto f(t,s,x)$ belongs to $C^1(\mathbb{R}, L^1(\mathbb{R}_+))$. If $\varphi$ is compactly supported in $s$, so is $f$.
\end{corollary}
\begin{proof}
The result follows by applying the exact same inductive procedure as in the proof of Corollary \ref{corollary:series continuity Linfini}, but within the Banach space $L^1(\mathbb{R}_+)$. 

\end{proof}
We can deduce the next property of the operators $\widehat{P}$ and $\widehat{P}'$.
\begin{proposition}\label{prop:stabilité C1c des semigroupes}
Under Assumptions \textbf{(H)} and those of Propositions \ref{prop:densité limite} and \ref{prop:produit_par_x_fini}, we have
\[
\forall t\ge0,\ \forall \varphi \in C^1_c(\mathbb{R},L^\infty(\mathbb{R}_+,\mathbb{R})),\quad
\widehat{P}_t\varphi\, \in C^1_c(\mathbb{R},L^\infty(\mathbb{R}_+,\mathbb{R})),
\]
\[
\forall t\ge0,\ \forall \psi \in C^1_c(\mathbb{R},L^1(\mathbb{R}_+,\mathbb{R})),\quad
\widehat{P}'_t\psi\, \in C^1_c(\mathbb{R},L^1(\mathbb{R}_+,\mathbb{R})).
\]
\end{proposition}

\begin{proof}
We recall that 
$\lambda_t=\langle \xi_t,1\rangle$ and we use the same notation as in Lemma \ref{lemma : analytical expression of functional}.  For the operator $\widehat{P}$ defined in \eqref{eq:semigroupe FK}, we have 
\begin{equation*}
\Lambda(s,x)=\gamma(x)+\frac{\tau}{\beta+\mu\lambda_s}\int_x^\infty \xi(s,z)\,dz,
\end{equation*}
\begin{equation*}
\nu(s,x,z)=\Lambda(s,x)^{-1}\Big(\gamma(x)m(x,z)+\frac{\tau}{\beta+\mu\lambda_s}\xi(s,z)\,\mathds{1}_{z>x}\Big).
\end{equation*}
\begin{equation*}
    A(s,x)=r^\lambda(s,x)
\end{equation*}
For the operator $\widehat{P}'$ defined in \eqref{def:aux semigroup fk}, 
\begin{equation*}
\Lambda'(s,x)=\int_0^\infty  \gamma(z)m(z,x)\,dz+\tau\frac{x\xi(s,x)}{\beta+\mu\lambda_s},
\end{equation*}
\begin{equation*}
\nu'(s,x,z)=\Lambda'(s,x)^{-1}\Big(\gamma(z)m(z,x)+\xi(s,x)\frac{\tau}{\beta+\mu\lambda_s}\,\mathds{1}_{x>z}\Big),
\end{equation*}
\begin{equation*}
    A'(s,x)=r^\lambda(s,x)+\kappa(s,x).
\end{equation*}

By the properties of $\xi$ given in Proposition \ref{prop:densité limite}, Proposition \ref{prop:produit_par_x_fini} and Assumption \textbf{(H)}, both semigroups satisfy the hypotheses of Corollaries \ref{corollary:series continuity Linfini} and \ref{corollary:series continuity L1}.

\end{proof}
\subsubsection{Duality of the Feynman-Kac Semigroups}
We recall the definition of the following set :
\begin{equation*}
    \mathcal{A} = \{ (s,x) \mapsto \mathds{1}_A(s) f(s,x) \mid A \subset \mathbb{R} \text{ finite interval}, f \in C^1(\mathbb{R},L^\infty(\mathbb{R}_+,\mathbb{R}))\}.
\end{equation*}

We define its counterpart :
\begin{equation*}
    \mathcal{A}' = \{ (s,x) \mapsto \mathds{1}_A(s) g(s,x) \mid A \subset \mathbb{R} \text{ finite interval}, g \in C^1(\mathbb{R},L^1(\mathbb{R}_+,\mathbb{R}))\}.
\end{equation*}
We can state the main duality result:
\begin{proposition}\label{prop: dualité principale}
The semigroups \(\widehat{P}\) and \(\widehat{P}'\) are in duality with respect to the Lebesgue measure on \(\mathbb{R}\times\mathbb{R}_+\):
\begin{equation*}
    \langle \widehat{P}_t \varphi,\psi\rangle 
    = \langle \varphi,\widehat{P}'_t \psi\rangle, 
    \quad \forall t\geq 0,\; \forall \varphi\in \mathcal{A},\,\forall\psi\in\mathcal{A}'.
\end{equation*}
\end{proposition}

\begin{proof}
We begin by establishing the duality for test functions \(\varphi\in C^1_c(\mathbb{R},L^\infty(\mathbb{R}_+,\mathbb{R}))\) and \(\psi\in C^1_c(\mathbb{R},L^1(\mathbb{R}_+,\mathbb{R}))\).  
We define the following quantity, for $u\in[0,t]$ :
\begin{equation*}
    F_t(u)=\int_\mathbb{R}\int_0^\infty \widehat{P}_u\varphi(s,x)\widehat{P}'_{t-u}\psi(s,x)dsdx.
\end{equation*}
It is well-defined by Proposition \ref{prop:stabilité C1c des semigroupes}. Let us observe that for all $u$, the function $(s,x)\mapsto\widehat{P}_u\varphi(s,x)\widehat{P}'_{t-u}\psi(s,x)$ is integrable. Also, for all $(s,x)$, the function $t\mapsto\widehat{P}_u\varphi(s,x)\widehat{P}'_{t-u}\psi(s,x)$ is differentiable (by Lemma \ref{formules de kolmogorov}) and we have the following domination :
\begin{equation*}
    \left|\partial_u\left(\widehat{P}_u\varphi(s,x)\widehat{P}'_{t-u}\psi(s,x)\right)\right|\leq e^{(2\overline{b}+\overline{\kappa})u}|\widehat{L}\varphi(s,x)\widehat{L}'\psi(s,x)|.
\end{equation*}
We can check that the function $\widehat{L}\varphi\widehat{L}'\psi$ belongs to $C_c(\mathbb{R},L^1(\mathbb{R}_+,\mathbb{R}))$.
By applying the classical differentiation under the integral sign, we have
\begin{equation*}
    \partial_u F_t(u)=\int_\mathbb{R}\int_0^\infty  \widehat{L}\widehat{P}_u\varphi(s,x) \widehat{P}'_{t-u}\psi(s,x)dsdx-\int_\mathbb{R}\int_0^\infty \widehat{P}_u\varphi(s,x) \widehat{L}'\widehat{P}'_{t-u}\psi(s,x)dsdx.
\end{equation*}
We know by Proposition \ref{prop:stabilité C1c des semigroupes} that the function $(s,x)\mapsto \widehat{P}_u\varphi(s,x)$ (respectively $(s,x)\mapsto \widehat{P}'_{t-u}\psi(s,x)$) belongs to $C^1_c(\mathbb{R},L^\infty(\mathbb{R}_+,\mathbb{R}))$ (resp. $C^1_c(\mathbb{R},L^1(\mathbb{R}_+,\mathbb{R}))$). Therefore we can apply the duality result given by Lemma \ref{lemma: dualité des générateurs} and write 
\begin{equation*}
    \int_\mathbb{R}\int_0^\infty \widehat{P}_u\varphi(s,x)\cdot \widehat{L}'\widehat{P}'_{t-u}\psi(s,x)dsdx=\int_\mathbb{R}\int_0^\infty  \widehat{L}\widehat{P}_u\varphi(s,x) \cdot\widehat{P}'_{t-u}\psi(s,x)dsdx.
\end{equation*}
It implies $\partial_uF_t(u)=0$ and in particular, taking $u=0$ and $u=t$, that
\begin{equation*}
    \langle \widehat{P}_t\varphi,\psi\rangle=\langle \varphi,\widehat{P}'_t\psi\rangle.
\end{equation*}

\medskip

We extend the duality to $\varphi\in\mathcal{A}$ and $\psi\in\mathcal{A}'$. By definition, there exist finite intervals $A,B$ and $f,g$ respectively elements of $C^1_b(\mathbb{R},L^\infty(\mathbb{R}_+,\mathbb{R}))$ and $C^1_b(\mathbb{R},L^1(\mathbb{R}_+,\mathbb{R}))$ such that $\varphi=\mathds{1}_Af$ and $\psi=\mathds{1}_B g$. We know by classical theory that $C^\infty_c(\mathbb{R})$ is dense in $L^1(\mathbb{R})$ (for the $\|.\|_1$ norm). Hence $\mathds{1}_A$ and $\mathds{1}_B$ can be approximated by sequences $(a_n)$ and $(b_n)$ of $C^\infty_c(\mathbb{R})$. We conclude by dominated convergence and Lemma \ref{lemma : extension} in Appendix \ref{sec:appendix approximation and extension results}.
\end{proof}

\section{Time reversal and return to the original population process}\label{sec:time_reversal}

Time reversal for the spinal process is achieved via Nagasawa's Theorem, which characterizes the backward semigroup through its duality with the forward semigroup.
In order to perform the duality between the semigroups, we ensure that the function \(m_T\) is sufficiently regular.  
Then, we conclude by expressing the generator of the time-reversed spine. Finally, we show that the backward trajectory of a uniformly sampled individual converges to the reversed spinal process in the large population limit.

In this work, the approach we adopt allows us to address the freezing of population interactions under a non-stationary density $\xi$. This framework, along with the resulting time inhomogeneity of the time-reversed spinal process, stands as a natural extension of the work in \cite{meleardB1, meleardB2}. The methodological advantage of our contribution is that we formulate the entire study directly using time-space extended processes. This formalism provides a clear and consistent framework, which allows us to precisely specify the required functional spaces, especially for handling the indicator functions that restrict the spinal processes to the interval $[0,T]$.
\subsection{Time reversal of the spinal process}\label{ssec:time reversal of the spinal process}
Let us denote by \(\widetilde{P}^\mathcal{R}\) the semigroup of the time-reversed spine.

We shall define the semigroup $\widetilde{P}^\mathcal{R}$ as the dual of $\widetilde{P}$ with respect to the measure $\eta$ given in Lemma \ref{mesure excessive} (in Appendix \ref{sec:appendix time reversal}) and define the notation $\langle\cdot,\cdot\rangle_\eta$ for the scalar product in $L^2(\eta)$. Recall the following expression:
\begin{equation*}\label{semigroup expression}
    \forall t\geq0,\,\widetilde{P}_t\varphi(s,x)
    =\frac{\widehat{P}_t(\varphi m_T \mathds{1}_{[t,T]})(s,x)}{m_T(s,x)}\,.
\end{equation*}
We have $m_T\in C^1([0,T], L^\infty(\mathbb{R}_+,\mathbb{R}))$. The semigroup $\widetilde{P}$ is well-defined if $\varphi m_T \mathds{1}_{[t,T]}\in \mathcal{A}$. This is the case if $\varphi\in C^1([0,T], L^\infty(\mathbb{R}_+,\mathbb{R}))$ by Lemma \ref{lemma : extension}.

Recall that $m_T$ is defined on $[0,T]\times\mathbb{R}_+$ by 
\[
m_T(s,x) = \widehat{P}_{T-s} 1(s,x).
\]
We prove that $m_T$ is sufficiently regular.
\begin{proposition}\label{prop : properties of mT}
The function $(s,x)\in[0,T]\times\mathbb{R}_+\mapsto m_T(s,x)$ is $C^1$ in time, measurable bounded in space and strictly positive.
\end{proposition}
\begin{proof}
First, observe that for all $x \in \mathbb{R}_+$ and $0 \leq s \leq T$, one has $\sup_{x \in \mathbb{R}_+,\,0 \leq s \leq T} m_T(s,x) \leq e^{\overline{b}T}$.

We now prove the continuity of $t \mapsto \widehat{P}_t 1(s,x)$.  
By Kolmogorov’s formula in Lemma \ref{formules de kolmogorov}, we have for all $s$, $\widehat{P}_t 1(s,x) = 1 + \int_0^t \widehat{P}_u r^\lambda(s,x) \, du$. Since the constant function $1$ belongs to the domain of $\widehat{L}$ and satisfies $\widehat{L}1(s,x)=r^\lambda(s,x)$ we get
\[
|\widehat{P}_t1(s,x) - 1| \leq \overline{b} \int_0^t e^{\overline{b}u}du 
\leq \overline{b} e^{\overline{b}t}t,
\]
which proves the continuity of $f$ at $0$ (with respect to the first coordinate). The semigroup property then implies continuity for all $t \geq 0$. Because $\partial_t\widehat{P}_t 1(s,x)= \widehat{P}_t r^\lambda(s,x)$, we have the $C^1$ regularity in $t$.

$C^1$ regularity of $\widehat{P}_t 1(s,x)$ with respect to $s$ follows from Lemma \ref{corollary:series continuity Linfini}. We can write $m_T(s,x)$ as a series of strictly positive terms and then the strict positivity of $m_T$ follows.
\end{proof}
 We need to recall the definition of the linear map $\mathcal{R}:\mathbb{D}([0,T],\mathbb{R})\to\mathbb{D}([0,T],\mathbb{R})$ given by
        \begin{equation*}\label{def:Rcal}
		\mathcal{R}(\varphi)(s)=\begin{cases}
            \lim\limits_{\varepsilon\to 0,\ \varepsilon>0}\varphi(T-s-\varepsilon) &\text{ if } s\neq T\\
            \varphi(0)&\text{ if } s=T
        \end{cases}
        \end{equation*}
		
As proven in \cite{meleardB2}, the linear map $\mathcal{R}$ is $1$-Lipschitz continuous with respect to the Skorokhod topology.

Now we can prove the main result:
\begin{proposition}\label{time-reversed semigroup}
Let $\mathcal{R}(Y)(t)$ denote the time reversal of the spinal process.
The process initiated with the law $\xi_T/\langle \xi_T,1\rangle$ (note that for all $x \in \mathbb{R}_+$, $m_T(T,x)=1$), has its semigroup (extended to time-space) given, for every $\psi$ in $C^1([0,T], L^1(\mathbb{R}_+,\mathbb{R}))\cup C^1([0,T], L^\infty(\mathbb{R}_+,\mathbb{R}))$, by
\begin{equation*}
    \widetilde{P}^\mathcal{R}_t\psi(s,x)=\frac{1}{\xi(s,x)}\widehat P'_t(\xi\psi)(s,x)\mathds{1}_{\{0\leq t\leq s\leq T\}}.
\end{equation*}
Since under our assumptions $\xi_t$ never vanishes for any $t$, the expression is well defined.

The generator of the time-reversed spine is given, for every $\psi\in D(\widetilde{L}^\mathcal{R})$, by
\begin{equation*}\label{backward generator}
    \begin{aligned}
        \widetilde{L}^\mathcal{R}\psi(s,x)=
        -\partial_1\psi(s,x)&+\int_0^\infty (\psi(s,z)-\psi(s,x))\frac{\xi(s,z)}{\xi(s,x)}\gamma(z)m(z,x)dz\\
        &+\tau\int_0^\infty (\psi(s,y)-\psi(s,x))\frac{\xi(s,y)}{\beta+\mu\langle \xi_s,1\rangle}\mathds{1}_{\{y<x\}}dy.
    \end{aligned}
\end{equation*}
\end{proposition}

\begin{proof}[Proof of Proposition \ref{time-reversed semigroup}]
We have introduced the reference measure $\eta$ in Lemma \ref{mesure excessive} in Appendix \ref{sec:appendix time reversal}:
\[
\eta(ds,dx)=\xi(s,x)m_T(s,x)\mathds{1}_{\{0\leq s\leq T\}}dsdx.
\]

Let $\varphi\in C^1([0,T],L^\infty(\mathbb{R}_+,\mathbb{R}))$, we have seen that $\varphi m_T \mathds{1}_{[t,T]}$ belongs to $\mathcal{A}$.
Let $\psi\in C^1([0,T], L^1(\mathbb{R}_+,\mathbb{R}))\cup C^1([0,T], L^\infty(\mathbb{R}_+,\mathbb{R}))$.

We compute
\begin{align*}
    \langle \widetilde{P}_t\varphi,\psi\rangle_\eta
    &=\int_0^\infty \int_\mathbb{R} \psi(s,x)
    (\widetilde{P}_t\varphi)(s,x)\xi(s,x)m_T(s,x){1}_{\{0\leq s\leq T\}}dsdx\\
    &=\int_0^\infty \int_\mathbb{R}\psi(s,x)\widehat P_t(\varphi m_T)(s,x)\xi(s,x)\mathds{1}_{\{s\geq 0,\,t\geq 0,\,0\leq s+t\leq T\}}\mathds{1}_{\{0\leq s\leq T\}}dsdx\\
    &=\mathds{1}_{\{0\leq t\leq T\}}\int_0^\infty \int_\mathbb{R}\widehat P_t (\varphi m_T)(s,x)\psi(s,x)\xi(s,x)\mathds{1}_{\{0\leq s+t\leq T\}}\mathds{1}_{\{0\leq s\leq T\}}dsdx.
\end{align*}

Define $g(s)=\mathds{1}_{\{0\leq s\leq T\}}\mathds{1}_{\{0\leq s+t\leq T\}}$.
For a test function $\phi\in C^1(\mathbb{R}, L^1(\mathbb{R}_+,\mathbb{R}))$, we have
\begin{align*}
    \widehat P'_t(\phi g)(s,x)&=\mathbb{E}_{(s,x)}\left[
    e^{\int_0^t r^\lambda_{s-u}(X_{s-u})du}\phi_{s-t}(X_{s-t})g(s-t)\right]\\
    &=g(s-t)\widehat P'_t\phi(s,x).
\end{align*}

We know that $\xi\in C^1([0,T], L^\infty(\mathbb{R}_+,\mathbb{R}))\cap C^1([0,T], L^1(\mathbb{R}_+,\mathbb{R}))$ since $\psi$ being an element of $C^1([0,T], L^1(\mathbb{R}_+,\mathbb{R}))\cup C^1([0,T], L^\infty(\mathbb{R}_+,\mathbb{R}))$, we have $\psi\xi\mathds{1}_{[0,T-t]}\in \mathcal{A}'$. Thus we can use the duality given by Proposition \ref{prop: dualité principale}:

\begin{align*}
    \langle \widetilde{P}_t\varphi,\psi\rangle_\eta
    &=\mathds{1}_{\{0\leq t\leq T\}}\int_0^\infty \int_\mathbb{R}\widehat P_t (\varphi m_T) (s,x)\psi(s,x)\xi(s,x)g(s)dsdx\\
    &=\mathds{1}_{\{0\leq t\leq T\}}
    \int_0^\infty \int_\mathbb{R}
    \varphi(s,x)m_T(s,x)g(s-t)\widehat P'_t(\psi\xi)(s,x)dsdx\\
    &=\int_0^\infty \int_\mathbb{R}
    \varphi(s,x)\cdot \frac{1}{\xi(s,x)}\widehat P'_t(\psi\xi)(s,x)
    \mathds{1}_{\{0\leq t\leq s\leq T\}}\cdot \eta(ds,dx)\\
    &=\langle\varphi,\widetilde{P}^\mathcal{R}_t\psi\rangle_\eta.
\end{align*}

The duality is thus established by applying Theorem \ref{Dellacherie}.

For the computation of the generator, we observe that (for $0\leq t\leq s\leq T$)
\[
\widetilde{P}^\mathcal{R}_t\psi(s,x)=\psi(s,x)+\frac{1}{\xi(s,x)}\int_0^t \widehat P'_u \widehat L' (\xi \psi)(s,x)du.
\]
By definition of the infinitesimal generator, we have
\[
\widetilde{L}^\mathcal{R}\psi(s,x)=\frac{1}{\xi(s,x)}\widehat L'(\xi \psi)(s,x),
\]
which rewrites as
\begin{align*}
    \widetilde{L}^\mathcal{R}\psi(s,x)=
        -\partial_1\psi(s,x)&+\int_0^\infty (\psi(s,z)-\psi(s,x))\frac{\xi(s,z)}{\xi(s,x)}\gamma(z)m(z,x)dz\\
        &+\int_0^\infty (\psi(s,y)-\psi(s,x))\xi(s,y)h(x,y,\xi_s)dy+H(s,x),
\end{align*}
where 
\begin{align*}
    H(s,x)=\frac{\psi(s,x)}{\xi(s,x)}\left( Q'\xi(s,x)+(r^\lambda+\kappa)(s,x)\xi(s,x) \right)=0
\end{align*}
since the term in parentheses vanishes by \eqref{xi as solution}.
\end{proof}

\subsection{Return to the original population process}\label{ssec:return_original}
Let us fix a time horizon $T>0$. We investigate the ancestral lineage of an individual sampled from the population at this time $T$. Let $V^K_T$ be the set of individuals alive at time $T$ and $U^K_T$ a uniform random variable on $V^K_T$. Given the historical process $H^K_T$, we define the empirical spinal process $Y^K$ by setting
\[
Y^{K}_{t}=X^{U^{K}_{T}}_{t},\quad \forall t\in[0,T],
\]
using the framework from \cite{meleardB2}. For any continuous and bounded test function $\Phi:\mathbb{D}([0,T],\mathbb{R})\to\mathbb{R}$, the law of $Y^K$ satisfies the relation
\begin{equation*}
	\label{eq:lawY}
	\mathbb{E}_{x}\left[\Phi\left(Y^{K}_{t},\ t\in[0,T]\right)\right]=
	\mathbb{E}_{\delta_{x}}\left[\cfrac{\langle H^{K}_{T},\Phi\rangle}{\langle H^{K}_{T},1\rangle}\right].
\end{equation*}

\begin{proposition}\label{prop:spine forward à la limite}
Assume that the initial sequence of measures $(Z^{K}_0)_{K}$ converges in probability and weakly to a deterministic finite measure $\xi_0(x)dx$. For a given $T>0$, we have
\begin{equation*}\label{eq:hist_tilde}
    \lim_{K\rightarrow +\infty} \mathbb{E}_{Z^K_0}\left[\Phi\big(Y^K_s,\ s\leq T\big)\right]= \int_0^\infty  \langle \mu^T_{x} , \Phi\rangle \ \frac{m_T(0,x)\xi_0(x)}{\langle\xi_T,1\rangle} \ dx.
\end{equation*}
Consequently, the typical lineage $Y^{K}$ behaves asymptotically as the spinal process $Y$ launched from the biased initial distribution $\frac{\xi_0(x)m_T(0,x)}{\langle\xi_T,1\rangle} dx$.
\end{proposition}

\begin{proof}
We omit the full derivation here, as it closely follows the proof of Propositions 4.1 and 4.6 in \cite{meleardB1}. Briefly, we use Theorem \ref{cv historique} to move from $H^K$ to $\widehat{H}^K$ and follow the proof of Proposition 4.6 in \cite{meleardB1}. We conclude by using \eqref{eq:mesure de la spine forward} to identify the asymptotic distribution.
\end{proof}

We have all the necessary tools to state the main result. Indeed, as $K \to \infty$, we know the limit distribution of a forward lineage of individuals by virtue of Proposition \ref{prop:spine forward à la limite}, and we know the distribution of the time-reversed version of this distribution by virtue of Proposition \ref{time-reversed semigroup}.
\begin{theorem}\label{thm:conclusion}
Suppose that Assumptions (H) and \eqref{eq:hypothèses du modèle sur les moments} hold, and that $(Z^K_0)_{K}$ converges in probability and weakly to $\xi_0(x)dx$. Then, for any bounded measurable function $\Phi$, we obtain
\begin{equation*}\label{eq:th_concl}
    \lim\limits_{K\to \infty}\mathbb{E}_{ Z^K_0}\left[\frac{\langle H^K_T,\Phi\circ \mathcal{R}\rangle}{\langle H^K_T,1\rangle}\right]=\mathbb{E}_{\xi_T}\left[\Phi\left(\mathcal{R}(Y)(s), \,s\in[0,T]\right)\right],
\end{equation*}
where $\mathcal{R}(Y)$ is the Markov process associated with the semigroup
\begin{equation*}\label{def:PR}
    \widetilde{P}^\mathcal{R}_t\psi(s,x)=\frac{1}{\xi(s,x)}\widetilde{P}'_t(\xi \psi)(s,x)\mathds{1}_{\{0\leq t\leq s\leq T\}}.
\end{equation*}

\end{theorem}
Equivalently, Theorem \ref{thm:conclusion} guarantees that in the macroscopic limit, ancestral lineages of present-day individuals behave as the time-reversed spinal process.
\subsection{Simulations and implications of the model}\label{ssec:simulations}

Our work addresses the following evolutionary trade-offs: do present-day individuals predominantly descend from lineages that maintained a low cost by limiting their trait level, or from ancestors carrying an excessive burden? Examining these backward trajectories provides a retrospective view of the dynamics driving population evolution.

To address these questions, in this section, we illustrate our theoretical results by simulating the random ancestral lineage process, whose dynamics are governed by the backward generator \eqref{generator_backward}. We use the same baseline parameters as in Section \ref{ssec:numerical}, with an initial density $\xi_0$ taken as a truncated Gaussian centered at $0$, and a transfer rate $\tau=3$. To highlight the impact of the variation kernel on the genealogy, we investigate two distinct frameworks: Gaussian and Cauchy variations. For both cases, we explore the behaviour of the spine over a stationary background density $\xi$ (obtained via numerical resolution) as well as over the non-stationary solution $\xi_t$ of \eqref{densité limite}.
\newline

\textbf{Case 1 - small-amplitude variations: } We first consider the case of Gaussian (reflected) variations. In Figure \ref{fig:spine_comparison_gaussian}, we present a simulation of the backward spinal process. On the previously displayed representations of the density $\xi$, we superimpose sample trajectories (in white) of the random process. The top panels display the dynamics when $\xi$ is a stationary state, whereas the bottom panels illustrate the non-stationary regime.

A striking observation in both cases is that the backward spine, once it reaches a trait close to $0$, stays strictly in this region until the initial time. Crucially, it does not jump into the regions where the macroscopic density is concentrated (such as the main branches driven by transfer dynamics). This clearly illustrates our earlier population-level findings: although a high transfer rate pushes most of the population toward an excessive burden and high cost, this heavily loaded majority does not form the evolutionary backbone of surviving lineages. In other words, present-day individuals descend almost exclusively from a very small subgroup of ancestors carrying low, highly adapted traits.
\begin{figure}[h!]
        \centering
        
        \setlength{\tabcolsep}{2pt}
        \begin{tabular}{cc}
            \includegraphics[width=0.45\textwidth, height=0.35\textheight, keepaspectratio]{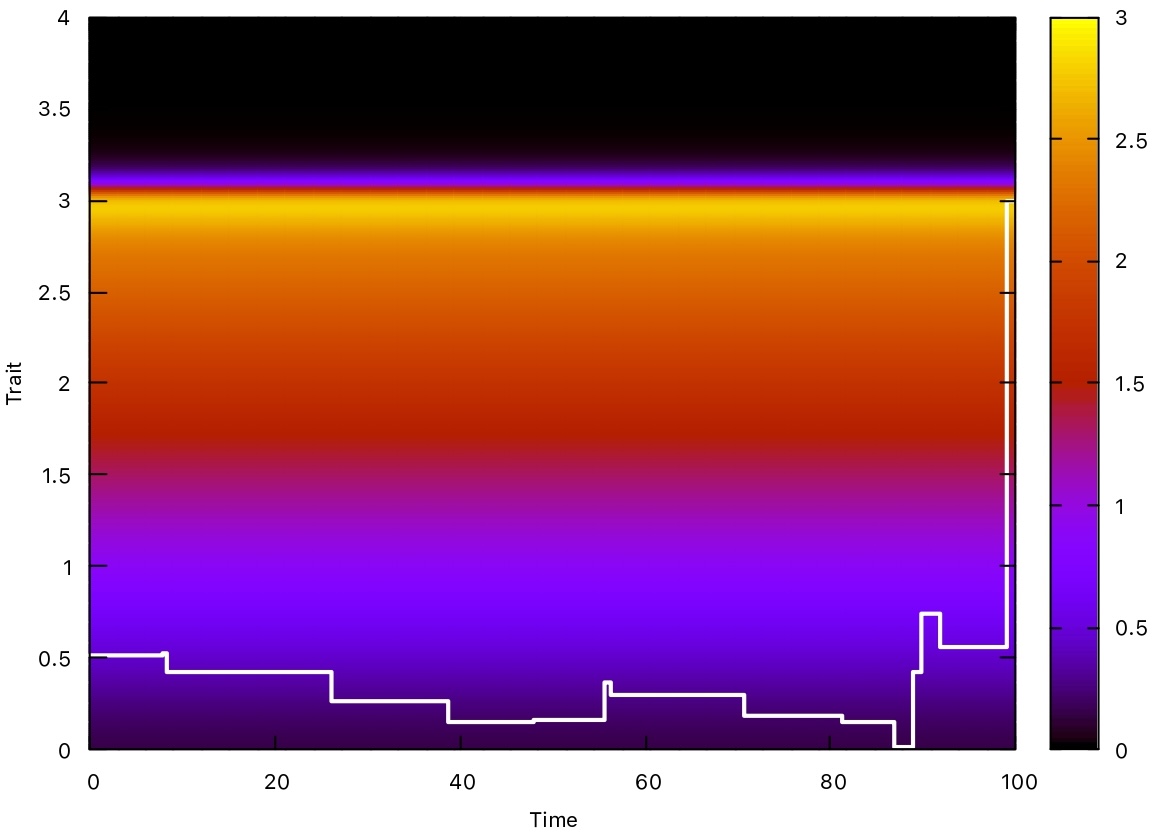} & 
            \includegraphics[width=0.45\textwidth, height=0.435\textheight, keepaspectratio]{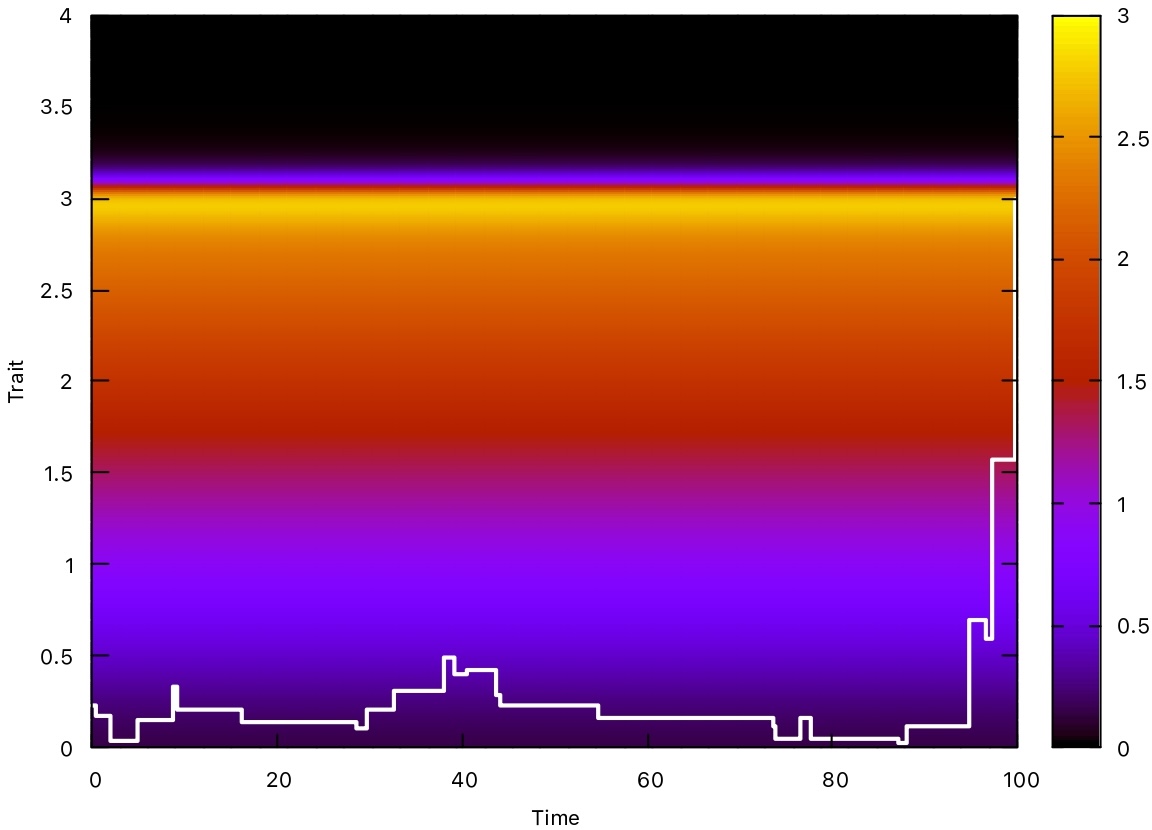} \\
            
            \includegraphics[width=0.45\textwidth, height=0.35\textheight, keepaspectratio]{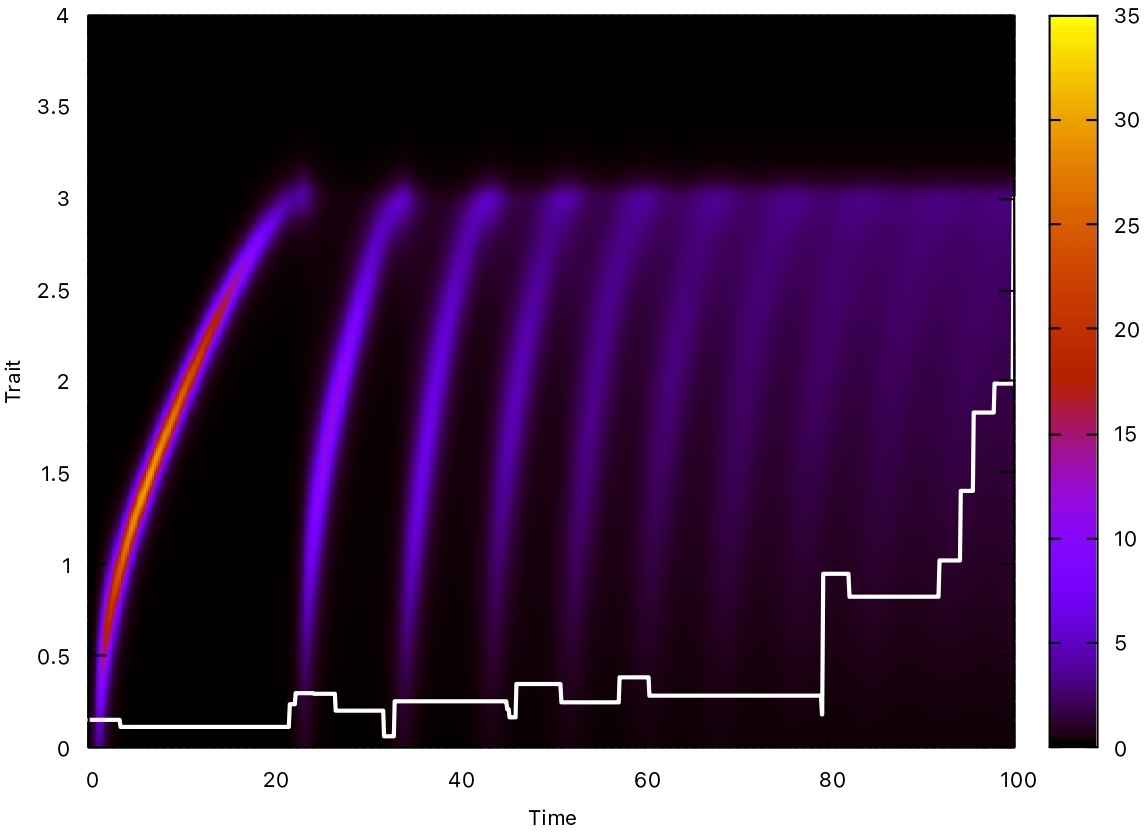} & 
            \includegraphics[width=0.45\textwidth, height=0.35\textheight, keepaspectratio]{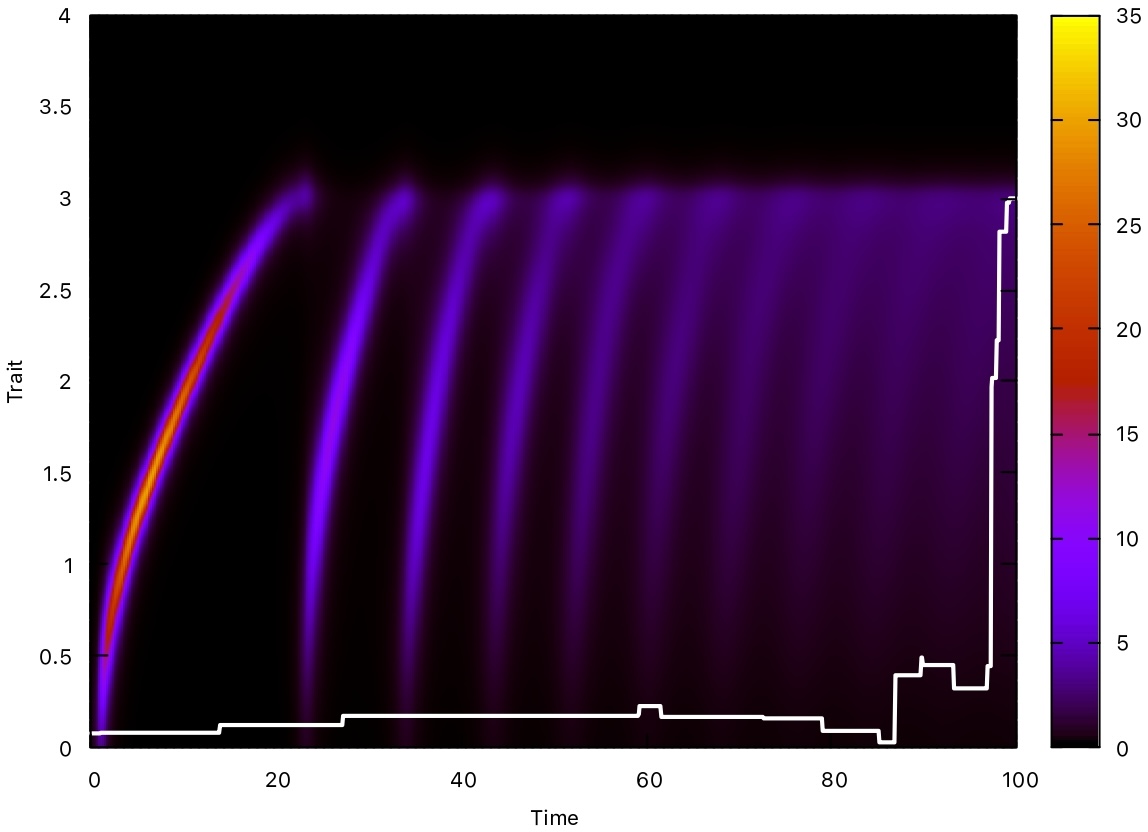} \\
        \end{tabular}

        \vspace{-0.2cm}
        \caption{\scriptsize Simulation of the ancestral lineage process in the stationary and non-stationary cases, with Gaussian variations - Parameters given in Table \ref{tab:experimental parameters}, $\tau=3,\,\beta=\mu=1.$}
  
    \label{fig:spine_comparison_gaussian}
    \end{figure}

Reasoning in forward time, this phenomenon highlights the evolutionary dead-end of carrying an excessive burden: if an individual's trait exceeds a certain threshold, the amplitude of Gaussian variations—which decay exponentially—is simply insufficient to allow a return to optimal traits before the lineage goes extinct due to selection. The severe demographic cost imposed by transfer dynamics acts as a strong filtering mechanism, effectively condemning these main branches to eventual extinction. Consequently, in the backward perspective, once the spine enters a region of adapted traits, large jumps toward unadapted traits are no longer observed. The victorious lineages are completely decoupled from the bulk of the highly burdened population. Under this short-range variation regime, contemporary individuals therefore originate strictly from those rare lineages that successfully evaded the transfer-induced accumulation of heavy burdens, maintaining a low-cost profile to ensure long-term viability.
\begin{figure}[h!]
        \centering
        \setlength{\tabcolsep}{2pt}
        \begin{tabular}{cc}
            \includegraphics[width=0.45\textwidth, height=0.35\textheight, keepaspectratio]{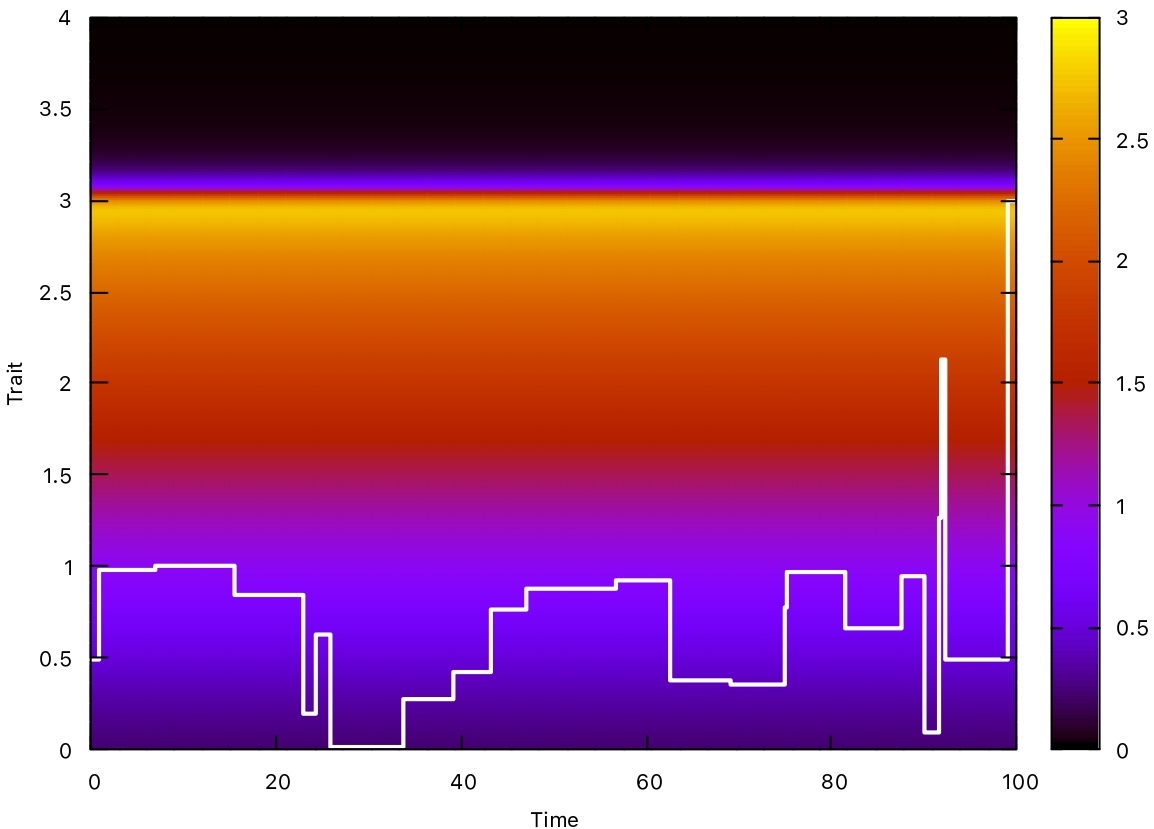} & 
            \includegraphics[width=0.45\textwidth, height=0.35\textheight, keepaspectratio]{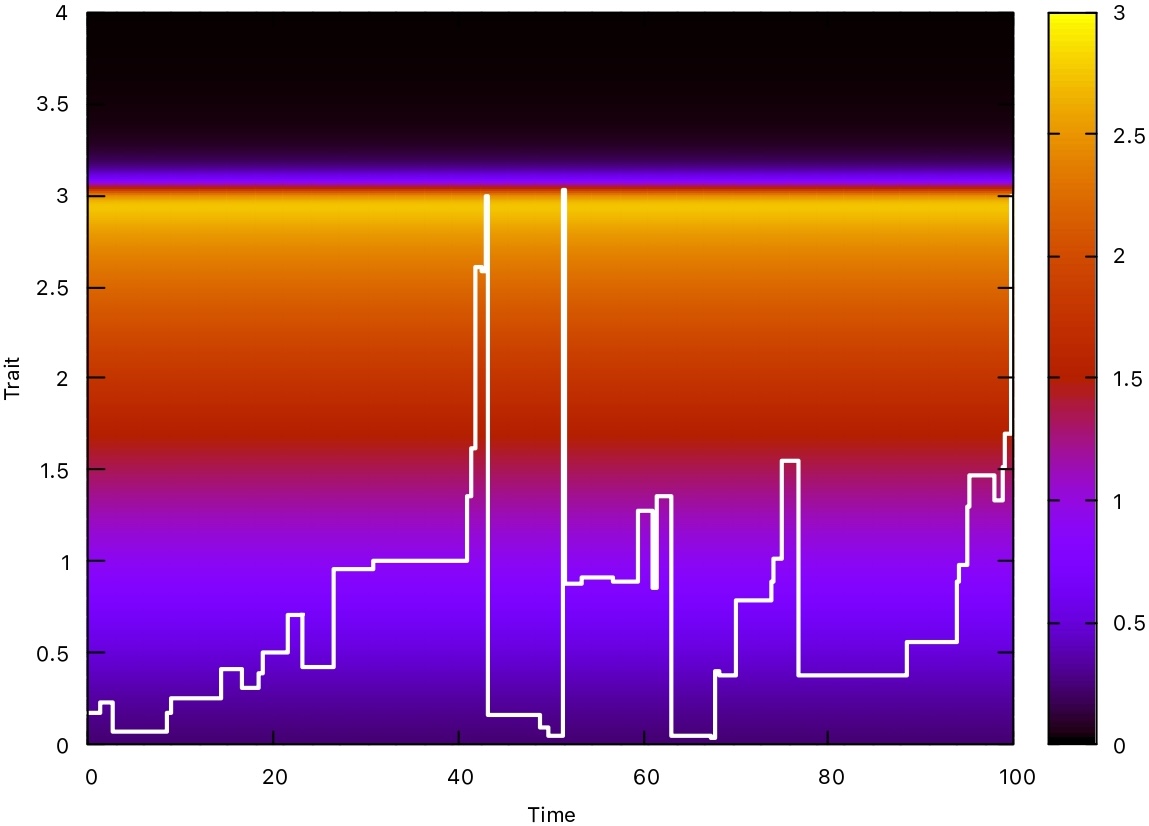} \\
            
            \includegraphics[width=0.45\textwidth, height=0.35\textheight, keepaspectratio]{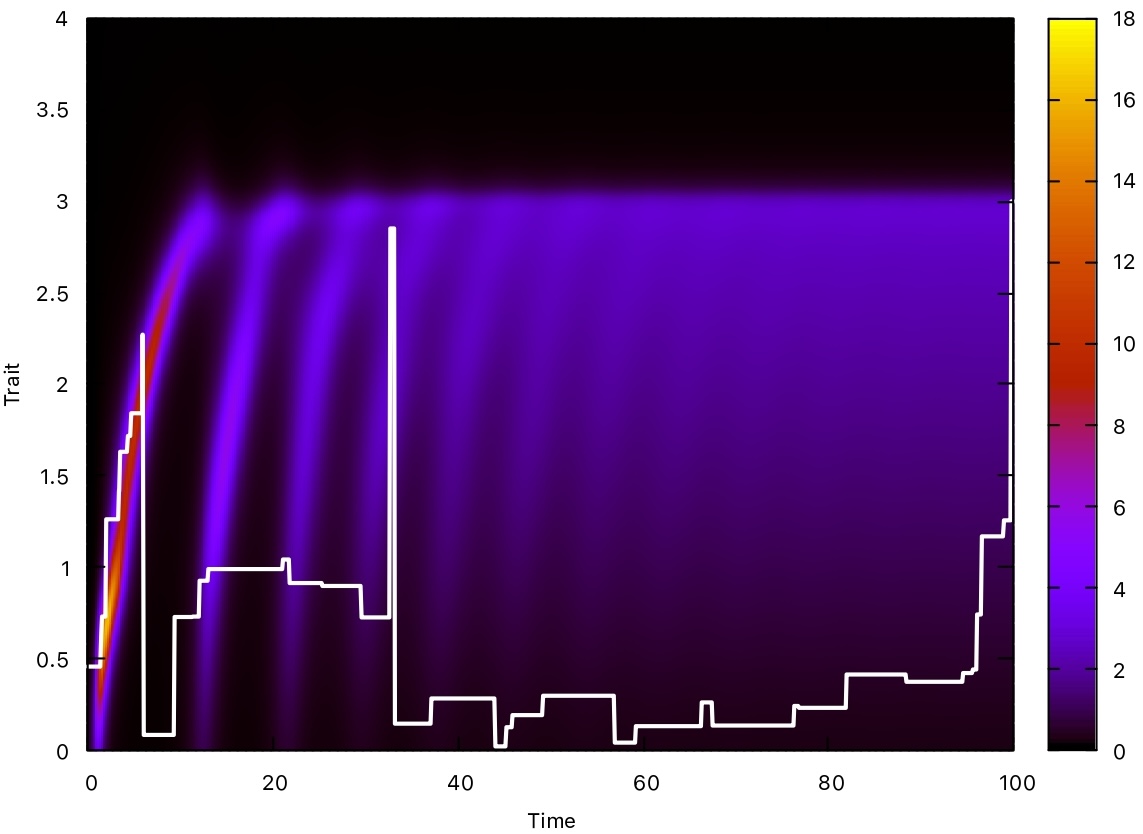} & 
            \includegraphics[width=0.45\textwidth, height=0.35\textheight, keepaspectratio]{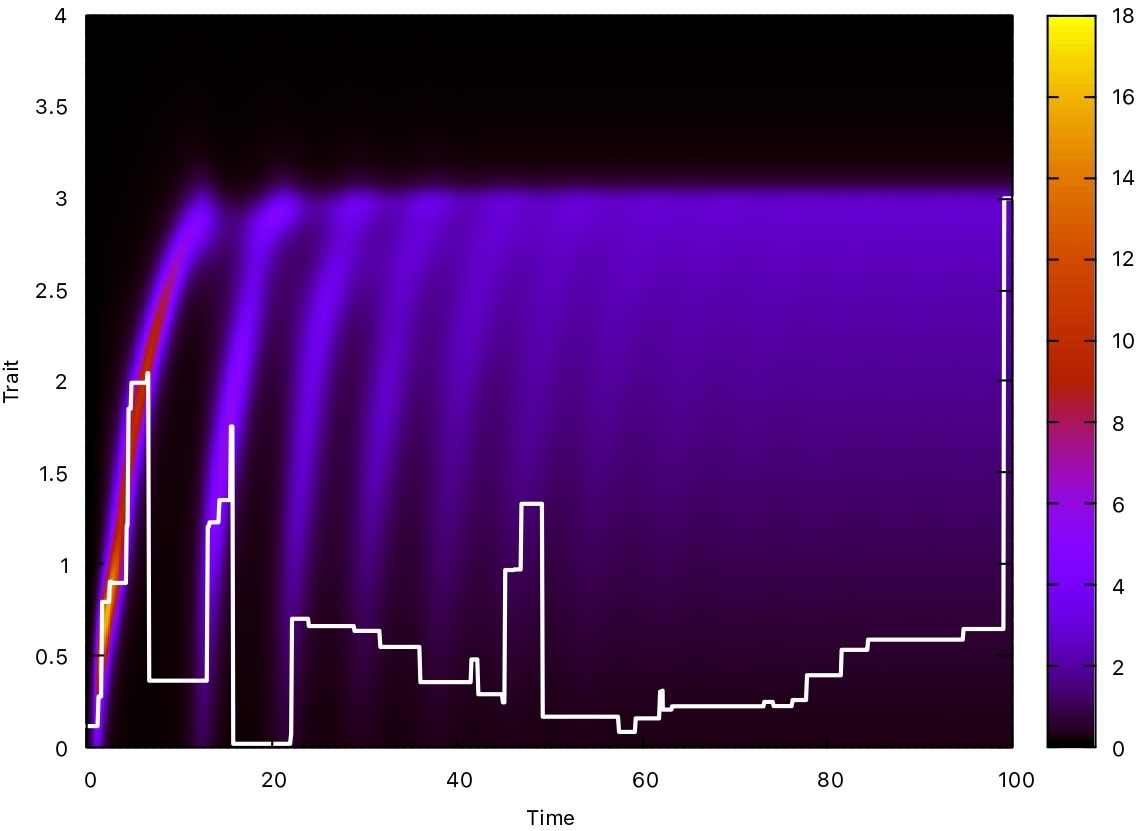} \\
        \end{tabular}
        
        \vspace{-0.2cm}
        \caption{\scriptsize Simulation of the ancestral lineage process in the stationary and non-stationary cases, with Cauchy variations - Parameters given in Table \ref{tab:experimental parameters}, $\widehat{m}(x)=\tfrac{\sigma}{\pi(\sigma^2+x^2)}$, $\sigma=0.1\,\tau=3,\,\beta=\mu=1$.}
  
    \label{fig:spine_comparison_Cauchy}
    \end{figure}
\newpage

\textbf{Case 2 - large-amplitude variations:} However, in the heavy-tailed regime, this genealogical picture changes drastically (when replacing the Gaussian variation kernel with a heavy-tailed Cauchy distribution). As shown in Figure \ref{fig:spine_comparison_Cauchy}, large jumps in the backward spinal process do indeed appear. The spine frequently escapes the optimal region near $0$ to visit regions of high trait values, and this behaviour persists whether the underlying macroscopic density is stationary or non-stationary.

This phenomenon arises directly from the heavy tail of the Cauchy distribution. In forward time, even when an individual carries a high and highly maladapted trait, there remains a small but non-negligible probability of being "rescued" by a large variation. This heavy-tailed rescue effect allows heavily burdened lineages to occasionally purge their excessive burden and give birth to highly adapted offspring, thereby contributing to the surviving population. In the backward perspective, this means that present-day adapted individuals can trace their ancestry back to highly unadapted ancestors, visually shown by sudden, large jumps of the spine into the main density clouds. Unlike in the Gaussian case, this regime reveals a highly fluctuating evolutionary path where long-term survival does not require avoiding high costs at all times. Instead, present-day lineages can indeed descend from heavily burdened ancestors, provided that catastrophic traits are sporadically purged by these rare, large-scale variational events.

\newpage
\section*{Acknowledgments}
This work has been funded by the Chair “Modélisation Mathématique
et Biodiversité” of Veolia Environnement-Ecole Polytechnique-Museum National d’Histoire
Naturelle-Fondation X and by the European Union(ERC, SINGER, 101054787). Views and opinions expressed are those of the author(s) only and do not necessarily reflect those of the European Union or the European Research Council. Neither the European Union nor the granting authority can be held responsible for them. The author would like to thank Sylvie Méléard and Viet-Chi Tran for their supervision, Sylvain Billiard, Meriem El Karoui and Alexandre Perrin for fruitful biological discussions, Théo Hérouard for his comprehensive answers on some analytical aspects and Daphné Giorgi for her kind help regarding the IBMPopSim Package \cite{giorgiibmpopsim}.
\nocite{*}
\bibliography{biblio}
\begin{appendix}
\section{Stochastic Differential Equations for the Processes}\label{appendix : EDS}
We write the stochastic differential equation associated with $(Z^K)$. Let $Z^K_0$ an $M_F(\mathbb{R}_+)$-valued random variable (the initial distribution). We consider three independent Poisson point measures: 
\begin{itemize}
    \item $N_1(ds,di,d\theta)$ on $\mathbb{R}_+\times\mathbb{N}^*\times\mathbb{R}_+$ with the intensity measure $ds\left( \sum_{k\geq1}\delta_k(di) \right)d\theta$
    \item $N_2(ds,di,d\theta,dz)$ on $\mathbb{R}_+\times\mathbb{N}^*\times\mathbb{R}_+\times\mathbb{R}_+$ with the intensity measure $ds\left( \sum_{k\geq1}\delta_k(di) \right)d\theta dz$
    \item $N_3(ds,di,dj,d\theta)$ on $\mathbb{R}_+\times\mathbb{N}^*\times\mathbb{N}^*\times\mathbb{R}_+$ with intensity measure $ds\left( \sum_{k\geq1}\delta_k(di) \right)\left( \sum_{k\geq1}\delta_k(dj) \right)d\theta$
\end{itemize}
The equation for $(Z^K)$ is:
\begin{equation}\label{eq:poisson representation}
    \begin{aligned}
    Z^K_t&=Z^K_0+\frac{1}{K}\int_0^t\int_{\mathbb{N}^*}\int_0^\infty \mathds{1}_{\{i\in V^K_{s-}\}} \delta_{X^i_{s-}}\left(\mathds{1}_{\{\theta\leq b(X^i_{s-})\}}-\mathds{1}_{\{0\leq \theta-b(X^i_{s-})\leq d(X^i_{s-})+C\langle Z^K_{s-},1\rangle\}}\right)N_1(ds,di,d\theta)\\
    &+\frac{1}{K}\int_0^t\int_{\mathbb{N}^*}\int_0^\infty \int_0^\infty \mathds{1}_{\{i\in V^K_{s-}\}}
    (\delta_z-\delta_{X^i_{s-}}) \mathds{1}_{\{\theta\leq \gamma(X^i_{s-})m(X^i_{s-},z)\}}N_2(ds,di,d\theta,dz)\\
    &+\frac{1}{K}\int_0^t\int_{\mathbb{N}^*}\int_{\mathbb{N}^*}\int_0^\infty \mathds{1}_{\{i,j\in V^K_{s-}\}}(\delta_{X^i_{s-}}-\delta_{X^j_{s-}})\mathds{1}_{\{\theta\leq h(X^i_{s-},X^j_{s-},Z^K_{s-})\}}N_3(ds,di,dj,d\theta)
    \end{aligned}
\end{equation}
For the auxiliary process $(\widehat{Z}^K)$, constructed on the same probability space as $(Z^K)$, with the same initial condition and Poisson point measures, we have for any measurable bounded function $\varphi$:
\begin{equation}\label{processus auxiliaire}
    \begin{aligned}
    \langle \widehat{Z}^K_t,\varphi\rangle &=\langle Z^K_0,\varphi\rangle+\widehat{M}^{K,\varphi}_t\\
    &+\int_0^t\int_0^\infty \left\{r(x)-C\langle\xi_s,1\rangle\right\}\varphi(x)\widehat{Z}^K_s(dx)ds\\
    &+\int_0^t\int_0^\infty \left\{\gamma(x)\int_0^\infty (\varphi(z)-\varphi(x))m(x,z)dz\right\}\widehat{Z}^K_s(dx)ds\\
    &+\int_0^t\int_0^\infty \left\{\int_0^\infty 
    h(x,y,\xi_s)(\varphi(x)-\varphi(y))\xi_s(dx)
    \right\}\widehat{Z}^K_s(dy)ds
\end{aligned}
\end{equation} where $\widehat{M}^{K,\varphi}$ is a square-integrable martingale with predictable quadratic variation  
\begin{equation*}
    \begin{aligned}
    \langle \widehat{M}^{K,\varphi}\rangle_t&=\frac{1}{K}\int_0^t\int_0^\infty \varphi^2(x)\{b(x)+d(x)+C\langle \xi_s,1\rangle\}\widehat{Z}^K_s(dx)ds\\
    &+\frac{1}{K}\int_0^t\int_0^\infty \left\{\gamma(x)\int_0^\infty  (\varphi(z)-\varphi(x))^2m(x,z)dz
  \right\}\widehat{Z}^K_s(dx)ds\\
  &+\frac{1}{K}\int_0^t\int_0^\infty \left\{ \int_0^\infty  h(x,y,\xi_s)(\varphi(x)-\varphi(y))^2 \xi_s(dy)\right\}\widehat{Z}^K_s(dx)ds.
    \end{aligned}
\end{equation*}
\section{Moments estimates and regularity results}\label{appendix : B}

\begin{proof}[Proof of Lemma \ref{lemma_moment_control}]\label{proof:lemma_moment_control}

The first estimate follows from a classical argument using Poisson representation \eqref{eq:poisson representation} and Gronwall's Lemma. We note that the transfer is not a problem because it leaves the mass unchanged. The second estimate is similar, we introduce the stopping-times $T_n^K=\inf\{t\geq0| \langle Z^K_t,1\rangle \geq n\text{ or }\int_0^t\langle Z^K_s,d\rangle ds \geq n\}$. By computing the quantity $\langle Z^K_t,1\rangle$, we get
\begin{equation*}
    \begin{aligned}
        \int_0^{T \wedge T^K_n} \langle Z^K_s, d\rangle ds &= \langle Z^K_0,1\rangle -\langle Z^K_{T \wedge T^K_n},1\rangle +\int_0^{T \wedge T^K_n} \langle Z^K_s, b(x)\rangle ds -C\int_0^{t \wedge T^K_n} \langle Z^K_s,1\rangle ^2ds+M^{K,1}_{T \wedge T^K_n}\\
        &\leq \langle Z^K_0,1\rangle +\overline{b}\int_0^{T \wedge T^K_n} \langle Z^K_s, 1\rangle ds+M^{K,1}_{T \wedge T^K_n}
    \end{aligned}
\end{equation*} and we can write 
\begin{equation*}
    \mathbb{E}\left[\int_0^{T \wedge T^K_n} \langle Z^K_s, d\rangle ds \right]\leq \sup_K \mathbb{E}[\langle Z^K_0,1\rangle]+ \overline{b}T\sup_K\mathbb{E}[\sup_{t\leq T}\langle Z^K_t,1\rangle].
\end{equation*}For any $K$, the sequence $(T^K_n)_n$ tends almost surely to infinity. If not, it would contradict the previous bound. We conclude by Fatou's Lemma.

The estimate of the time-integral squared is obtained the following way. We have 
\begin{equation*}
    \left(\int_0^T \langle Z^K_s, d\rangle ds\right)^2\leq 4\langle Z^K_0,1\rangle^2 +4\overline{b}^2\int_0^T \langle Z^K_s,1\rangle^2 ds+4(M^{K,1}_{T})^2
\end{equation*}
then
\begin{equation*}
    \mathbb{E}\left[\left(\int_0^T \langle Z^K_s, d\rangle ds\right)^2\right]\leq 4\sup_K\mathbb{E}[\langle Z^K_0,1\rangle^2] +4\overline{b}^2T \mathbb{E}[\sup_{t\leq T}\langle Z^K_t,1\rangle^2]+4\mathbb{E}[(M^{K,1}_{T})^2].
\end{equation*}
The first two terms are finite. The last term can be written:
\begin{equation*}
    \mathbb{E}[(M^{K,1}_T)^2]=\mathbb{E}[\langle M^{K,1}\rangle_T]\leq\frac{1}{K}\mathbb{E}[\langle Z^K_0,1\rangle^2]+\frac{C'}{K}\mathbb{E}[\sup_{t\leq T}\langle Z^K_t,1\rangle^2]+\frac{1}{K}\mathbb{E}\left[ \int_0^T \langle Z^K_s,d\rangle ds\right]
\end{equation*}
We conclude using our previous computations.

For the first estimate with $d(x)^\eta$,  let denote by $Y^K_t=\langle Z^K_t,d^\eta\rangle$ and introduce $T_n^K=\inf\{t\geq0|\langle Z^K_t,d^\eta\rangle\geq n\}$. Because $Z^K$ is pure jump Markov process we can write $Y^K_t=Y^K_0+P^K_t-D^K_t$ where
\begin{equation*}
P^K_t=\sum_{s\leq t}(\Delta Y^K_s)^+,\, D^K_t=\sum_{s\leq t}(\Delta Y^K_s)^-.
\end{equation*}
$P^K$ and $D^K$ are pure non decreasing jump processes. In particular, we have
\begin{equation*}
\sup_{s\leq t}Y^K_s\leq Y^K_0+P^K_t
\end{equation*} 
and 
\begin{equation*}
\mathbb{E}[\sup_{s\leq t\wedge T^K_n}Y^K_s]\leq \mathbb{E}[Y^K_0]+\mathbb{E}[P^K_{t\wedge T^K_n}].
\end{equation*}
The process $P^K$ counts only the jumps that increase $Y^K$. We write $\widetilde{P}^K$ its predictable compensator. $P^K-\widetilde{P}^K$ is a local martingale so
\begin{equation*}
\mathbb{E}[P^K_{t\wedge T^K_n}]=\mathbb{E}[{\widetilde{P}^K_{t\wedge T^K_n} }]
\end{equation*}
and 
\begin{equation*}
    \begin{aligned}
\widetilde{P}^K_t&=\int_0^t\int_0^\infty b(x)d(x)^\eta Z^K_s(dx)ds+\int_0^t\int_0^\infty \gamma(x)\left(\int_0^\infty (d(z)^\eta-d(x)^\eta)^+m(x,z)dz\right)Z^K_s(dx)ds\\
&+\int_0^t\int_0^\infty \int_0^\infty h(x,y,Z^K_s)(d(x)^\eta-d(y)^\eta)^+Z^K_s(dy)Z^K_s(dx)ds.
\end{aligned}
\end{equation*}

Immediately we have
\begin{equation*}
\int_0^t\int_0^\infty b(x)d(x)^\eta Z^K_s(dx)ds\leq \overline{b}\int_0^t Y^K_sds.
\end{equation*}
By the properties of $d^\eta$ and using the inequality $(A-B)^+\leq (A-B)+B$ for $A,B\geq0$, we get
\begin{equation*}
\int_0^t\int_0^\infty \gamma(x)\left(\int_0^\infty (d(z)^\eta-d(x)^\eta)^+m(x,z)dz\right)Z^K_s(dx)ds\leq 2\overline{\gamma}\int_0^t Y^K_s ds.
\end{equation*}
The last term can be roughly bounded by
\begin{equation*}
\int_0^t\int_0^\infty \int_0^\infty h(x,y,Z^K_s)(d(x)^\eta-d(y)^\eta)^+Z^K_s(dy)Z^K_s(dx)ds\leq \int_0^t \frac{2\tau\langle Z^K_s,1\rangle}{\beta+\mu\langle Z^K_s,1\rangle} Y^K_s ds.
\end{equation*}
If $\mu>0$, then
\begin{equation*}
\mathbb{E}[\widetilde{P}^K_{t\wedge T^K_n}]\leq (\overline{b}+2\overline{\gamma}+2\tfrac{\tau}{\mu})\int_0^t\mathbb{E}\left[Y^K_{s\wedge T^K_n}\right] ds
\end{equation*}
and we can directly conclude by Gronwall's Lemma. Now consider $\mu=0$. We write 
\begin{equation*}
\int_0^t 2\tfrac{\tau}{\beta}\langle Z^K_s,1\rangle Y^K_sds \leq 2\tfrac{\tau}{\beta}\sup_{s\leq t}\langle Z^K_s,1\rangle\int_0^t Y^K_sds
\end{equation*}
and by Cauchy-Schwarz inequality
\begin{equation}\label{astuce cauchy schwarz transfert}
\mathbb{E}\left[ \int_0^{t\wedge T^K_n}\langle Z^K_s,1\rangle Y^K_s ds \right]\leq \tfrac{2\tau}{\beta}\sqrt{\mathbb{E}[\sup_{s\leq t}\langle Z^K_s,1\rangle^2]}\sqrt{\mathbb{E}\left[\left(\int_0^t Y^K_{s\wedge T^K_n}ds\right)^2\right]}.
\end{equation}
Because $\eta<1$, there exists a constant $D'\geq0$ such that $d(x)^\eta\leq d(x)+D'$. Therefore $Y^K_t\leq \langle Z^K_t,d\rangle+D'\langle Z^K_t,1\rangle$ and by the estimate \eqref{eq:estimate death}, we can also conclude using Gronwall's Lemma in the case $\mu=0$.

Lastly, let denote $T^K_n=\inf\{t\geq0| \langle Z^K_t,1\rangle +\langle Z^K_t,d^\eta\rangle+\int_0^t\langle Z^K_s, d^{1+\eta}\rangle ds \geq n\}$. By evaluating the quantity $\langle Z^K_t,d^\eta\rangle$, we have
\begin{equation*}
    \begin{aligned}
        \int_0^t \langle Z^K_{s\wedge T^K_n}, d^{1+\eta}\rangle ds&\leq \langle Z^K_0,d^\eta\rangle +\overline{b}\int_0^t \langle Z^K_{s\wedge T^K_n}, d^\eta\rangle ds+M^{K,d^\eta}_{t\wedge T^K_n}\\
        &+\overline{\gamma}\int_0^t (\langle Z^K_{s\wedge T^K_n}, d^\eta\rangle+D\langle Z^K_{s\wedge T^K_n},1\rangle)ds\\
        &+\int_0^t\frac{\tau \langle Z^K_{s\wedge T^K_n},1\rangle}{\beta+\mu\langle Z^K_{s\wedge T^K_n},1\rangle}\langle Z^K_{s\wedge T^K_n},d^\eta\rangle ds
    \end{aligned}
\end{equation*}
If $\mu > 0$, it suffices to upper-bound the transfer term by $\tau/\mu$, take the expectation, and conclude using the previous estimates. If $\beta > 0$ and $\mu = 0$, we follow the same procedure as previously but also use the same upper bound as in \eqref{astuce cauchy schwarz transfert}.
\end{proof}

\begin{proof}[Proof of Proposition \ref{prop:produit_par_x_fini}]\label{proof:reflected_supremum}
The integrable case is straightforward by Gronwall's Lemma.

Define $\hat{C}=\sup_{x\in\mathbb{R}_+}x\hat{m}(x)$, $v_n(t,x)=\min(x,n)\xi_t(x)$ and $W_n(t)=\sup_{x\geq 0}v_n(t,x)$. If $\sup_{x\geq0}x\xi_0(x)<\infty$ then $\sup_{x\geq0}\xi_0(x)<\infty$ and $\sup_{t\leq T,\, x\geq0}\xi_t(x):=\overline{\xi}_T<\infty$. We also have $\sup_{t\geq0}\langle\xi_t,1\rangle\leq \overline{b}/C$.

We multiply equation \eqref{densité limite} by $\min(x,n)$:
\begin{equation*}
        \partial_t v_n(t,x)=\left(r(x)-C\lambda_t -\gamma(x) +\tfrac{\tau}{\beta+\mu\lambda_t}\int_0^\infty  sgn(x-y)\xi_t(y)dy\right)v_n(t,x)+ I(x)
\end{equation*}
where
\begin{equation*}
    I_n(x):=\min(x,n)\int_0^\infty  \gamma(y)(\hat{m}(x-y)+\hat{m}(x+y))\xi_t(y)dy.
\end{equation*}
We start by bounding the variation term. Using the subadditivity of $w_n$, we have $\min(x,n)\leq |x-y|+\min(y,n)$
\begin{equation*}
    \begin{aligned}
        \int_0^\infty  \gamma(y)\hat{m}(x-y)\min(y,n)\xi_t(y)dy&\leq\overline{\gamma}\int_0^\infty |x-y|\hat{m}(x-y)\xi_t(y)dy\\
        &+\overline{\gamma}\int_0^\infty  \hat{m}(x-y)\min(y,n)\xi_t(y)dy\\
        &\leq \hat{C}\tfrac{\overline{b}}{C}+W_n(t).
    \end{aligned}
\end{equation*}
and because $\min(x,n)\leq x \leq x+y$,
\begin{equation*}
    \begin{aligned}
        \int_0^\infty  \gamma(y)\hat{m}(x+y)\min(x,n)\xi_t(y)dy&\leq\overline{\gamma}\int_0^\infty  (x+y)\hat{m}(x+y)\xi_t(y)dy\\
        &\leq\hat{C}\tfrac{\overline{b}}{C}.
    \end{aligned}
\end{equation*}
In consequence, it exists $B_1, B_2>0$ such that  $I_n(x)\leq B_1 W_N(t)+B_2$.
Also, we can find a finite $B_3>0$ such that 
\begin{equation*}
    \partial_t v_n(t,x)\leq (r(x)+B_3)v_n(t,x)+B_1W_n(t)+B_2.
\end{equation*}
Because $\lim_{x\rightarrow\infty}r(x)=-\infty$, there exists $R>0$ such that for all $x\leq R$, $v_n(t,x)\leq R\overline{\xi}$ and for all $x>R$, $r(x)<-B_3$, so $v_n(t,x)\leq W_n(0)+B_2 T+ B_1\int_0^t W_n(s)ds$.
Taking the supremum over $x$, we obtain
\begin{equation*}
    W_n(t)\leq B_4+B_1\int_0^t W_n(s)ds
\end{equation*}
where $B_4=R\overline{\xi}+W(0)+B_2 T$. The boundedness of $W_n$ allows us to use the Gronwall's Lemma:
\begin{equation*}
    \forall t\leq T,\, W_n(t)\leq B_4 e^{B_1T}.
\end{equation*}
The last bound does not depend on $n$, we can make it tends to infinity and get
\begin{equation*}
    \sup_{x\geq0} x\xi_t(x)\leq B_4 e^{B_1 T}, \, \forall t\leq T.
\end{equation*}
\end{proof}
\begin{proof}[Detailed proof of Lemma \ref{lemma : analytical expression of functional}]\label{appendix:detailed_proof_series}
Let $N_t$ be a homogeneous Poisson process of intensity $\overline{\Lambda}$, representing the potential jumps, with ordered jump times $T_1<\dots<T_{N_t}$. By the law of total probability,
\begin{equation*}
    f(t,s,x) = \sum_{k\ge0} \mathbb{P}(N_t=k)\,\mathbb{E}[Y \mid J_s=x, N_t=k] 
= \sum_{k\ge0} e^{-\overline{\Lambda} t}\frac{(\overline{\Lambda} t)^k}{k!}\,\varphi_k(t,s,x)
\end{equation*}
where
\[
Y:=e^{\int_0^t A_{s+u}(J_{s+u})du}\,\varphi_{s+t}(J_{s+t}), 
\qquad \varphi_k(t,s,x):=\mathbb{E}[Y\mid J_s=x,N_t=k].
\]

- For $k=0$ (no jumps), the process stays at $x$ on $[s,s+t]$, hence
\[
\varphi_0(t,s,x)=e^{\int_0^t A_{s+v}(x)dv}\,\varphi(s+t,x).
\]

- For $k\ge1$, let $(T_1,...,T_k)$ denote the vector of jump times, by increasing order. By Theorem 2.15 of \cite{durrettprocesses}, the random vector has a density given by
\begin{equation*}
    f_{(T_1,...,T_k)}(t_1,...,t_k)=\frac{k!}{t^k}\mathds{1}_{\{0\leq t_1\leq...\leq t_k\leq t\}}.
\end{equation*}

We deduce that the conditional density of $T_1$ given $N_t=k$ is
\[
f_{T_1|N_t=k}(u)=\frac{k(t-u)^{k-1}}{t^k},\quad u\in[0,t].
\]

Then
\[
\varphi_k(t,s,x)=\int_0^t \mathbb{E}[Y\mid J_s=x,N_t=k,T_1=u]\,
f_{T_1|N_t=k}(u)\,du.
\]

Let $U_1$ be an uniform random variable on $[0,1]$, independent of $N$. We define the event $B_1 := \{U_1 \le \Lambda(s+u,x) / \overline{\Lambda}\}$, which corresponds to the acceptance and realization of the potential jump. Then
\[
\mathbb{E}[Y\mid J_s=x,N_t=k,T_1=u]
=\mathbb{E}[Y\mid B_1,\Gamma]\mathbb{P}(B_1|\Gamma)
+\mathbb{E}[Y\mid B_1^c,\Gamma]\mathbb{P}(B_1^c|\Gamma),
\]
with $\Gamma=\{J_s=x,N_t=k,T_1=u\}$.  

We have 
\[
\mathbb{E}[Y\mid B_1,\Gamma]\mathbb{P}(B_1|\Gamma)
=e^{\int_0^u A_{s+v}(x)dv}\frac{\Lambda(s+u,x)}{\overline{\Lambda}}
\int_0^\infty \varphi_{k-1}(t-u,s+u,z_1)\,\nu(s+u,x,z_1)\,dz_1
\]
and 
\[
\mathbb{E}[Y\mid B_1^c,\Gamma]\mathbb{P}(B_1^c|\Gamma)
=e^{\int_0^u A_{s+v}(x)dv}\Big(1-\frac{\Lambda(s+u,x)}{\overline{\Lambda}}\Big)\varphi_{k-1}(t-u,s+u,x).
\]

Define the recurrence operator: for bounded measurable function $\psi$ on $\mathbb{R}_+$,
\[
(\mathcal K_{s,u}\psi)(x):=
e^{\int_0^u A_{s+v}(x)dv}\Bigg[
\frac{\Lambda(s+u,x)}{\overline{\Lambda}}\int_0^\infty \psi(z_1)\nu(s+u,x,z_1)\,dz_1
+\Big(1-\frac{\Lambda(s+u,x)}{\overline{\Lambda}}\Big)\psi(x)
\Bigg].
\]

Then
\[
\varphi_k(t,s,x)
=\int_0^t \frac{k(t-u)^{k-1}}{t^k}\,
(\mathcal K_{s,u}\varphi_{k-1}(t-u,s+u,\cdot))(x)\,du.
\]
\end{proof}
\section{Historical processes}\label{app:historical}
Following \cite{meleardB2, meleardB0}, we consider the following class of functions:  
for elements of $\mathbb{D}(\mathbb{R},\mathbb{R})$, we take test functions of the form $\Phi_\varphi$, where $\Phi \in C_b(\mathbb{R},\mathbb{R})$ and $\varphi \in C(\mathbb{R}_+\times\mathbb{R},\mathbb{R})$, defined by  
\begin{equation}\label{eq:test functions on D}
    \Phi_\varphi(y) = \Phi\left(\int_0^T \varphi(t,y_t)\,dt\right), \qquad y\in \mathbb{D}(\mathbb{R},\mathbb{R}).
\end{equation}

For $y \in \mathbb{D}(\mathbb{R},\mathbb{R})$, $s \in [0,T]$, and $x \in \mathbb{R}$, we denote by $(y|s|x)$ the càdlàg function given by  
\[
(y|s|x)(t) = y(t)\,\mathds{1}_{\{t<s\}} + x\,\mathds{1}_{\{t\geq s\}}.
\]

The non-branching historical process $H$ satisfies the following for all $\Phi_\varphi$:
\begin{equation}\label{eq:historical process semimartingale}
    \begin{aligned}
        \langle H^K_t, \Phi_\varphi \rangle &=
        \langle H^K_0, \Phi_\varphi \rangle
        +\int_0^t\int_{\mathbb{D}(\mathbb{R},\mathbb{R})}\left\{ b(y_s)-d(y_s)-C\langle H^K_s,1\rangle\right\} \Phi_\varphi(y)H^K_s(dy)ds\\
        &+\int_0^t \int_{\mathbb{D}(\mathbb{R},\mathbb{R})}\left\{\gamma(y_s)\int_0^\infty (\Phi_\varphi(y|s|x)-\Phi_\varphi(y)) m(y_s,x) dx \right\} H^K_s(dy)ds\\
        &+\int_0^t\int_{\mathbb{D}(\mathbb{R},\mathbb{R})}\left\{\int_0^\infty  h(x,y_s,Z^K_s)(\Phi_\varphi(y|s|x)-\Phi_\varphi(y))Z^K_s(dx)\right\}H^K_s(dy)ds+M^{K,\Phi,\varphi}_t
    \end{aligned}
\end{equation}
where $M^{K,\Phi,\varphi}$ a square-integrable martingale with predictable quadratic variation 
\begin{equation*}\label{variation quadratique prévisible historique}
    \begin{aligned}
        \langle M^{K,\Phi,\varphi}\rangle_t&=\frac{1}{K}\int_0^t \int_{\mathbb{D}(\mathbb{R},\mathbb{R})}\{b(y_s)+d(y_s)+C\langle H^K_s,1\rangle\}\Phi^2_\varphi(y)H^K_s(dy)ds\\
        &+\frac{1}{K}\int_0^t\int_{\mathbb{D}(\mathbb{R},\mathbb{R})}\left\{\int_0^\infty  \gamma(y_s)(\Phi_\varphi(y|s|x)-\Phi_\varphi(y))^2 m(y_s,x)dx\right\} H^K_s(dy)ds\\
        &+\frac{1}{K}\int_0^t\int_{\mathbb{D}(\mathbb{R},\mathbb{R})}\left\{\int_0^\infty h(x,y_s,Z^K_s)(\Phi_\varphi(y|s|x)-\Phi_\varphi(y))^2 Z^K_s(dx)\right\}H^K_s(dy)ds
    \end{aligned}
\end{equation*}
The auxiliary branching historical process $\widehat{H}$ verifies for all $\Phi_\varphi$:
\begin{equation}\label{processus historique auxiliaire}
    \begin{aligned}
        \langle \widehat{H}^K_t, \Phi_\varphi \rangle &=
        \langle H^K_0, \Phi_\varphi \rangle
        +\int_0^t\int_{\mathbb{D}(\mathbb{R},\mathbb{R})}\left\{ b(y_s)-d(y_s)-C\langle \xi_s,1\rangle \right\} \Phi_\varphi(y)\widehat{H}^K_s(dy)ds\\
        &+\int_0^t \int_{\mathbb{D}(\mathbb{R},\mathbb{R})}\left\{\gamma(y_s)\int_0^\infty (\Phi_\varphi(y|s|x)-\Phi_\varphi(y)) m(y_s,x) dx \right\} \widehat{H}^K_s(dy)ds\\
        &+\int_0^t\int_{\mathbb{D}(\mathbb{R},\mathbb{R})}\left\{\int_0^\infty  h(x,y_s,\xi_s)(\Phi_\varphi(y|s|x)-\Phi_\varphi(y))\xi_s(dx)\right\}\widehat{H}^K_s(dy)ds+\widehat{M}^{K,\Phi,\phi}_t
    \end{aligned}
\end{equation}
with $\widehat{M}^{K,\Phi,\varphi}$ a square-integrable martingale with predictable quadratic variation 
\begin{equation*}\label{variation quadratique prévisible historique auxiliaire}
    \begin{aligned}
        \langle \widehat{M}^{K,\Phi,\varphi}\rangle_t&=\frac{1}{K}\int_0^t \int_{\mathbb{D}(\mathbb{R},\mathbb{R})}\{b(y_s)+d(y_s)+C\langle \xi_s,1\rangle \}\Phi^2_\varphi(y)\widehat{H}^K_s(dy)ds\\
        &+\frac{1}{K}\int_0^t\int_{\mathbb{D}(\mathbb{R},\mathbb{R})}\left\{\int_0^\infty  \gamma(y_s)(\Phi_\varphi(y|s|x)-\Phi_\varphi(y))^2 m(y_s,x)dx\right\} \widehat{H}^K_s(dy)ds\\
        &+\frac{1}{K}\int_0^t\int_{\mathbb{D}(\mathbb{R},\mathbb{R})}\left\{\int_0^\infty h(x,y_s,\xi_s)(\Phi_\varphi(y|s|x)-\Phi_\varphi(y))^2 \xi_s(dx)\right\}\widehat{H}^K_s(dy)ds
    \end{aligned}
\end{equation*}
The following lemma generalises the technique of the auxiliary function $\varphi^t$, thereby avoiding the need to deal with the death term $d(x)$ (and the historical processes evaluation difference) when the trajectories of the lineages take values in $\mathbb{D}([0,T],\mathbb{R}_+)$. Consider a test function of the form $\Phi_\varphi$ and a time-space Markov process $(t,Y_t)_{t\in[0,T]}$ such that for all $t\in[0,T]$, $Y_t\in\mathbb{D}([0,T],\mathbb{R}_+)$ and $Y_t(s)\in\mathbb{R}_+$. The dynamics are as follows: we start from a trajectory $y\in \mathbb{D}([0,T],\mathbb{R})$, time elapses and when a jump occurs, the trajectory changes at the moment of the jump. The generator of the process is given by
\begin{equation*}
\begin{aligned}
    \mathcal{L}\Phi_\varphi(s,y)&=\partial_s\Phi_\varphi(s,y)+\gamma(y_{s-})\int_0^\infty (\Phi_\varphi(s,(y|s|x))-\Phi_\varphi(s,y))m(y_{s-},x) dx\\
&+\int_0^\infty  h(x,y_{s-},\xi_s)(\Phi_\varphi(y|s|x)-\Phi_\varphi(y))\xi_s(x)dx.
\end{aligned}
\end{equation*}
\begin{lemma}\label{lemma : transformée FK}
Let $R(s,y)=b(y_{s-})-d(y_{s-})-C\langle\xi_s,1\rangle\leq \overline{b}$ and let, for all $s\leq t$,
\begin{equation*}
\psi^t(s,y)=\mathbb{E}\left[\Phi_\varphi(Y_t)e^{\int_s^t R(u,Y_u)du}|Y_s=y\right].
\end{equation*}
Then $\psi^t$ is uniformly bounded by $\|\Phi_\varphi\|_\infty e^{\overline{b}t}$ and satisfies 
\begin{equation}\label{eq:fk skorohod}
    \mathcal{L}\psi^t(s,y)+R(s,y)\psi^t(s,y)=0,\, \psi^t(t,y)=\Phi_\varphi(y).
\end{equation} 
We deduce that 
\begin{equation}\label{eq:historical without death}
    \begin{aligned}
        \langle H^K_t,\Phi_\varphi\rangle & = \langle H^K_0, \psi^t(0,\cdot)\rangle +C\int_0^t\int_{{\mathbb{D}(\mathbb{R},\mathbb{R})}}\Big\{\langle\xi_s- Z^K_s,1\rangle\psi^t(s,y)\\
        &+\int_0^\infty (\psi^t(s,(y|s|x))-\psi^t(s,y))
        [h(x,y_s,Z^K_s)Z^K_s(dx)-h(x,y_s,\xi_s)\xi_s(dx)]\Big\}H^K_s(dy)ds+R^K_{t,t}
    \end{aligned}
\end{equation}
where $t\mapsto R^K_{t,t}$ is not a martingale but $u\mapsto R^K_{u,t}$ is, with
\begin{equation*}
    \begin{aligned}
        \langle R^K_{u,t}\rangle &=\frac{1}{K}\int_0^u\Big\{ (b(y_s)+d(y_s)+C\langle H^K_s,1\rangle)\psi^t(s,y)^2
        +\gamma(y_s)\int_0^\infty (\psi^t(s,(y|s|x))-\psi^t(s,y))^2m(y_s,x)dx\\
        &+\int_0^\infty h(x,y_s,Z^K_s)(\psi^t(s,(y|s|x))-\psi^t(s,y))^2 Z^K_s(dx)\Big\}H^K_s(dy)ds.
    \end{aligned}
\end{equation*}
There exists a constant $C_T>0$ such that 
\begin{equation*}
    \mathbb{E}[\sup_{u\leq t} (R^K_{u,t})^2]\leq \frac{C_T}{K}.
\end{equation*}

For the auxiliary process, we get
\begin{equation}\label{eq:historical aux without death}
        \langle \widehat{H}^K_t,\Phi_\varphi\rangle = \langle H^K_0, \psi^t(0,\cdot)\rangle +\widehat{R}^K_{t,t}
\end{equation}
where $t\mapsto \widehat{R}^K_{t,t}$ is not a martingale but $u\mapsto \widehat{R}^K_{u,t}$ is, with
\begin{equation*}
    \begin{aligned}
        \langle \widehat{R}^K_{u,t}\rangle &=\frac{1}{K}\int_0^u\Big\{ (b(y_s)+d(y_s)+C\langle \xi_s,1\rangle)\psi^t(s,y)^2
        +\gamma(y_s)\int_0^\infty (\psi^t(s,(y|s|x))-\psi^t(s,y))^2m(y_s,x)dx\\
        &+\int_0^\infty h(x,y_s,\xi_s)(\psi^t(s,(y|s|x))-\psi^t(s,y))^2 \xi_s(dx)\Big\}\widehat{H}^K_s(dy)ds.
    \end{aligned}
\end{equation*}
There exists a constant $\widehat{C}_T>0$ such that 
\begin{equation*}
    \mathbb{E}[\sup_{u\leq t} (\widehat{R}^K_{u,t})^2]\leq \frac{\widehat{C}_T}{K}.
\end{equation*}
\end{lemma}
\begin{proof}
Equation \eqref{eq:fk skorohod} is an application of the Feynman–Kac formula. Equations \eqref{eq:historical without death} and \eqref{eq:historical aux without death} are obtained by applying \eqref{eq:fk skorohod} to \eqref{eq:historical process semimartingale} and \eqref{processus historique auxiliaire}. We find $C_T$ and $\widehat{C}_T$ using the estimates of Lemma  \ref{lemma_moment_control}
and the fact that $\langle H^K_s,1\rangle=\langle Z^K_s,1\rangle$ and $\langle H^K_s,d(y_s)\rangle=\langle Z^K_s,d\rangle$.
\end{proof}
We need to ensure the continuity in the $J_1$ topology of the auxiliary function $\psi^t(s,\cdot)$.
\begin{proposition}[Continuity criterion for $\psi^t(s,\cdot)$]\label{prop:continuity criterion in J1}
As in Lemma \ref{lemma : analytical expression of functional}, we can write a series representation for $\psi^t(s,y)$:
\begin{equation*}
        \psi^t(s,y)=e^{-\overline{\Lambda}(t-s)}\sum_{k\geq0}\tfrac{(\overline{\Lambda}(t-s))^k}{k!}\psi^t_k(s,y)
\end{equation*}
with $\psi^t_0(s,y)=\Phi_\varphi(y)e^{\int_s^t R(u,y)du}$,
\begin{equation*}
        \psi^t_k(s,y)=\int_s^t e^{\int_s^u R(v,y)dv}  \left[\tfrac{\Lambda(u,y)}{\overline{\Lambda}}\int_0^\infty \psi^t_{k-1}(u,(y|u|x))\nu(u,y,x)dx+\left(1-\tfrac{\Lambda(u,y)}{\overline{\Lambda}}\right)\psi^t_{k-1}(u,y)\right]\tfrac{k(t-u)^{k-1}}{(t-s)^k}du ,
\end{equation*}
\begin{equation*}
    \overline{\Lambda}\geq\Lambda(u,y)=\gamma(y_{u-})+\tfrac{\tau}{\beta+\mu\langle\xi_u,1\rangle}\int_0^\infty \mathds{1}_{\{x>y_{u-}\}}\xi_u(x)dx
\end{equation*}
and
\begin{equation*}
    \nu(u,y,x)=\Lambda(u,y)^{-1}\left(\gamma(y_{u-})m(y_{u-},x)+\tfrac{\tau}{\beta+\mu\langle\xi_u,1\rangle}\mathds{1}_{\{x>y_{u-}\}}\xi_u(x)\right)
\end{equation*}
We can state that the mapping $y \mapsto \psi^t(s,y)$ is continuous (with respect to the $J_1$ topology).
\end{proposition}
\begin{proof}[Proof of Proposition \ref{prop:continuity criterion in J1}]

Similarly to Corollary \ref{corollary:series continuity Linfini}, we find that if the following conditions hold:
\begin{itemize}
    \item The function $\psi^t_0(s,\cdot)$ is continuous with respect to the $J_1$ topology and uniformly bounded.
    \item The function $R(u,y)$ is uniformly bounded from above by $\overline{b}$ and for any sequence $y^n \xrightarrow{J_1} y$, we have $R(u,y^n) \to R(u,y)$ for almost every $u \in [s,t]$.
    \item For every $y \in D$ such that there exists a sequence $y^{n} \xrightarrow{J_1} y$, we have
\begin{equation*}
\lim_{n\rightarrow\infty}\Lambda(u,y^{n})=\Lambda(u,y) \quad \text{for almost every } u \in [s,t].
\end{equation*}
    \item For every $y \in D$ such that there exists a sequence $y^{n} \xrightarrow{J_1} y$, and for every $f \in C_b(\mathbb{R}_+)$, we have
\begin{equation*}
\lim_{n\rightarrow\infty}\int_0^\infty f(x)\nu(u,y^{n},x)dx=\int_0^\infty f(x)\nu(u,y,x)dx \quad \text{for almost every } u \in [s,t].
\end{equation*}
\end{itemize}
Then, the mapping $y \mapsto \psi^t(s,y)$ is continuous with respect to the $J_1$ topology.

Indeed:
\begin{itemize}
    \item The functions $\gamma$ and $u\mapsto \tfrac{\tau}{\beta+\mu\langle\xi_u,1\rangle}$ are continuous, bounded and positive.
    \item $\xi$ is positive, continuous, integrable and uniformly bounded.
    \item The kernel $m$ is assumed to be continuous.
    \item By the properties of the $J_1$ topology, if $y^n \overset{J_1}{\rightarrow} y$, then $y^{n}_{u-} \overset{J_1}{\rightarrow} y_{u-}$ for almost all $u$.
    \item Since $\xi_u$ is a probability density, it has no atoms, meaning that $\lim_{n\to\infty} \mathds{1}_{\{x>y^n_{u-}\}} = \mathds{1}_{\{x>y_{u-}\}}$ holds $dx$-almost everywhere, which guarantees the convergence of $\nu(u,y^n,x)$ for almost every $u$.
\end{itemize}
So, in this particular case, we obtain continuity.
\end{proof}
In order to prove the convergence of the historical processes, we need to enhance the convergence of $Z^K$.
\begin{proposition}[Enhanced convergence in law for $(Z^K)_K$]\label{prop:enhanced convergence}
Let $f(x,y)=\mathds{1}_{\{x>y\}}\varphi(x)$, where $\varphi\in C_b(\mathbb{R}_+)$. Under the assumptions of Theorem \ref{theoreme cv} and if $\xi_0$ has a density, then we have 
\begin{equation}\label{eq:enhanced cv 1}
    (i)\qquad\lim_{K\rightarrow\infty}\mathbb{E}\left[\sup_{t\leq T}|\langle Z^K_t-\xi_t, f(\cdot,y)\rangle|\right]=0
\end{equation}
and 
\begin{equation}\label{eq:enhanced cv 2}
    (ii)\,\qquad\lim_{K\rightarrow\infty}\mathbb{E}\left[ \sup_{t\leq T,\, y\in D}|\langle Z^K_t-\xi_t,\mathds{1}_{\{\cdot>y_t\}}\rangle|\right]=0.
\end{equation}
\end{proposition}
\begin{proof}[Proof of Proposition \ref{prop:enhanced convergence}]

Let $l_n$ be the continuous function defined such that $l_n(x) = 0$ for $x \le y$, $l_n(x) = 1$ for $x \ge y + n^{-1}$, and $l_n$ is linear on the interval $[y, y + n^{-1}]$. Similarly, let $u_n$ be the continuous function such that $u_n(x) = 0$ for $x \le y - n^{-1}$, $u_n(x) = 1$ for $x \ge y$, and $u_n$ is linear on $[y - n^{-1}, y]$.

We define $\delta_n(x) = u_n(x) - l_n(x)$. Note that $\delta_n$ is continuous, bounded by $1$, and supported on the interval $(y - n^{-1}, y + n^{-1})$. Finally, let $f_n = l_n \varphi$.

We have $|f-f_n|\leq \delta_n\|\varphi\|_\infty$ so we can write 
\begin{equation*}
    \begin{aligned}
        |\langle Z^K_t-\xi_t, f\rangle|&\leq|\langle Z^K_t,f-f_n\rangle|+|\langle Z^K_t-\xi_t,f_n\rangle|
        +|\langle\xi_t, f_n-f\rangle|\\
        &\leq \|\varphi\|_\infty\langle Z^K_t,\delta_n\rangle+|\langle Z^K_t-\xi_t,f_n\rangle|+\|\varphi\|_\infty \langle\xi_t,\delta_n\rangle\\
        &\leq 2\|\varphi\|_\infty \langle\xi_t,\delta_n\rangle+|\langle Z^K_t-\xi_t,f_n\rangle|+\|\varphi\|_\infty |\langle Z^K_t-\xi_t,\delta_n \rangle|
    \end{aligned}
\end{equation*}
By taking the supremum over $t$ and the expectation, the convergence of the second and third terms of the last inequality follows from the continuity of $f_n$ and $\delta_n$, and from the convergence of $Z^K$ to $\xi$ in distribution. For the first term, we know that $\xi_t$ has a density and that $t\mapsto\langle\xi_t,\delta_n\rangle$ is continuous. For any time $t$ in the compact interval $[0,T]$, the sequence of functions $(\langle\xi_t,\delta_n\rangle)_n$ converges to $0$. By Dini’s Theorem, $\lim_{n\rightarrow\infty}\sup_{t\leq T}\langle\xi_t,\delta_n\rangle=0$. We then conclude that we have \eqref{eq:enhanced cv 1}.

For the second limit, we set $\varepsilon>0$. We know that $\xi$ is continuous in time on $[0,T]$, and we know that there exists a $B>0$ such that for all $t\in[0,T]$, $\xi_t([B,\infty))<\varepsilon$. On $[0,B]$, the density is bounded by $A:=\sup_x\xi_t(x)<\infty$. We consider a subdivision $0=u_0<u_1<...<u_{N-1}=B$ with steps of length less than $\varepsilon/A$. We set $u_N=\infty$. For all $i$, we have $\xi_t([u_i,u_{i+1}])\leq \varepsilon$. 

For any path $y \in D$ and any $t \in [0, T]$, there exists an $i$ such that $y_t \in [u_i, u_{i+1})$, so 
$\mathds{1}_{\{x>u_{i+1}\}} \leq \mathds{1}_{\{x>y_t\}} \leq \mathds{1}_{\{x>u_i\}}$ and we can write
\begin{equation*}
    |\langle Z^K_t-\xi_t,\mathds{1}_{\{\cdot>y_t\}}\rangle|\leq \max_{j=i,\,i+1}|\langle Z^K_t-\xi_t,\mathds{1}_{\{\cdot>u_j\}}\rangle|+\xi_t([u_i,u_{i+1}]).
\end{equation*}
Since the mass of $\xi$ is bounded and $y_t$ must fall within one of the intervals of the subdivision, we can take the supremum:
\begin{equation*}
    \sup_{t\leq T,\,y\in D}|\langle Z^K_t-\xi_t,\mathds{1}_{\{\cdot>y_t\}}\rangle|\leq \varepsilon+\sum_{j=0}^N \sup_{t\leq T}|\langle Z^K_t-\xi_t,\mathds{1}_{\{\cdot>u_j\}}\rangle|.
\end{equation*}
By taking the expectation, letting $K$ tend to infinity and using \eqref{eq:enhanced cv 1} for the terms in the sum, we obtain \eqref{eq:enhanced cv 2}.
\end{proof}

\begin{proof}[Proof of Proposition \ref{cv historique}]
Recall that $H^K$ and $\widehat{H}^K$ are built on the same probability space with the same initial condition. As in the proof of Proposition 3.4 of \cite{meleardB1}, we introduce 
\begin{equation*}
    \begin{aligned}
        \tau^K_M&=\inf\{t\geq 0\,|\, \langle H^K_t,1\rangle=\langle Z^K_t,1\rangle\geq M\}\\
        \widehat{\tau}^K_M&=\inf\{t\geq 0\,|\, \langle \widehat{H}^K_t,1\rangle=\langle \widehat{Z}^K_t,1\rangle\geq M\}.
    \end{aligned}
\end{equation*}
We fix $\varepsilon>0$. We know that the processes $(\langle Z^K_t,1\rangle)_t$ and $(\langle \widehat{Z}^K_t,1\rangle)_t$ converge to $(\langle\xi_t,1\rangle)_t$. It implies that there exists a $M=M(\varepsilon)>0$, not depending on $K$, such that
\begin{equation*}
    \mathbb{P}(\tau^K_M\wedge \widehat{\tau}^K_M\leq T)\leq \varepsilon.
\end{equation*}
Let us use a test-function of the form $\Phi_\varphi$ with $\|\Phi_\varphi\|_\infty\leq 1$ and the transform $\psi^t$ given by Lemma \ref{lemma : transformée FK}. We have 
\begin{equation*}
\begin{aligned}
    |\langle & H^K_{t\wedge\tau^K_M\wedge \widehat{\tau}^K_M}, \Phi_\varphi\rangle -\langle\widehat{H}^K_{t\wedge\tau^K_M\wedge \widehat{\tau}^K_M},\Phi_\varphi\rangle|
    =|\langle H^K_{t\wedge\tau^K_M\wedge \widehat{\tau}^K_M}, \psi^t(t,\cdot)\rangle-\langle\widehat{H}^K_{t\wedge\tau^K_M\wedge \widehat{\tau}^K_M},\psi^t(t,\cdot)\rangle|\\
    &\leq |R^K_{t\wedge\tau^K_M\wedge \widehat{\tau}^K_M,t}|+|\widehat{R}^K_{t\wedge\tau^K_M\wedge \widehat{\tau}^K_M,t}|+\int_0^{t\wedge\tau^K_M\wedge \widehat{\tau}^K_M}\int_{\mathbb{D}(\mathbb{R},\mathbb{R})}\psi^t(s,y)|\langle\xi_s- Z^K_s,1\rangle|H^K_s(dy)ds\\
    &+\left|\int_0^{t\wedge\tau^K_M\wedge \widehat{\tau}^K_M}\int_{\mathbb{D}(\mathbb{R},\mathbb{R})}\int_0^\infty (\psi^t(s,(y|s|x))-\psi^t(s,y))
        [h(x,y_s,Z^K_s)Z^K_s(dx)-h(x,y_s,\xi_s)\xi_s(dx)]\widehat{H}^K_s(dy)ds\right|\\
    &\leq \sup_{s\leq t}|R^K_{s\wedge\tau^K_M\wedge \widehat{\tau}^K_M,t}|+\sup_{s\leq t}|\widehat{R}^K_{s\wedge\tau^K_M\wedge \widehat{\tau}^K_M,t}|+CM\|\Phi_\varphi\|_\infty e^{\overline{b}t}T \sup_{s\leq T}|\langle Z^K_s-\xi_s,1\rangle|\\
    &+2M\|\Phi_\varphi\|_\infty\mathds{1}_{\mu>0}\sup_{s\leq T}(\mu\langle\xi_s,1\rangle)^{-1}T\sup_{s\leq T}|\langle Z^K_s-\xi_s,1\rangle|\\
    &+\sup_{s\leq T}\frac{\tau\|\Phi_\varphi\|_\infty e^{\overline{b}T}}{\beta+\mu\langle\xi_s,1\rangle}MT\sup_{s\leq T,\,y\in D}|\langle Z^K_s-\xi_s,\mathds{1}_{\{\cdot>y_s\}}\rangle|\\
    &+\int_0^t\tfrac{\tau}{\beta+\mu\langle\xi_s,1\rangle}M \sup_{y\in D}|\langle Z^K_s-\xi_s, \Psi(s,\cdot)\mathds{1}_{\{\cdot>y_s\}}\rangle|ds
\end{aligned}
\end{equation*}
where $\Psi(s,x)=\psi^t(s,(y|s|x))$. Indeed, we observe that $\psi^t(s,(y|s|x))$ does not depend on $y$, since the function evaluates the trajectory at time $s$, which is therefore fixed in $x$ and independent of $y$. Furthermore, the mapping $x\mapsto(y|s|x)$ is continuous with respect to the $J_1$ topology and $y\mapsto\psi^t(s,y)$ is continuous by Proposition \ref{prop:continuity criterion in J1}, so the mapping $\Psi(s,\cdot)$ is also continuous.

We take the supremum over time and the expectation and the second moments for the terms $R^K$ and $\widehat{R}^K$, and obtain by Doob's inequality:
\begin{equation*}
\begin{aligned}
    \mathbb{E}[\sup_{s\leq t}\langle H^K_{s\wedge\tau^K_M\wedge \widehat{\tau}^K_M}-\widehat{H}^K_{s\wedge\tau^K_M\wedge \widehat{\tau}^K_M},\Phi_\varphi\rangle]&\leq \sqrt{\mathbb{E}[\langle R^K_{t,t}\rangle]}+\sqrt{\mathbb{E}[\langle \widehat{R}^K_{t,t}\rangle]}\\
    &+M\|\Phi_\varphi\|_\infty T (Ce^{\overline{b}T}+2\mathds{1}_{\mu>}\sup_{s\leq T}(\mu\langle\xi_s,1\rangle)^{-1})\mathbb{E}[\sup_{s\leq T}|\langle Z^K_s-\xi_s,1\rangle|]\\
    &+\sup_{s\leq T}\frac{\tau\|\Phi_\varphi\|_\infty e^{\overline{b}T}}{\beta+\mu\langle\xi_s,1\rangle}MT\mathbb{E}[\sup_{s\leq T,\,y\in D}|\langle Z^K_s-\xi_s,\mathds{1}_{\{\cdot>y_s\}}\rangle|]\\
    &+\int_0^t\tfrac{\tau}{\beta+\mu\langle\xi_s,1\rangle}M\mathbb{E}[\sup_{y\in D}|\langle Z^K_s-\xi_s, \Psi(s,\cdot)\mathds{1}_{\{\cdot>y_s\}}\rangle|]ds.
\end{aligned}
\end{equation*}
Let $Q_M(K)$ denote the right-hand side of the inequality. 

By partitioning over the events $\{\tau^K_M \wedge \widehat{\tau}^K_M \leq T\}$ and $\{\tau^K_M \wedge \widehat{\tau}^K_M > T\}$: 
\begin{equation*}
    \begin{aligned}
        \mathbb{E}[\sup_{t\leq T}|\langle H^K_t-\widehat{H}^K_t, \Phi_\varphi\rangle|]&\leq\mathbb{E}[\sup_{s\leq t}\langle H^K_{s\wedge\tau^K_M\wedge \widehat{\tau}^K_M}-\widehat{H}^K_{s\wedge\tau^K_M\wedge \widehat{\tau}^K_M},\Phi_\varphi\rangle]\\
        &+\sqrt{\mathbb{P}(\tau^K_M\wedge \widehat{\tau}^K_M\leq T)}\sqrt{2\sup_K\mathbb{E}[\sup_{t\leq T}\langle Z^K_t,1\rangle^2+\sup_{t\leq T}\langle \widehat{Z}^K_t,1\rangle^2]}\\
        &\leq Q_M(K)+\sqrt{\varepsilon}\sqrt{2\sup_K\mathbb{E}[\sup_{t\leq T}\langle Z^K_t,1\rangle^2+\sup_{t\leq T}\langle \widehat{Z}^K_t,1\rangle^2]}.
    \end{aligned}
\end{equation*}
The last inequality is obtained by taking $M>M(\varepsilon)$. By our assumptions, $Q_M(K)\underset{K\rightarrow\infty}{\rightarrow}0$ thus concluding the proof. To extend the convergence to $L^2$, we make use of the estimates given by Lemma \ref{lemma_moment_control}.
\end{proof} 
\section{Time reversal and Nagasawa's Theorem}\label{sec:appendix time reversal}
We recall some definitions.
\begin{definition}
Let \(P\) be the semigroup (extended to time-space) of a Markov process. The potential $U$ is defined by 
\begin{equation*}
    U = \int_0^\infty  P_t \, dt.
\end{equation*}
A measure $\eta$ is said to be excessive for $P$ if it is positive, $\sigma$-finite and for any positive, bounded, and measurable function $f$ we have 
\begin{equation*}
    \eta(P_t f)\leq \eta(f) \,\text{ and }\lim_{t\rightarrow 0^+} \eta(P_tf)=\eta(f).
\end{equation*}
\end{definition}
The following result can be found in Theorem 3.5 of \cite{nagasawa64} or 2.2.1 of \cite{nagasawa}:

\begin{theorem}[Nagasawa]\label{Dellacherie}
Let \(P\) be the semigroup (extended to time-space) of a Markov process of potential $U$. Let \(\mu\) be an initial law, and define \(\eta = \mu U\) which is excessive for $P$.
 The semigroup \(S\) of the time-reversed process satisfies the duality relation $\forall \varphi,\psi\in L^\infty([0,T]\times\mathbb{R}_+)$:
\begin{equation*}
    \int_{[0,T]\times\mathbb{R}_+}P_t\varphi(s,x)\psi(s,x)\eta(ds,dx)=\int_{[0,T]\times\mathbb{R}_+}\varphi(s,x) S_t\psi(s,x)\eta(ds,dx)
\end{equation*}

\end{theorem}

We make the measure \(\eta\) explicit:

\begin{lemma}\label{mesure excessive}
For the initial measure on $\mathbb{R}\times\mathbb{R}_+$
\begin{equation*}
   \mu(ds,dx) = \xi(s,x) m_T(s,x) (\delta_0(ds) \otimes dx) 
\end{equation*}
and for the potential \(U\) associated with \(\widetilde{P}\), we have, for any measurable set \(A \times B \subset \mathbb{R} \times \mathbb{R}_+\):
\begin{equation*}\label{expression eta}
    \eta(A\times B)=\mu U(A \times B) = \int_{A \times B} \mathds{1}_{\{0 \leq s \leq T\}} \, \xi(s,x) m_T(s,x) \, ds \, dx.
\end{equation*}
\end{lemma}

\begin{proof}
We compute $\mu U$:
\begin{align*}
    \mu U(A \times B)
    &= \int_{\mathbb{R}} \int_0^\infty  \int_0^\infty  \widetilde{P}_u(\mathds{1}_{A \times B})(s,x) \, \xi(s,x) m_T(s,x) \, du \, dx \, \delta_0(ds)\\
    &= \int_0^\infty  \int_0^\infty  \widetilde{P}_u(\mathds{1}_{A \times B})(0,x) \, \xi(0,x) m_T(0,x) \, du \, dx.
\end{align*}
We apply the definition of $\widetilde{P}$ expressed in terms of $\widehat{P}$. Next, we use Lemma \ref{lemme de transition}:
\begin{align*}
    \mu U(A \times B)
    &= \int_0^\infty  \langle \delta_0 \otimes \xi_0, \widehat{P}_u(\mathds{1}_{A \times B} m_T) \rangle \mathds{1}_{\{0 \leq u \leq T\}} \, du\\
    &= \int_0^\infty  \langle \delta_u \otimes \xi_u, \mathds{1}_{A \times B} m_T \rangle \mathds{1}_{\{0 \leq u \leq T\}} \, du\\
    &= \int_0^\infty  \int_0^\infty  \mathds{1}_{A \times B}(u,x) \mathds{1}_{\{0 \leq u \leq T\}} \, \xi(u,x) m_T(u,x) \, du \, dx\\
    &= \int_{A \times B} \mathds{1}_{\{0 \leq u \leq T\}} \, \xi(u,x) m_T(u,x) \, du \, dx.
\end{align*}
\end{proof}
We prove that the measure $\eta$ that we computed is excessive.
\begin{lemma}
The measure \(\eta\) is excessive with respect to the semigroup \(\widetilde{P}\).
\end{lemma}
\begin{proof}
The measure is positive. From Proposition \ref{bornitude de la taille de population}, we have \(\sup_{s \leq T} \langle \xi_s, 1 \rangle < \infty\), and by the assumptions on the birth rate, \(r^\lambda(s,x) \leq \overline{b}\). It follows that \(m_T(s,x) \leq e^{\overline{b}T}\). Therefore,
\[
\eta(\mathbb{R}\times\mathbb{R}_+) = \int_0^T \int_0^\infty  \xi(s,x) m_T(s,x) \, ds \, dx \leq e^{\overline{b}T} T \sup_{s \leq T} \langle \xi_s, 1 \rangle < \infty,
\]
so the measure is finite and hence \(\sigma\)-finite.

Let us prove the inequality $\eta(\widetilde P_t f)\leq \eta(f)$. By the definition of $\widetilde{P}$, we have:
\begin{equation*}
    \begin{aligned}
    \eta(\widetilde{P}_t f)&=\int_\mathbb{R}\int_0^\infty  \widetilde{P}_tf(s,x)\xi(s,x)m_T(s,x)\mathds{1}_{\{0\leq s \leq T\}}dxds\\
    &=\int_\mathbb{R}\int_0^\infty \widehat{P}_t(m_T f)(s,x)\xi(s,x)\mathds{1}_{\{0\leq s \leq T\}}\mathds{1}_{\{0\leq s +t\leq T\}}dxds\\
    &=\int_\mathbb{R}\langle \delta_s\otimes\xi_s, m_T f\rangle \mathds{1}_{\{0\leq s \leq T\}}\mathds{1}_{\{0\leq s +t\leq T\}}ds
    \end{aligned}
\end{equation*}
Using Lemma \ref{lemme de transition}, we get
\begin{equation*}
    \begin{aligned}
    \eta(\widetilde{P}_tf)&=\int_\mathbb{R}\langle \delta_{s+t}\otimes\xi_{s+t},m_T f\rangle \mathds{1}_{\{0\leq s \leq T\}}\mathds{1}_{\{0\leq s +t\leq T\}}ds\\
    &=\int_t^T\int_0^\infty \xi(s,x)m_T(s,x)f(s,x)dxds\\
    &\leq \int_0^T\int_0^\infty \xi(s,x)m_T(s,x)f(s,x)dxds = \eta(f)
    \end{aligned}
\end{equation*}
Moreover, applying the dominated convergence Theorem to the integral term as $t\rightarrow 0^+$ yields $\lim_{t\rightarrow 0^+}\eta(\widetilde{P}_tf)=\eta(f)$. We conclude that $\eta$ is indeed excessive.
\end{proof}
\section{Approximation and Extension Results}\label{sec:appendix approximation and extension results}
\begin{lemma}\label{lemma : approximation de l'indicatrice}
For any finite interval $A$ on $\mathbb{R}$, there exists a sequence $(c_n)$ of elements of $C^1_b(\mathbb{R},\mathbb{R})$ that converges pointwise to $\mathds{1}_A$.
\end{lemma}
\begin{proof}
Let $S(x)=\frac{1}{1+e^{-x}}$ and $A=[a,b]$. We define $I_n(s)=S\left(n^2\left(s-a+\frac{1}{n}\right)\right)-S\left(n^2\left(s-b-\frac{1}{n}\right)\right)$. For all $n$, $I_n$ belongs to $C^1_b$ and when $n$ goes to infinity, it converges pointwise to $\mathds{1}_{[a,b]}$. A similar construction yields the approximation of $[a,b),\, (a,b]$ and $(a,b)$.
\end{proof}
Recall the following definitions:
\begin{itemize}
    \item $\mathcal{A}_b = \{ (s,x) \mapsto \mathds{1}_A(s) f(s,x) \mid A \subset \mathbb{R} \text{ finite interval}, f \in C^1_b(\mathbb{R},L^\infty(\mathbb{R}_+,\mathbb{R}))\}$.
    \item $\mathcal{A} = \{ (s,x) \mapsto \mathds{1}_A(s) f(s,x) \mid A \subset \mathbb{R} \text{ finite interval}, f \in C^1(\mathbb{R},L^\infty(\mathbb{R}_+,\mathbb{R}))\}$.
    \item $\mathcal{A}'_b = \{ (s,x) \mapsto \mathds{1}_A(s) f(s,x) \mid A \subset \mathbb{R} \text{ finite interval}, f \in C^1_b(\mathbb{R},L^1(\mathbb{R}_+,\mathbb{R}))\}$.
    \item $\mathcal{A}' = \{ (s,x) \mapsto \mathds{1}_A(s) f(s,x) \mid A \subset \mathbb{R} \text{ finite interval}, f \in C^1(\mathbb{R},L^1(\mathbb{R}_+,\mathbb{R}))\}$.
\end{itemize}
\begin{lemma}\label{lemma : extension}
We have
\begin{equation*}
    \mathcal{A}=\mathcal{A}_b\,\text{and } \mathcal{A}'=\mathcal{A}'_b
\end{equation*}
And in particular, we have
\begin{equation*}
    \{\mathds{1}_A(s)f(s,x)| A\subset\mathbb{R}\text{ compact interval }, f \in C^1(A,L^\infty(\mathbb{R}_+,\mathbb{R}))\}\subset\mathcal{A}
\end{equation*}
and 
\begin{equation*}
    \{\mathds{1}_A(s)g(s,x)| A\subset\mathbb{R}\text{ compact interval }, g \in C^1(A,L^1(\mathbb{R}_+,\mathbb{R}))\}\subset\mathcal{A}'
\end{equation*}
\end{lemma}
\begin{proof}
If $A$ is a finite interval, then its closure $\overline{A}$ is compact. Any function $f\in C^1(\mathbb{R},L^\infty(\mathbb{R}_+,\mathbb{R}))$ restricted to $\overline{A}$ is bounded and its derivative is also bounded.

We can use a truncation function $\chi\in C^\infty_c(\mathbb{R},L^\infty(\mathbb{R}_+,\mathbb{R}))$ such that $\chi(s)$ is equal to $1$ if $s\in\overline{A}$ and $\chi(s)=0$ outside a compact neighbourhood of $\overline{A}$. 

Define $\widetilde{f}(s,x)=\chi(s)f(s,x)$ then $\widetilde{f}\in C^1_b(\mathbb{R},L^\infty(\mathbb{R}_+,\mathbb{R}))$. If $s\in A$, $\widetilde{f}(s,x)=f(s,x)$ and $\mathds{1}_A(s)f(s,x)=\mathds{1}_A(s)\widetilde{f}(s,x)$. We deduce that $\mathds{1}_Af$ belongs to $\mathcal{A}_b$. The same goes for $\mathcal{A}'$ and $\mathcal{A}'_b$.

An example of $\chi$ can be constructed in the following way:

Take for example $A=[a,b]$. Set $l(t)=e^{-\tfrac{1}{t}}\mathds{1}_{t>0}$ and $g(t)=\frac{l(t)}{l(t)+l(1-t)}$. We have $g(t)=0$ for all $t\leq0$ and $g(t)=1$ for all $t\geq1$. Let $\varepsilon>0$, we define
\begin{equation*}
    \chi(t)=g\left( \frac{t-a+\varepsilon}{\varepsilon} \right)g\left(\frac{b+\varepsilon-t}{\varepsilon}\right)
\end{equation*}
We have that $\chi\in C^\infty_c(\mathbb{R},[0,1])$, with compact support in $[a-\varepsilon,b+\varepsilon]$ and $\chi(s)=1$ on $A$.

Let $A=[a,b]$ be a compact interval of $\mathbb{R}$ and $f\in C^1(A,L^\infty(\mathbb{R}_+,\mathbb{R}))$. We define
\begin{equation*}
    f_{ext}(t,x)=[f(a,x)+\partial_1f(a,x)(t-a)]\mathds{1}_{t<a}+f(t,x)\mathds{1}_{a\leq t \leq b}+[f(b,x)+\partial_1f(b,x)(t-b)]\mathds{1}_{t>b}
\end{equation*}
Then $f_{ext}\in C^1(\mathbb{R},L^\infty(\mathbb{R}_+,\mathbb{R}))$.
We have $\mathds{1}_A(s) f(s,x)=\mathds{1}_A(s) f_{ext}(s,x)$ and we conclude that $\mathds{1}_A f\in\mathcal{A}$. The same goes for $\mathcal{A}'$.
\end{proof}

\end{appendix}
\end{document}